\documentclass[11pt,twoside]{amsart}

\usepackage{amssymb}
\usepackage[left=2.5cm,right=2.5cm,top=2.5cm,bottom=2.5cm]{geometry}
\usepackage{graphicx}
\usepackage{multirow}
\usepackage{hyperref}
\usepackage{bbm}
\usepackage{amsmath,amssymb,amsfonts, amsthm,epsfig}
\usepackage{verbatim,color}
\usepackage{enumerate}
\usepackage{mathtools}

\usepackage{graphics}
\usepackage{latexsym}
\usepackage{psfrag}
\usepackage{relsize}
\usepackage{tikz}
\usepackage{esint}
\usepackage{float}
\usepackage{pgfplots}

\usetikzlibrary{positioning}
\usetikzlibrary{arrows}
\usetikzlibrary{decorations.pathreplacing}

\newcommand{\loza}
{--++ (1/2,{sqrt(3)/2})
--++ (-1/2,{sqrt(3)/2})
--++ (-1/2,{-sqrt(3)/2})
--++ (1/2,{-sqrt(3)/2})
[fill=yellow!100!white!]}

\newcommand{\lozb}
{--++ (1,0)
--++ (-1/2,{sqrt(3)/2})
--++ (-1,0)
--++ (1/2,{-sqrt(3)/2})
[fill=blue!80!white!]}

\newcommand{\lozc}
{--++ (1,0)
--++ (1/2,{sqrt(3)/2})
--++ (-1,0)
--++ (-1/2,{-sqrt(3)/2})
[fill=red!60!white!]}

\newcommand{\lozyt}
{--++ (0,1)
--++ (-1,1)
--++ (0,-1)
--++ (1,-1)
[fill=yellow!100!white!]}

\newcommand{\lozbt}
{--++ (1,0)
--++ (-1,1)
--++ (-1,0)
--++ (1,-1)
[fill=blue!80!white!]}

\newcommand{\lozrt}
{--++ (1,0)
--++ (0,1)
--++ (-1,0)
--++ (0,-1)
[fill=red!60!white!]}

\numberwithin{equation}{section}

\numberwithin{equation}{section}

\newtheorem{thma}{Theorem}[section]
\newtheorem{lemma}[thma]{Lemma}

\newtheorem{defi}[thma]{Definition}
\newtheorem{prop}[thma]{Proposition}

\newtheorem{remark}[thma]{Remark}
\newtheorem{condition}{Condition}

\renewcommand{\Re}{\mathrm{Re }}
\renewcommand{\Im}{\mathrm{Im }}
\newcommand{\eps}{\varepsilon}

\newcommand{\Ai}{\text{Ai\,}}

\newcommand{\Tr}{\text{tr\,}}

\newcommand{\sgn}{\text{sgn\,}}

\begin{document}

\title{Edge fluctuations of limit shapes}

\author{Kurt Johansson} 

\address{
Department of Mathematics,
KTH Royal Institute of Technology,
SE-100 44 Stockholm, Sweden}

\thanks{Supported by the Knut and
  Alice Wallenberg Foundation grant KAW:2010.0063 and by the Swedish Science Research Council (VR)}

\maketitle
\begin{abstract}
In random tiling and dimer models we can get various limit shapes which gives the boundaries between different types of phases. The shape fluctuations at
these boundaries give rise to universal limit laws, in particular the Airy process. We survey some models which can be analyzed in detail based on the fact that
they are determinantal point processes with correlation kernels that can be computed. We also discuss which type of limit laws that can be obtained.

\end{abstract}
\tableofcontents

\section{Introduction} \label{sec:1}

This paper surveys various results on edge fluctuations of limit shapes. By this we mean that we have a microscopic stochastic model that in some macroscopic limit has a
non-random limit shape with some interfaces or boundaries between different regions. Such a  boundary is what we will call an \emph{edge}. At an intermediate scale the corresponding random interface has fluctuations, and we can obtain stochastic processes or limit laws by considering appropriate scaling limits. Whereas the macroscopic limit shape is non-universal, i.e. depends on the details
of the model, the edge fluctuations are often given by universal limit laws that occur in many different contexts. Microscopic limits, like Gibbs measures, are also less universal
since they also depend more on the details of the model. The simplest example of this universality phenomenon is a random walk. Let $X_1, X_2,\dots$ be a sequence of independent identically distributed random variables with finite variance, and let $S_N=X_1+\dots+X_N$ be the sum of the first $N$ of them. We can think of $S_N$ as the
position of an interface. Then, we have the law of large numbers which says that  $S_N/N\to c=\mathbb{E}[X_1]$ almost surely as $N\to \infty$. This is the macroscopic limit, the limit shape,
and it is non-universal since it depends on $\mathbb{E}[X_1]$ and hence on the exact distribution of $X_1$. The intermediate level scaling limit is the central limit theorem, which
says that $(S_N-Nc)/\sigma \sqrt{N}\to N(0,1)$ in distribution as $N\to\infty$, where $\sigma^2$ is the variance of $X_1$. Here, the limit is always a standard Gaussian irrespective of the exact distribution of $X_1$. This is the classical example of a universal scaling limit.

In this paper we will consider edge fluctuations in two-dimensional random tiling models or dimer models on bipartite graphs which have interesting edge scaling limits.
The same type of scaling limits also occur in random matrix theory, in local random growth models, directed random polymers and in interacting particle systems. We will not discuss these other very interesting models but only give references to some survey papers, see section \ref{rem:Schur}.
Let us give some examples in the form of pictures of the type of random tiling models we will be looking at.

\begin{center}
\begin{figure}
\includegraphics[height=3in]{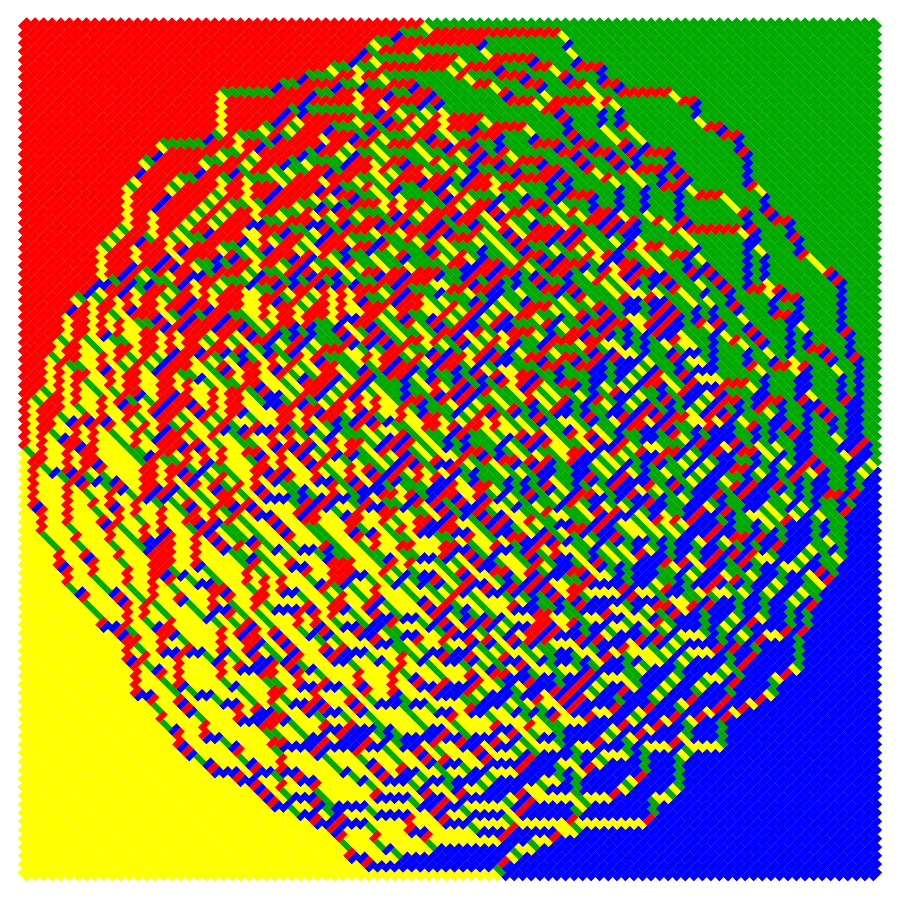}
\hspace{2mm}
\includegraphics[height=3in]{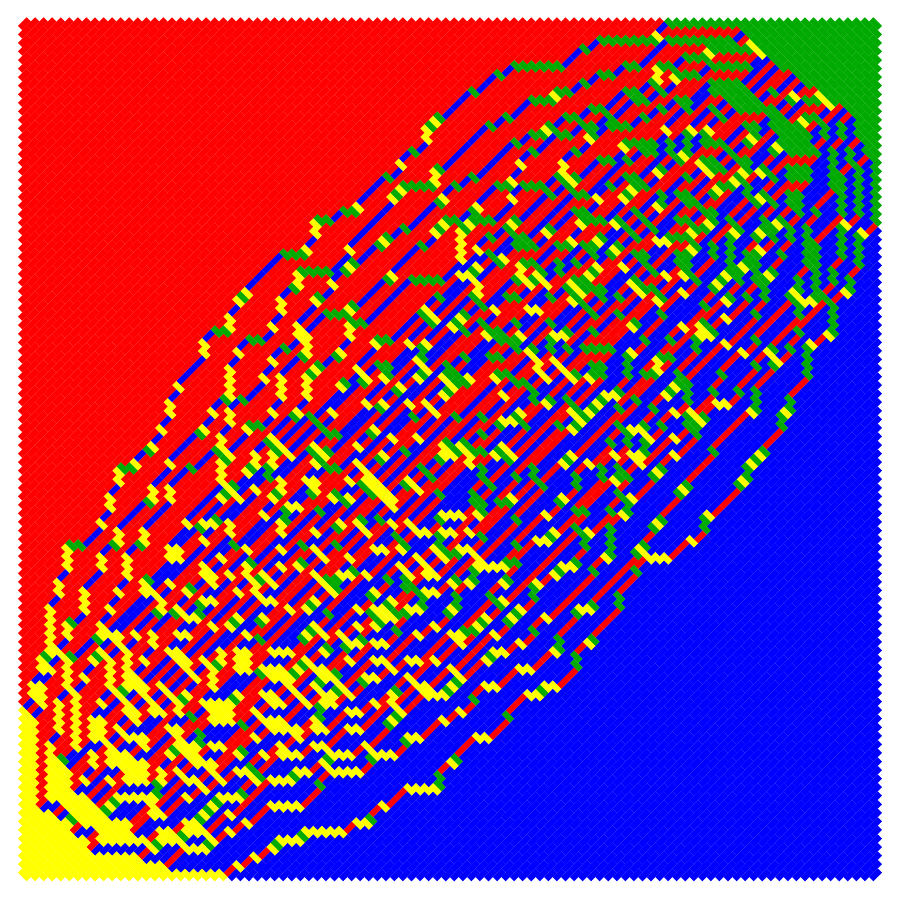}
\caption{The left figure is a uniform sample of a tiling of an Aztec diamond. The right figure is a sampling where vertical dominos (yellow and green) have weight $a=1/2$ and horizontal dominos (blue and red) have weight =1. The model and the four colours are explained in section \ref{sec:Aztec}.}
\label{fig:aztecsim}
\end{figure} 
\end{center}

In figure \ref{fig:aztecsim} we see that a random domino tiling of a shape called the Aztec diamond splits into regions with different tiling structure. There is a disordered tiling region in the center, called the {\it liquid region},
surrounded by a non-random or frozen region, called a {\it solid region}. The interface or edge, is the boundary between these two regions, and it is this boundary
that is our main interest. A similar phenomenon, but with three different types of regions and two boundary curves, is seen in figure \ref{fig:twoperiodicsim}. In this case, around the
center of the picture, we encounter a third type of tiling structure and we call it a {\it gas region}.

\begin{figure}
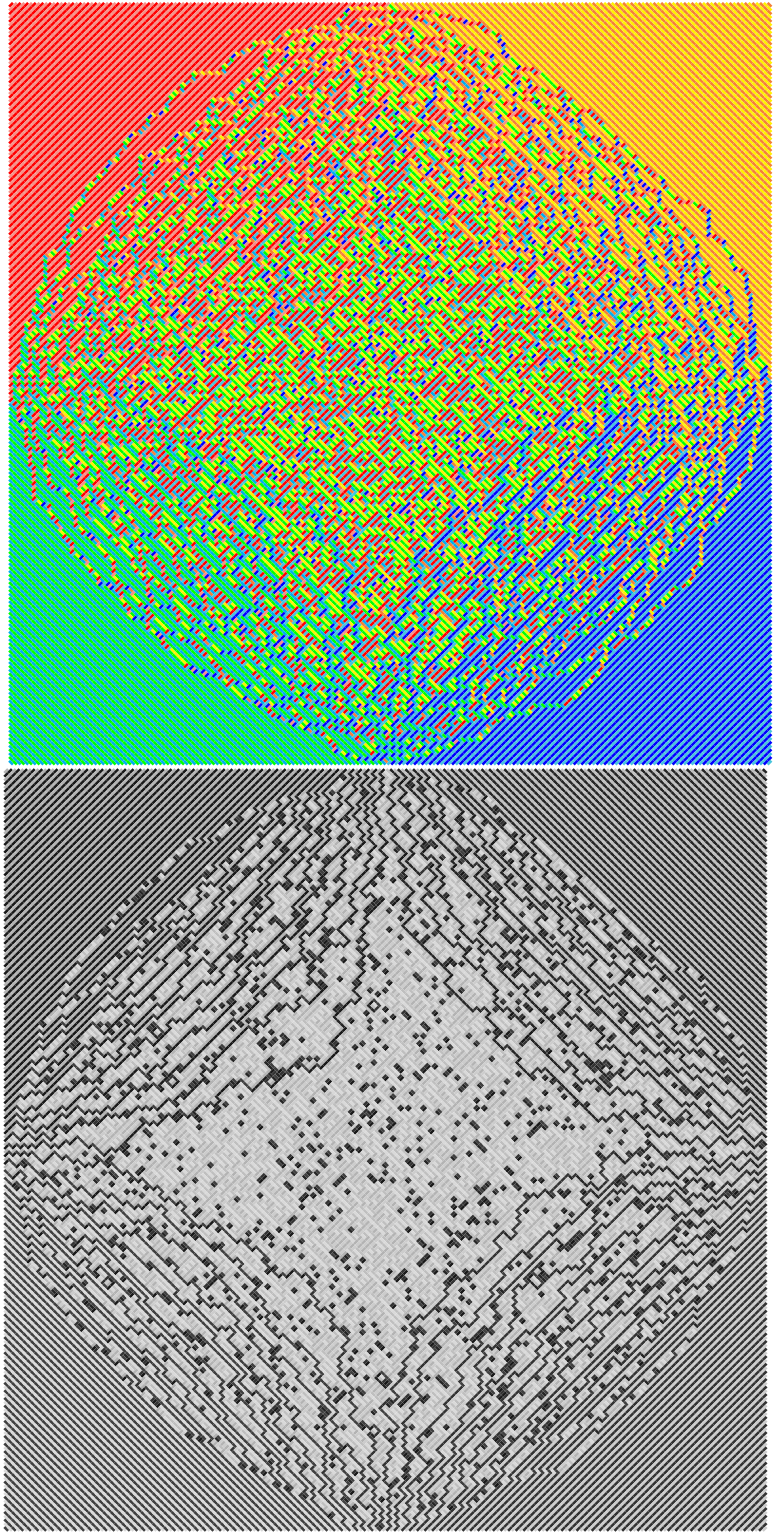

\begin{center}
\includegraphics[height=4in]{ttpn200a05.pdf}
\includegraphics[height=4in]{ttpn200a05bw.pdf}
\caption{Two different drawings of a simulation of a domino tiling of the two-periodic Aztec diamond of size 200 with $a=0.5$.  The top figure contains eight different colors, highlighting the solid and liquid phases. The bottom figure contains eight different gray-scale colors to accentuate the gas phase. }
\label{fig:twoperiodicsim}
\end{center}
\end{figure}

In the limit, see figure \ref{fig:arcticcurves}, the boundary converges to a non-random limit curve, the {\it arctic curve}. Before the limit we have fluctuations around this limit curve which we can study by taking appropriate scaling limits. Further examples of random tilings can be seen in figures \ref{fig:hexagonsim} and \ref{fig:tpn200}, and we can ask if we get the same
scaling limits in all the different models? The fact that we get these boundaries between different types of regions depends sensitively on the shape of the boundary. A random domino tiling of a rectangle will look disordered everywhere and we can see no boundary between different types of regions.

\begin{figure}
\begin{center}
\includegraphics[height=2in]{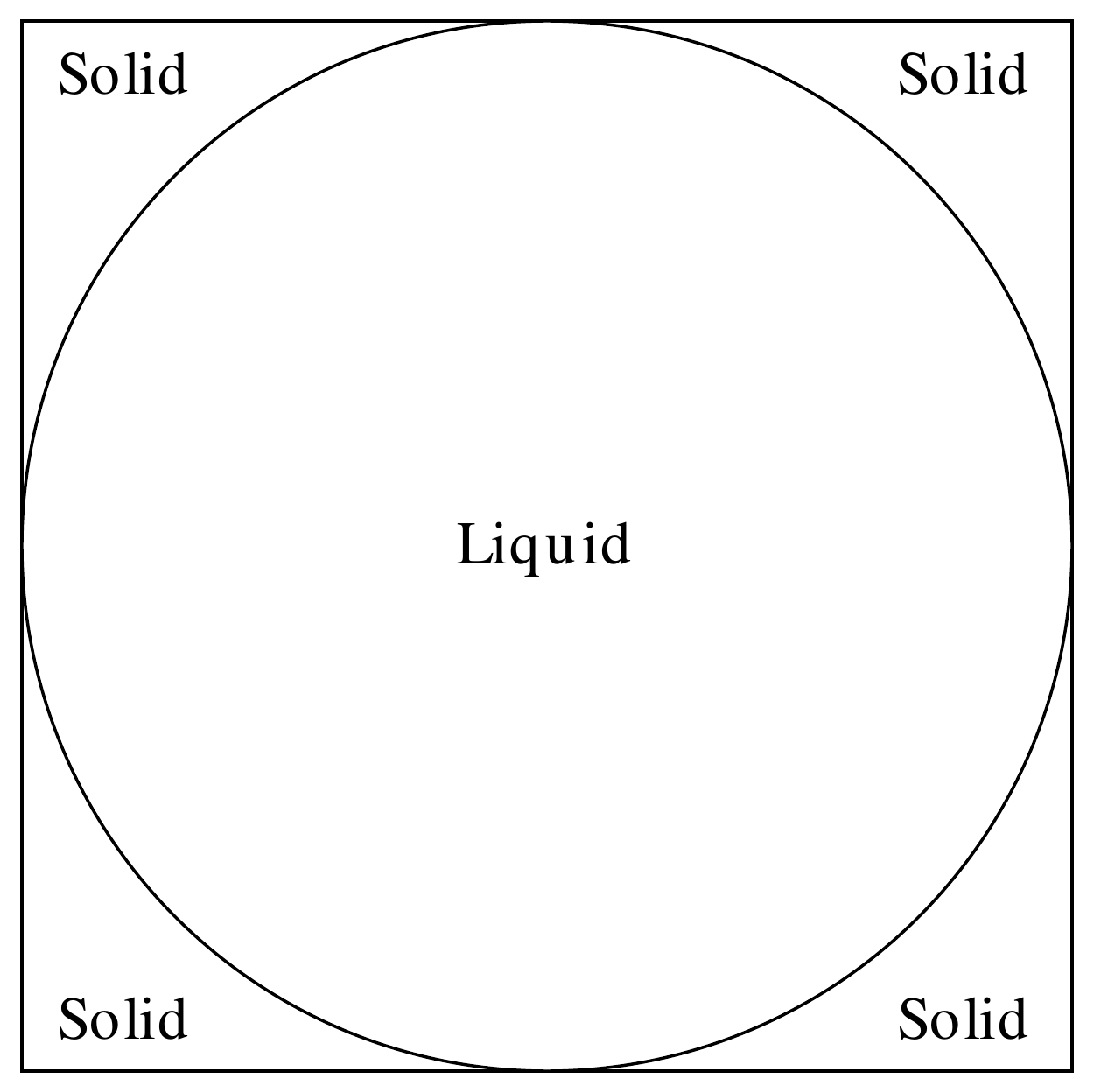}
\hspace{5mm}
\includegraphics[height=2in]{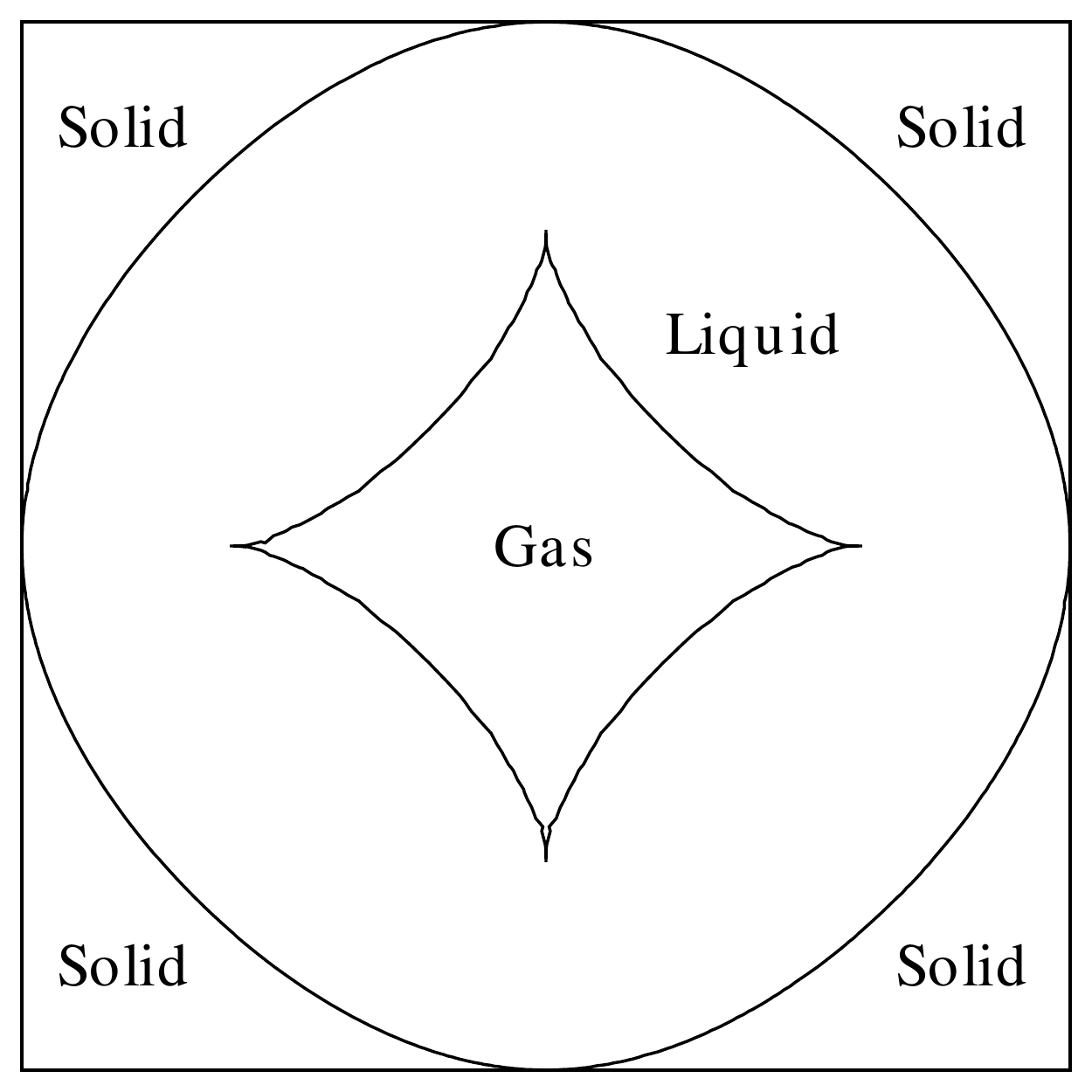}
\caption{The arctic curves or limit shapes of the boundaries between different types of tiling regions. To the left for the uniform tiling of an Aztec diamond, and to the right for the
two-periodic Aztec diamond.}
\label{fig:arcticcurves}
\end{center}
\end{figure}

\begin{figure}
\begin{center}
\includegraphics[height=3.5in]{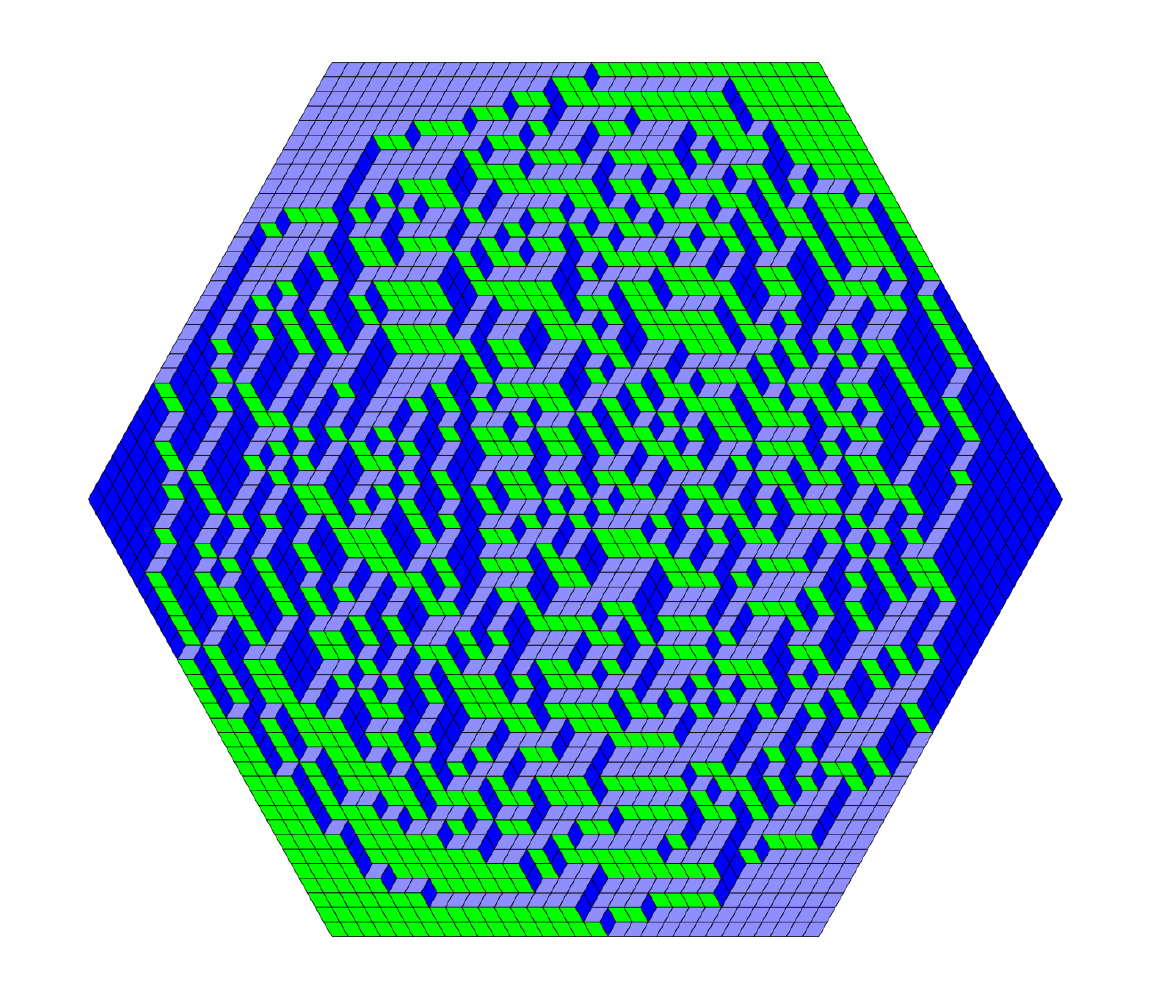}
\caption{A uniform random rhombus or lozenge tiling of a regular hexagon. (Picture by L. Petrov)}
\label{fig:hexagonsim}
\end{center}
\end{figure}

\begin{figure}
\begin{center}
\includegraphics[height=3.5in]{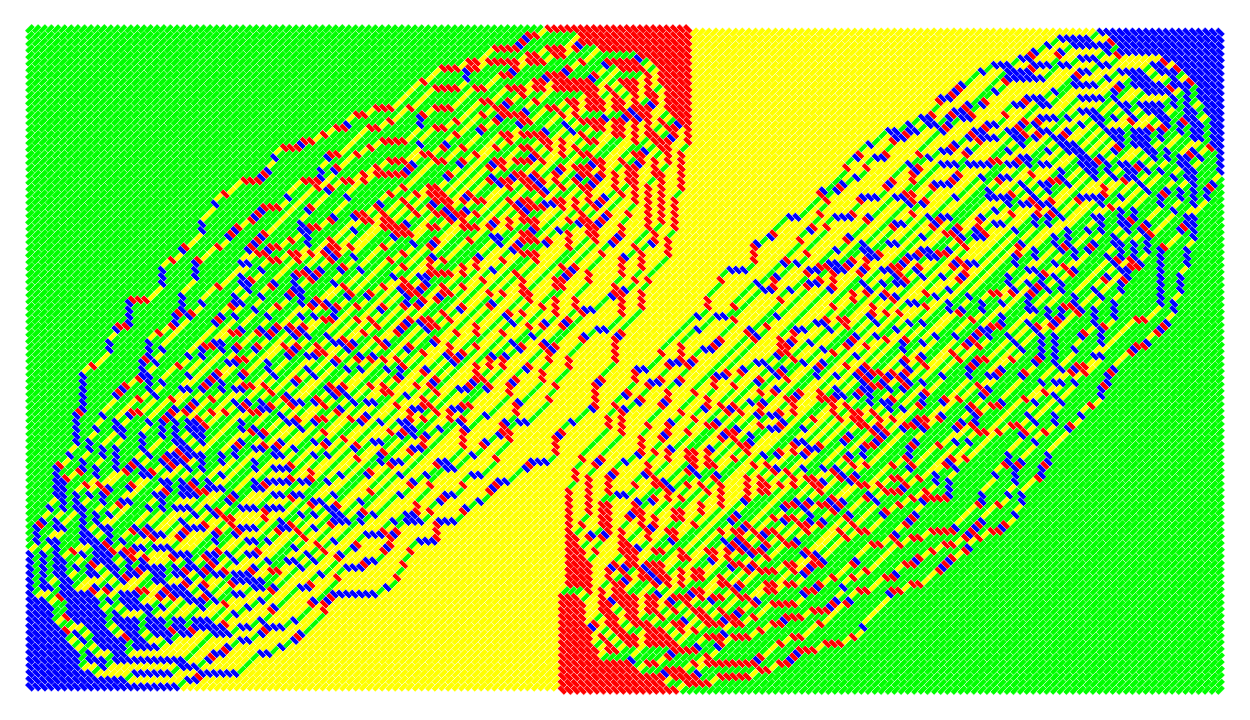}
\caption{A simulation of a random tiling of the double Aztec diamond.}
\label{fig:tpn200}
\end{center}
\end{figure}

The main questions we will address are

\begin{enumerate}  
\item How can we find useful formulas to evaluate the edge scaling limits at least in some models as those seen in the pictures?
\item Which natural scaling limits  can we get? Can we compute them and understand their properties?

\end{enumerate}

A third, and very important question, is that about universality. Can we prove that the same universal scaling limits occur in broader classes of models? We will
not say much about this question.

The basic structure that makes it possible to answer the first question are determinantal point processes. In sections 2-5 we will develop much of the machinery
that is needed.  The second question will be discussed in section \ref{sec:7} where we will go through many of the relevant scaling limits, and indicate how they can be obtained in various models using the techniques from sections 2-5. In section \ref{sec:8} we will discuss a more complicated model, the two-periodic
Aztec diamond, see figure \ref{fig:twoperiodicsim}, which is still determinantal but that cannot be handled by the methods of sections 2-5. Also, in this model, as can be seen in the figure, the inner boundary between the liquid and gas phases does not have a clear local microscopic definition. 

\subsubsection{Some notation}\label{Notation} Throughout the paper $\gamma_r$ will be the positively oriented circle of radius $r$ around the origin in the complex plane.
The indicator function will be denoted by $\mathbb{I}$ and $\mathrm{i}$ always denotes the imaginary unit. If $x=(x_1,\dots,x_n)$, the size $n$ Vandermonde determinant is denoted by
\begin{equation*}
\Delta_n(x)=\det(x_i^{j-1})_{1\le i,j\le n}.
\end{equation*}
Some basic facts about partitions and Schur polynomials are summarized in the Appendix.

\subsection*{Acknowledgments} I thank Erik Duse  for several pictures. Many thanks to Sunil Chhita for many pictures and helpful comments.

\section{Determinantal point processes} \label{sec:2}

This section gives an overview of determinantal point processes, and discusses an important way to obtain such processes via probability measures defined by products
of determinants.

\subsection{Definition and basic properties} \label{sec:2.1}

We will be rather brief in our presentation of determinantal point processes and refer to \cite{BorDPP}, \cite{JoHouch}, or \cite{SoshDPP} for more details. Let $\mathfrak{X}$ be a
complete, separable metric space. By $\mathcal{N}(\mathfrak{X})$ we denote the space of all boundedly finite counting measures $\xi$ on $\mathfrak{X}$, i.e. $\xi$ is a Borel
measure on $\mathfrak{X}$ such that $\xi(B)\in\{0,1,2,\dots\}$ for all Borel sets $B$ in $\mathfrak{X}$, and $\xi(B)<\infty$ for all bounded Borel sets $B$. We say that
$\xi$ is {\it simple} if $\xi(\{x\})\le 1$ for all $x\in\mathfrak{X}$. Typical examples of $\mathfrak{X}$ are $\mathbb{R}^d$, $\mathbb{Z}^d$, finite sets, and subsets or products of
such spaces. We assume that $\mathfrak{X}$ has a reference measure $\lambda$; for example in $\mathbb{R}$, we can take Lebesge measure and in a discrete space the
standard counting measure. Define a $\sigma$-algebra on $\mathcal{N}(\mathfrak{X})$ by taking the smallest $\sigma$-algebra for which the map $A\to\xi(A)$ is measurable
for all $\xi\in\mathcal{N}(\mathfrak{X})$ and all Borel sets $A$ in $\mathfrak{X}$.

A {\it point process} on $\mathfrak{X}$ is a probability measure $\mathbb{P}$ on $\mathcal{N}(\mathfrak{X})$. Let $\mathbb{E}$ denote the corresponding expectation. The point
process $\xi$ is simple if $\mathbb{P}(\xi\text{\,simple})=1$. If $\xi\in\mathcal{N}(\mathfrak{X})$ and $B$ is bounded, we can write
\begin{equation*}
\xi|_B=\sum_{i=1}^{\xi(B)}\delta_{x_i},
\end{equation*}
for some $x_i\in\mathfrak{X}$, $1\le i\le \xi(B)$, which can coincide if we have points of multiplicity greater than one. We often think of the $x_i$'s as the position of particles and simply call $x_i$ a particle, and talk about a particle process. Notice that there is no ordering of the points or particles.
If $\phi:\mathfrak{X}\to\mathbb{C}$ is a bounded
function with support in a bounded set $B$, we write
\begin{equation*}
\prod_{i}(1-\phi(x_i))=\prod_{i=1}^{\xi(B)}(1-\phi(x_i)).
\end{equation*}
(The left side is $=1$ if $\xi(B)=0$.)

\begin{defi}\label{def:corrfcns}
Let $\xi$ be a point process on $\mathfrak{X}$ and let $\rho_n:\mathfrak{X}^n\to\mathbb{C}$, $n\ge 1$, be a sequence of measurable functions. We say that $\xi$ has
{\it correlation functions} $\rho_n$, if
\begin{equation}\label{corrfcnsdef}
\mathbb{E}\left[\prod_{i}(1-\phi(x_i))\right]=\sum_{n=0}^\infty\frac{(-1)^n}{n!}\int_{\mathfrak{X}^n}\prod_{j=1}^n\phi(x_j)\rho_n(x_1,\dots,x_n)\,d^n\lambda(x),
\end{equation}
for all bounded functions $\phi$ on $\mathfrak{X}$ with bounded support, and the sum in the right side is convergent. The term $n=0$ in the right side of (\ref{corrfcnsdef})
is $=1$.
\end{defi}

It is not hard to see that if $\mathfrak{X}$ is a discrete space and $\lambda$ is counting measure, then
\begin{equation*}
\rho_n(x_1,\dots,x_n)=\mathbb{P}[\text{there are particles at\,} x_1,\dots, x_n].
\end{equation*}
If instead $\mathfrak{X}=\mathbb{R}$ and $\lambda$ is Lebesgue measure, then we can think of the correlation function as
\begin{equation*}
\rho_n(x_1,\dots,x_n)=\lim_{\Delta x_i\to 0}\frac{\mathbb{P}[\text{one particle in each\,} [x_i,,x_i+\Delta x_i], 1\le i\le n]}{\Delta x_1\dots\Delta x_n}.
\end{equation*}
Thus, in the continuous case $\rho_n(x_1,\dots,x_n)$ is the density of $n$-tuples of points in the process, and it is not a probability density. If $\xi$ has correlation functions $\rho_n$, $n\ge 1$, then it is uniquely determined by them.

We can now define what is meant by a determinantal point process.

\begin{defi}\label{def:determinantal}
Let $\xi$ be a point process on a complete, separable metric space $\mathfrak{X}$, with reference measure $\lambda$, all of whose correlation functions $\rho_n$ exist.
Furthermore, let $K:\mathfrak{X}\times\mathfrak{X}\to\mathbb{C}$ be a measurable function. We say that $\xi$ is a determinantal point process on $\mathfrak{X}$ with
correlation kernel $K$ if
\begin{equation}\label{detcorrfcns}
\rho_n(x_1,\dots,x_n)=\det(K(x_i,x_j))_{1\le i,j\le n}
\end{equation}
for all $x_1,\dots, x_n\in\mathfrak{X}$, $n\ge 1$.
\end{defi}

A determinantal point process is a simple point process. Note that the correlation kernel $K(x,y)$ is not unique. If $c:\mathfrak{X}\to\mathbb{C}$ is measurable
and $\neq 0$, then the \emph{conjugated kernel}
\begin{equation}\label{conjugatekernel}
\frac{c(x)}{c(y)}K(x,y)
\end{equation}
is also a correlation kernel. By combining (\ref{corrfcnsdef}) and (\ref{detcorrfcns}) we see that for a determinantal point process,
\begin{equation}\label{detprocexpansion}
\mathbb{E}\left[\prod_{i}(1-\phi(x_i))\right]=\sum_{n=0}^\infty\frac{(-1)^n}{n!}\int_{\mathfrak{X}^n}\prod_{j=1}^n\phi(x_j)\det(K(x_i,x_j))_{1\le i,j\le n}\,d^n\lambda(x).
\end{equation}
We often think of $K(x,y)$ as the kernel of an integral operator $K:L^2(\mathfrak{X},\lambda)\to L^2(\mathfrak{X},\lambda)$,
\begin{equation*}
(Kf)(x)=\int_{\mathfrak{X}} K(x,y)f(y)\,d\lambda(y),
\end{equation*}
provided that this operator is well-defined. If $\phi$ has support in $B$, we can write (\ref{detprocexpansion}) as
\begin{equation}\label{fredholmexp}
\mathbb{E}\left[\prod_{i}(1-\phi(x_i))\right]=\sum_{n=0}^\infty\frac{(-1)^n}{n!}\int_{\mathfrak{X}^n}\det(\phi(x_i)K(x_i,x_j)\mathbb{I}_B(x_j))_{1\le i,j\le n}\,d^n\lambda(x).
\end{equation}
If $\phi K\mathbb{I}_B$ is the kernel of a trace class operator on $L^2(\mathfrak{X},\lambda)$, such that 
$\Tr \phi K\mathbb{I}_B=\int_{\mathfrak{X}}(\phi K\mathbb{I}_B)(x,x)\,d\lambda(x)$,
the right side of (\ref{fredholmexp}) is the Fredholm expansion of a Fredholm determinant, and we obtain
\begin{equation}\label{fredholmdet}
\mathbb{E}\left[\prod_{i}(1-\phi(x_i))\right]=\det(I-\phi K\mathbb{I}_B)_{L^2(\mathfrak{X},\lambda)}=\det(I-\phi K)_{L^2(B,\lambda)}.
\end{equation}

If $\psi:\mathbb{X}\to\mathbb{C}$ is bounded and has bounded support, we can investigate the point process $\xi$ through the linear statistic
\begin{equation*}
\sum_i \psi(x_i)=\int_{\mathfrak{X}}\psi(t)\,d\xi(t).
\end{equation*}
The study of this linear statistic connects well with the formulas above since its Laplace transform
\begin{equation}\label{Laplacetransf}
\mathbb{E}\left[e^{\sum_i\psi(x_i)}\right]=\mathbb{E}\left[\prod_{i}(1-\phi(x_i))\right],
\end{equation}
where $\phi=1-e^{\psi}$. We could, for example, take $\psi(x)=-t\mathbb{I}(x)$ in order to study the number of points in a bounded set $B$.
Other quantities that fit the above framework are hole or gap probabilities. If $B$ is a bounded, measurable set, the taking $\phi=-\mathbb{I}_B$,
we obtain
\begin{equation}\label{Laplacetrans}
\mathbb{P}[\text{no particle in\,}B]=\det(I-K)_{L^2(B,\lambda)}.
\end{equation}

Let $\mathfrak{X}$ be a finite set, with reference measure the counting measure, and $\xi$ a determinantal point process on $\mathfrak{X}$,
\begin{equation*}
\xi= \sum_i \delta_{x_i}.
\end{equation*}
We can then construct a new point process, the {\it dual point process} $\xi^*$ on $\mathfrak{X}$, by letting
\begin{equation*}
\xi^*= \sum_i \delta_{x_i'},
\end{equation*}
where $\{x_i'\}=\mathfrak{X}\setminus\{x_i\}$. Then $\xi^*$ is also a determinantal point process with kernel $I-K$, since
\begin{align}\label{holeprobFredholm}
&\mathbb{P}[\xi^* \,\,\,\text{has particles at\,}y_1,\dots y_n]=\mathbb{P}[\xi \,\,\text{has no particles at\,}y_1,\dots y_n]\\
&=\det(I-K)_{L^2(\{y_1,\dots,y_n\},\lambda)}=\det(\delta_{ij}-K(y_i,y_j))_{1\le i,j\le n}.\notag
\end{align}

\subsection{Measures defined by products of determinants} \label{sec:2.2}

An important way to get a determinantal point process is via measures defined by products of determinants which can be thought of as determinantal transition functions.
These types of measures occur naturally in connection with non-intersecting paths and in random matrix theory. We will meet several examples below.
Let $L<R$ be two given integers, and let $X$ be a complete, separable metric space with a Borel measure $\mu$, and let $X_L$ and $X_R$ be given
sets. For $L<r<R$, we let $X_r=X$. Also, let $p_{r,r+1}:X_r\times X_{r+1}\to\mathbb{C}$, $L\le r<R$, be a sequence of measurable functions that we call
{\it transition functions}. Fix $(x_1^L,\dots,x_M^L)\in X_L^M$ and $(x_1^R,\dots,X_M^R)\in X_R^M$, for a given $M\ge 1$, and consider the measure on
$(X^M)^{R-L-1}$ with density
\begin{equation}\label{productmeasure}
q_{L,R,M}(\mathbf{x})=\frac 1{(M!)^{R-L-1}Z_{L,R,M}}\prod_{r=L}^{R-1}\det(p_{r,r+1}(x_j^r,x_k^{r+1}))_{1\le j,k\le M},
\end{equation}
with respect to $d\mu(\mathbf{x})=\prod_{r=L+1}^{R-1}d^M\mu(x^r)$, where $\mathbf{x}=(x^{L+1},\dots,x^{R-1})$, $x^r=(x_1^r,\dots,x_M^r)\in X^M$.
We assume that all the determinants are $\ge 0$, and choose the normalization constant, the partition function, $Z_{L,R,M}>0$, so that (\ref{productmeasure})
becomes a probability measure. Let $\mathfrak{X}=\{L+1,\dots,R-1\}\times X$ with reference measure $\lambda=\nu\otimes\mu$, where $\nu$ is counting measure on
$\{L+1,\dots,R-1\}$. We can map $x_j^r$ to $(r,x_j^r)$ and in this way $\mathbf{x}$ is mapped to a point in $\mathfrak{X}$. In this way we get a point process on
$\mathfrak{X}$.

If $L\le r<s\le R$, $x\in X_r$ and $y\in X_s$, we define
\begin{equation}\label{prs}
p_{r,s}(x,y)=\int_{X^{s-r-1}}p_{r,r+1}(x,t_1)\dots p_{s-1,s}(t_{s-r-1},y)\,d^{s-r-1}\mu(t),
\end{equation}
and if $r\ge s$, then $p_{r,s}=0$. Furthermore, we define the matrix
\begin{equation}\label{matrixA}
A=(p_{L,R}(x_j^L,x_k^R))_{1\le j,k\le M}.
\end{equation}
It follows from the Cauchy-Binet or Andrieff identity (\ref{CBAnd}) that $\det A=Z_{L,R,M}>0$, so $A$ is invertible and we can define
\begin{equation}\label{generalKtilde}
\tilde{K}_{L,R,M}(r,u;s,v)=\sum_{i,j=1}^M p_{r,R}(u,x_j^R)(A^{-1})_{ji}p_{L,s}(x_i^L,v),
\end{equation}
where $(r,u), (s,v)\in\mathfrak{X}$. Set
\begin{equation}\label{generalK}
K_{L,R,M}(r,u;s,v)=-p_{r,s}(u,v)+\tilde{K}_{L,R,M}(r,u;s,v).
\end{equation}
We then have the following theorem.

\begin{thma}\label{productmeasurethm}
The point process on $\mathfrak{X}$ defined above is a determinantal point process with correlation kernel $K_{L,R,M}$ given by (\ref{generalK}).
\end{thma}

For the proof see e.g \cite{JoHouch}.

Although we have a formula for the correlation kernel, it is often difficult to find a useful formula for it since the inverse $A^{-1}$ may be hard to compute. We will
consider three cases where it is possible to rewrite the formula further into a form that may be more convenient for further analysis.

\subsubsection{Using Cramer's rule} \label{sec:2.2.1}
Assume that there is a linear operator $T_v$, $v\in X$, acting on the variable $v$, such that
\begin{equation*}
p_{L,s}(x_i^L,v)=T_vp_{L,R}(x_i^L,g(s,v)),
\end{equation*}
where $g(s,v)$ is some function of $s,v$. Let $A[j,s,v]$ be the matrix $A$ in (\ref{matrixA}) with column $j$ replaced by
\begin{equation*}
\left(\begin{matrix}
  p_{L,R}(x_1^L,g(s,v))\\
  \vdots\\
   p_{L,R}(x_M^L,g(s,v))
 \end{matrix}\right).
 \end{equation*}
Then, by Cramer's rule,
\begin{equation}\label{CramersKtilde}
\tilde{K}(r,u;s,v)=T_v\sum_{j=1}^N p_{r,R}(u,x_j^R)\frac{\det A[j,g(s,v)]}{\det A}.
\end{equation}
Whether this is useful depends of course on whether we can find a good operator $T_v$ and compute the determinants in some
good form. Recall from above that $\det A$ equals the partition function and in many models that are of interest the partition function
actually has some nice form which indicates that this approach may be useful. The idea behind the operator $T_v$ is that the determinant $\det A[j,s,v]$
should have the same form as $\det A$ so that it also can be computed. We will give an application to interlacing particle systems in section \ref{sec:5}.

\subsubsection{Infinite Toeplitz matrix} \label{sec:2.2.2}

Consider the case when $X=X_L=X_R=\mathbb{Z}$, $x_j^L=x_j^R=d-j$, $1\le j\le M$, and the transition functions are given by
\begin{equation}\label{prphir}
p_{r,r+1}(x,y)=\hat{\phi}_r(y-x),
\end{equation}
for $x,y\in\mathbb{Z}$, and $\hat{\phi}_r(k)$ is the $k$:th Fourier coefficient of the complex-valued function function $\phi_r(z)$ on $\mathbb{T}$, the unit circle in the complex plane. This situation occurs in particular for Schur processes, see section \ref{sec4.3}.
We assume that the functions $\phi_r$, $L\le r<R$, satisfy:

\begin{condition}\label{WH}
\item There is an $\epsilon>0$ such that $\phi_r$ has a Wiener-Hopf factorization $\phi_r(z)=\phi_r^+(z)\phi_r^-(z)$, $z\in\mathbb{T}$, where $\phi_r^+$ is analytic and non-zero in
$|z|<1+\epsilon$, and  $\phi_r^-$ is analytic and non-zero in
$|z|>1-\epsilon$ including at infinity. Also assume that we have normalized so that $\phi^+(0)=\phi^-(\infty)=1$
\end{condition}

For more on Wiener-Hopf factorizations see e.g. \cite{Boetch}. 

For $L\le r<s<R$, we write
\begin{equation}\label{phirs}
\phi_{r,s}(z)=\phi_r(z)\phi_{r+1}(z)\cdots\phi_{s-1}(z).
\end{equation}
Note that $\phi_{r,s}$ has a Wiener-Hopf factorization, $\phi_{r,s}=\phi_{r,s}^+(z)\phi_{r,s}^-(z)$
\begin{equation}\label{phirspm}
\phi_{r,s}^\pm(z)=\phi_r^\pm(z)\phi_{r+1}^\pm(z)\cdots\phi_{s-1}^\pm(z).
\end{equation}
Then,
\begin{equation}\label{prsintformula}
p_{r,s}(x,y)=\mathbb{I}_{r<s}\hat{\phi}_{r,s}(y-x)=\frac{\mathbb{I}_{r<s}}{2\pi\mathrm{i}}\int_{\gamma_1}z^{x-y}\phi_{r,s}(z)\frac{dz}z.
\end{equation}
We see from (\ref{matrixA}) that 
\begin{equation}\label{matrixAToeplitz}
A=(\hat{\phi}_{L,R}(j-k))_{1\le j,k\le M}\doteq T_M(\phi_{L,R}),
\end{equation}
is a size $M$ Toeplitz matrix with symbol $\phi_{L,R}$.

In some situations, as we will see in section \ref{sec:4}, we can take the limit $M\to\infty$ in the kernel (\ref{generalK}) and get a limiting kernel
\begin{equation}\label{KLR}
K_{L,R}(r,u;s,v)=-p_{r,s}(u,v)+\sum_{i,j=1}^\infty\hat{\phi}_{r,R}(d-j-u)T(\phi_{L,R})^{-1}_{ji}\hat{\phi}_{L,s}(v+i-d),
\end{equation}
where $T(\phi_{L,R})^{-1}$ is the infinite Toeplitz matrix with symbol $\phi_{L,R}$,
\begin{equation*}
T(\phi_{L,R})=(\hat{\phi}_{L,R}(j-k))_{1\le j,k<\infty}.
\end{equation*}
The good thing about (\ref{KLR}) is that since we have the Wiener-Hopf factorization $\phi_{r,s}=\phi_{r,s}^+\phi_{r,s}^-$, the following formula holds,
\begin{equation}\label{Toeplitzinverse}
T(\phi_{L,R})^{-1}=T(\frac 1{\phi_{L,R}^+})T(\frac 1{\phi_{L,R}^-}),
\end{equation}
see e.g. \cite{Boetch}.
Our analyticity assumptions imply that all Fourier coefficients decay exponentially and from this it is not so hard to see that the infinite sum in (\ref{KLR}) converges.
In this case we can get a formula for the kernel that is useful for asymptotic analysis.

\begin{thma}\label{InfiniteToeplitzcasekernel}
We have the following contour integral expression for the kernel (\ref{KLR}),
\begin{equation}\label{KLRcontour}
K_{L,R}(r,u;s,v)=-\frac {\mathbb{I}_{r<s}}{2\pi\mathrm{i}}\int_{\gamma_1}z^{u-v}\phi_{r,s}(z)\frac{dz}z+\frac 1{(2\pi \mathrm{i})^2}\int_{\gamma_{\rho_1}}dz\int_{\gamma_{\rho_2}}dw
\frac{z^{u-d}}{w^{v-d+1}(w-z)}\frac{\phi^-_{r,R}(z)\phi^+_{L,s}(w)}{\phi^-_{s,R}(w)\phi^+_{L,r}(z)},
\end{equation}
where $1-\epsilon<\rho_1<\rho_2<1+\epsilon$,
where $\epsilon$ is as in condition \ref{WH} above.
\end{thma}

\begin{proof}  It follows from (\ref{prsintformula}) and analyticity that we can use Cauchy's theorem to deform the contours and get
\begin{align}\label{kerneltoeplitzcomp}
&\sum_{i,j=1}^\infty\hat{\phi}_{r,R}(d-j-u)T(\phi_{L,R})^{-1}_{ji}\hat{\phi}_{L,s}(v+i-d)\\
&=\frac 1{(2\pi \mathrm{i})^2}\int_{\gamma_{\rho_1}}dz\int_{\gamma_{\rho_2}}dw\phi_{r,R}(z)\phi_{L,s}(w)\frac{z^{u-d-1}}{w^{v-d+1}}
\left(\sum_{i,j=1}^\infty z^jT(\phi_{L,R})^{-1}_{ji}w^{-i}\right).\notag
\end{align}
Using (\ref{Toeplitzinverse}) we obtain
\begin{align}\label{InfToeplitzgen}
\sum_{i,j=1}^\infty z^jT(\phi_{L,R})^{-1}_{ji}w^{-i}&=\sum_{i,j=1}^\infty z^j\left(\sum_{k=1}^\infty \widehat{\bigg(\frac 1{\phi_{L,R}^+}\bigg)}_{j-k}
\widehat{\bigg(\frac 1{\phi_{L,R}^-}\bigg)}_{k-i}\right)w^{-i}\\
&=\sum_{k=1}^\infty\left(\sum_{j=1}^\infty\widehat{\bigg(\frac 1{\phi_{L,R}^+}\bigg)}_{j-k}z^{j-k}\right)\left(\sum_{i=1}^\infty\widehat{\bigg(\frac 1{\phi_{L,R}^-}\bigg)}_{j-k}w^{k-i}\right)
\left(\frac zw\right)^k\notag\\
&=\frac z{w-z}\frac 1{\phi_{L,R}^+(z)\phi_{L,R}^-(w)},\notag
\end{align}
since $|z/w|<1$, and by the fact that $\widehat{\big(\frac 1{\phi_{L,R}^+}\big)}_{k}=0$ for $k<0$, and
$\widehat{\big(\frac 1{\phi_{L,R}^-}\big)}_{k}=0$ for $k>0$. Inserting this into the right side of (\ref{kerneltoeplitzcomp}) and cancelling factors in the Wiener-Hopf factorization proves the theorem.
\end{proof}

\subsubsection{Finite Toeplitz matrix} \label{sec:2.2.3}

We consider the same setting as in the previous section for the case of an infinite Toeplitz matrix. If we cannot take the limit $M\to\infty$ in (\ref{generalK}), then we have to deal with 
the inversion of a finite Toeplitz matrix. We will encounter this case below in the context of the random tiling model called the Double Aztec diamond, see section \ref{sec4.5}
The correlation kernel is then given by
\begin{equation}\label{KfiniteMkernel}
K_{L,R,M}(r,u;s,v)=-p_{r,s}(u,v)+\tilde{K}_{L,R,M}(r,u;s,v),
\end{equation}
where $p_{r,s}$ is given by (\ref{prsintformula}), and
\begin{equation}\label{finiteMkernel}
\tilde{K}_{L,R,M}(r,u;s,v)=\sum_{i,j=1}^M\hat{\phi}_{r,R}(d-j-u)T_M(\phi_{L,R})^{-1}_{ji}\hat{\phi}_{L,s}(v+i-d).
\end{equation}
If $f:\mathbb{T}\mapsto\mathbb{C}$ is a symbol we denote the size $n$ Toeplitz determinant with symbol $f$ by
\begin{equation*}
D_n[f(\zeta)]=\det T_n(f)=\det\left(\frac 1{2\pi \mathrm{i}}\int_{\gamma_1}\zeta^{k-j}f(\zeta)\frac{d\zeta}{\zeta}\right)_{1\le j,k\le n}.
\end{equation*}
The following proposition is the key result to rewrite (\ref{finiteMkernel}) in a good way.

\begin{prop}\label{prop:finiteToeplitz}
Let $f\in L^1(\mathbb{T})$ and $z,w\in\mathbb{C}$. Then
\begin{equation}\label{FiniteToeplitzinverse}
\sum_{i,j=1}^n z^jT_n(f)^{-1}_{ji}w^{-i}=\frac zw\frac{D_{n-1}\big[(1-\zeta/w)(1-z/\zeta)f(\zeta)\big]}{D_n[f(\zeta)]}.
\end{equation}
\end{prop}

\begin{proof}
If $A$ is a matrix and $A^{(ij)}$ denotes the matrix we obtain when we erase row $i$ and column $j$, then
\begin{equation*}
(A^{-1})_{ji}=(-1)^{i+j}\frac{\det A^{(ij)}}{\det A}.
\end{equation*}
Now,
\begin{equation*}
\det T_n(f)^{(ij)}=\det\left(\frac 1{2\pi \mathrm{i}}\int_{\gamma_1}\psi_{\ell}^{(i)}\big(\frac 1{\zeta}\big)\psi_m^{(j)}(\zeta)f(\zeta)\frac{d\zeta}{\zeta}\right)_{1\le \ell,m<n},
\end{equation*}
where
\begin{equation*}
\psi_m^{(j)}(\zeta)=
\begin{cases} \zeta^m &,1\le m<j\\ \zeta^{m+1} &,j\le m<n. \end{cases}
\end{equation*}
Note that, by (\ref{ClassicalSchur}) and (\ref{JacobiTrudi2}),
\begin{equation*}
\frac{\det\big(\psi_{\ell}^{(j)}(\zeta_m)\big)_{1\le \ell,m<n}}{\Delta_{n-1}(\zeta)}=s_{1^{n-j}}(\zeta_1,\dots,\zeta_{n-1})=e_{n-j}(\zeta_1,\dots,\zeta_{n-1}).
\end{equation*}
Using the Andrieff identity (\ref{CBAnd}) we find
\begin{align*}
\det T_n(f)^{(ij)}&=\frac 1{(2\pi \mathrm{i})^{n-1}(n-1)!}\int_{\gamma_1^{n-1}}\det\left(\psi_{\ell}^{(i)}\big(\frac 1{\zeta_m}\big)\right)
\det\left(\psi_{\ell}^{(j)}(\zeta_m)\right)\prod_{m=1}^{n-1}f(\zeta_m)\frac{d\zeta_m}{\zeta_m}\\
&=\frac 1{(2\pi \mathrm{i})^{n-1}(n-1)!}\int_{\gamma_1^{n-1}}e_{n-i}\big(\frac 1{\zeta_1},\dots,\frac 1{\zeta_{n-1}}\big)e_{n-j}(\zeta_1,\dots,\zeta_{n-1})
\big|\Delta_{n-1}(\zeta)\big|^2\prod_{m=1}^{n-1}f(\zeta_m)\frac{d\zeta_m}{\zeta_m}.
\end{align*}
Hence, by (\ref{elsymmpol})
\begin{align*}
&\sum_{i,j=1}^n z^j(-1)^{i+j}T_n(f)^{(ij)}w^{-i}\\
&=\frac 1{(2\pi \mathrm{i})^{n-1}(n-1)!}\int_{\gamma_1^{n-1}}\left(\sum_{i=1}^n\bigg(-\frac 1w\bigg)^i e_{n-i}\big(\frac 1{\zeta_1},\dots,\frac 1{\zeta_{n-1}}\big)\right)\\
&\times\left(\sum_{j=1}^n(-z)^je_{n-j}(\zeta_1,\dots,\zeta_{n-1})\right)\big|\Delta_{n-1}(\zeta)\big|^2\prod_{m=1}^{n-1}f(\zeta_m)\frac{d\zeta_m}{\zeta_m}\\
&=\frac zw\frac 1{(2\pi \mathrm{i})^{n-1}(n-1)!}\int_{\gamma_1^{n-1}}\big|\Delta_{n-1}(\zeta)\big|^2
\prod_{m=1}^{n-1}(z-\zeta_m)\big(\frac 1w-\frac 1{\zeta_m}\big)f(\zeta_m)\frac{d\zeta_m}{\zeta_m}\\
&=\frac zw D_{n-1}\big[(1-\zeta/w)(1-z/\zeta)f(\zeta)\big].
\end{align*}

\end{proof}

\begin{remark} {\rm Note that the expression in the right side of  (\ref{FiniteToeplitzinverse}) converges to the last expression in   (\ref{InfToeplitzgen})
as $M\to\infty$ by the strong Szeg\H{o} limit theorem, compare (\ref{etotheG}).}
\end{remark}

We want to rewrite the Toeplitz determinants in (\ref{FiniteToeplitzinverse}) as Fredholm determinants and for this we will use the Geronimo-Case/Borodin-Okounkov identity
which is given in the next proposition.
For a proof see \cite{BoOk} or \cite{BaWi}.

\begin{prop}\label{prop:BOformula}
Assume that $g:\mathbb{T}\to\mathbb{C}$ has a Wiener-Hopf factorization $g=g^+g^-$, where $g^+$ is analytic and non-zero in $|\zeta|<1$, and $g^-$ is analytic and
non-zero in $|\zeta|>1$. Also, assume that $g^+(0)=g^-(\infty)=1$.
Set
\begin{equation}\label{BOKjk}
\mathcal{K}(j,k)=\sum_{\ell=0}^\infty\widehat{\bigg(\frac{g^-}{g^+}\bigg)}_{j+\ell}\widehat{\bigg(\frac{g^+}{g^-}\bigg)}_{-k-\ell}
\end{equation}
and
\begin{equation}\label{BOG}
G=\sum_{j=1}^\infty j\big(\widehat{\log g}\big)_j\big(\widehat{\log g}\big)_{-j}.
\end{equation}
Then
\begin{equation}\label{BOformula}
D_n\big[g(\zeta)\big]= e^G \det(I-\mathcal{K})_{\bar{\ell}^2(n+1)},
\end{equation}
where $\bar{\ell}^2(n)=\ell^2\big(\{n,n+1,\dots\}\big)$.
\end{prop}

Before we can give a formula for the kernel (\ref{finiteMkernel}) we have to intoduce some notation. Choose the radii $\rho_1,\rho_2, \rho_3,\sigma_1, \sigma_2$ so that
\begin{equation}\label{radii1}
1-\epsilon<\rho_3<\rho_1<\sigma_1<\sigma_2<\rho_2<1+\epsilon,
\end{equation}
where $\epsilon$ is the same as that in condition \ref{WH} above. Assuming that the radii satisfy (\ref{radii1}), we define
\begin{equation}\label{Kzero}
\mathcal{K}_0(j,k)=\frac 1{(2\pi \mathrm{i})^2}\int_{\gamma_{\sigma_1}}d\omega\int_{\gamma_{\sigma_2}}d\zeta \frac{\omega^k}{\zeta^{j+1}(\zeta-\omega)}
\frac{\phi_{L,R}^+(\omega)\phi_{L,R}^-(\zeta)}{\phi_{L,R}^-(\omega)\phi_{L,R}^+(\zeta)},
\end{equation}
and
\begin{align}\label{abformulas}
a_{s,v}(j)&=\frac 1{(2\pi \mathrm{i})^2}\int_{\gamma_{\rho_2}}dw\int_{\gamma_{\sigma_2}}d\zeta \frac{w^{d-v-1}}{\zeta^{j+1}(\zeta-w)}
\frac{\phi_{L,s}^+(w)\phi_{L,R}^-(\zeta)}{\phi_{s,R}^-(w)\phi_{L,R}^+(\zeta)}\\
b_{r,u}(k)&=\frac 1{(2\pi \mathrm{i})^2}\int_{\gamma_{\rho_1}}dz\int_{\gamma_{\sigma_1}}d\omega \frac{\omega^k}{z^{d-u}(\omega-z)}
\frac{\phi_{r,R}^-(z)\phi_{L,R}^+(\omega)}{\phi_{L,r}^+(z)\phi_{L,R}^-(\omega)}.\notag
\end{align}
Furthermore, we define
\begin{equation}\label{MLR}
M_{L,R}(r,u;s,v)=\frac 1{(2\pi \mathrm{i})^2}\int_{\gamma_{\rho_1}}dz\int_{\gamma_{\rho_2}}dw \frac{z^{u-d}}{w^{v-d+1}(w-z)}
\frac{\phi_{r,R}^-(z)\phi_{L,s}^+(w)}{\phi_{L,r}^+(z)\phi_{s,R}^-(w)},
\end{equation}
and
\begin{equation}\label{Mstar}
M_{L,R}^*(r,u;s,v)=\frac 1{(2\pi \mathrm{i})^2}\int_{\gamma_{\rho_1}}dz\int_{\gamma_{\rho_3}}dw \frac{z^{u-d}}{w^{v-d+1}(z-w)}
\frac{\phi_{r,R}^-(z)\phi_{L,s}^+(w)}{\phi_{L,r}^+(z)\phi_{s,R}^-(w)}.
\end{equation}

We can now state

\begin{thma}\label{thm:KMfinite} 
The kernel (\ref{KfiniteMkernel}) is given by
\begin{align}\label{KLRMfinal}
K_{L,R,M}(r,u;s,v)&=-\frac{\mathbb{I}_{r<s}}{2\pi \mathrm{i}}\int_{\gamma_1} z^{u-v}\phi_{r,s}(z)\frac{dz}z+M_{L,R}(r,u;s,v)\\
&-\sum_{k=M}^\infty((I-\mathcal{K}_0)_M^{-1}a_{s,v})(k)b_{r,u}(k).\notag
\end{align}
Here the notation $(I-\mathcal{K}_0)_M^{-1}$ means that we take the inverse on the space $\ell^2(\{M,M+1,\dots\})$. The correlation kernel for the dual particle system
as discussed at (\ref{holeprobFredholm}) in section \ref{sec:2.1} is given by
\begin{align}\label{dualKLRMfinal}
K_{L,R,M}^*(r,u;s,v)&=-\frac{\mathbb{I}_{s<r}}{2\pi \mathrm{i}}\int_{\gamma_1} \frac{z^{u-v}}{\phi_{s,r}(z)}\frac{dz}z+M_{L,R}^*(r,u;s,v)\\
&+\sum_{k=M}^\infty((I-\mathcal{K}_0)_M^{-1}a_{s,v})(k)b_{r,u}(k).\notag
\end{align}
\end{thma}

\begin{proof} It follows from (\ref{prsintformula}), (\ref{finiteMkernel}) and (\ref{FiniteToeplitzinverse}) that
\begin{align}\label{Ktildefinite}
\tilde{K}_{L,R,M}(r,u;s,v)&=\frac 1{(2\pi \mathrm{i})^2}\int_{\gamma_{\rho_1}}\frac{dz}z\int_{\gamma_{\rho_2}}\frac{dw}w \frac{z^{u-d}}{w^{v-d}}\phi_{r,R}(z)\phi_{L,s}(w)
\left(\sum_{i,j=1}^Mz^jT_M(\phi_{L,R})^{-1}_{ji}w^{-i}\right)\\
&=\frac 1{(2\pi \mathrm{i})^2}\int_{\gamma_{\rho_1}}\frac{dz}z\int_{\gamma_{\rho_2}}\frac{dw}w \frac{z^{u+1-d}}{w^{v+1-d}}\phi_{r,R}(z)\phi_{L,s}(w)
\frac{D_{M-1}\big[(1-\zeta/w)(1-z/\zeta)\phi_{L,R}(\zeta)\big]}{D_M[\phi_{L,R}(\zeta)]}.\notag
\end{align}
We now rewrite the Toeplitz determinant in the numerator in the right side of (\ref{Ktildefinite}) using the identity (\ref{BOformula}). When we have the symbol
$(1-\zeta/w)(1-z/\zeta)\phi_{L,R}(\zeta)$,  the kernel (\ref{BOKjk}) becomes
\begin{equation}\label{Kzw}
\mathcal{K}^{(z,w)}(j,k)=\frac 1{(2\pi \mathrm{i})^2}\int_{\gamma_{\sigma_1}}d\omega\int_{\gamma_{\sigma_2}}d\zeta \frac{\omega^k}{\zeta^{j+1}(\zeta-\omega)}
\frac{\phi_{L,R}^+(\omega)\phi_{L,R}^-(\zeta)}{\phi_{L,R}^-(\omega)\phi_{L,R}^+(\zeta)}\frac{(\zeta-z)(\omega-w)}{(\zeta-w)(\omega-z)}.
\end{equation}
With the same symbol we denote the quantity (\ref{BOG}) by $G(z,w)$. A computation gives
\begin{equation}\label{etotheG}
e^{G(z,w)}=\frac 1{(1-z/w)\phi_{L,R}^-(w)\phi_{L,R}^+(z)}e^{G_0},
\end{equation}
where
\begin{equation*}
G_0=\sum_{j=1}^\infty j\big(\widehat{\log \phi_{L,R}}\big)_j\big(\widehat{\log \phi_{L,R}}\big)_{-j}.
\end{equation*}
Note that $G(0,\infty)=G_0$ and that if we set $z=0, w=\infty$ in (\ref{Kzw}) we get
\begin{equation*}
\mathcal{K}^{(0,\infty)}(j,k)=\mathcal{K}_0(j-1,k-1).
\end{equation*}
It follows from this, proposition \ref{prop:BOformula}, (\ref{Kzw}) and (\ref{etotheG}) that
\begin{align*}
\frac{D_{M-1}\big[(1-\zeta/w)(1-z/\zeta)\phi_{L,R}(\zeta)\big]}{D_M[\phi_{L,R}(\zeta)]}&=
\frac 1{(1-z/w)\phi_{L,R}^-(w)\phi_{L,R}^+(z)}\frac{\det\big(I-\mathcal{K}^{(z,w)}\big)_{\bar{\ell}^2(M)}}{\det\big(I-\mathcal{K}^{(0,\infty)}\big)_{\bar{\ell}^2(M+1)}}\\
&=\frac 1{(1-z/w)\phi_{L,R}^-(w)\phi_{L,R}^+(z)}\frac{\det\big(I-\mathcal{K}^{(z,w)}\big)_{\bar{\ell}^2(M)}}{\det\big(I-\mathcal{K}_0\big)_{\bar{\ell}^2(M)}}.
\end{align*}
We can insert this into (\ref{Ktildefinite}) and obtain
\begin{align}\label{KtildeFredholm}
\tilde{K}_{L,R,M}(r,u;s,v)&=\frac 1{(2\pi \mathrm{i})^2}\int_{\gamma_{\rho_1}}dz\int_{\gamma_{\rho_2}}dw\frac{z^{u-d}}{w^{v+1-d}(w-z)}
\frac{\phi_{r,R}^-(z)\phi_{L,s}^+(w)}{\phi_{L,r}^-(z)\phi_{s,R}^-(w)}\\
&\times\frac{\det\big(I-\mathcal{K}^{(z,w)}\big)_{\bar{\ell}^2(M)}}{\det\big(I-\mathcal{K}_0\big)_{\bar{\ell}^2(M)}}.
\end{align}
We can rewrite this formula further by noticing that
\begin{equation*}
\frac 1{\zeta-\omega}\frac{(\zeta-z)(\omega-w)}{(\zeta-w)(\omega-z)}=\frac 1{\zeta-\omega}-\frac{(w-z)}{(\zeta-w)(\omega-z)}.
\end{equation*}
Hence, if we define
\begin{align*}
c_1^w(j)&=\frac 1{2\pi \mathrm{i}}\int_{\gamma_{\sigma_2}}d\zeta \frac 1{\zeta^{j+1}(\zeta-w)}\frac{\phi_{L,R}^-(\zeta)}{\phi_{L,R}^+(\zeta)},\\
c_2^z(k)&=\frac 1{2\pi \mathrm{i}}\int_{\gamma_{\sigma_1}}d\omega \frac {\omega^k}{\omega-z}\frac{\phi_{L,R}^+(\omega)}{\phi_{L,R}^-(\omega)},
\end{align*}
we see that
\begin{equation*}
\mathcal{K}^{(z,w)}(j,k)=\mathcal{K}_0(j,k)-(w-z)c_1^w(j)c_2^z(k).
\end{equation*}
Thus
\begin{equation}\label{Kzwformula}
\det\big(I-\mathcal{K}^{(z,w)}\big)_{\bar{\ell}^2(M)}=\det\big(I-\mathcal{K}_0-(w-z)c_1^w\otimes c_2^z\big)_{\bar{\ell}^2(M)}.
\end{equation}
Exactly the same proof as that of lemma 2.1 in \cite{JoNIBM} gives the following formula
\begin{align}\label{Fredholmformula}
&\int_{\gamma_{\rho_1}}dz\int_{\gamma_{\rho_2}}dw\frac{F(z)G(w)}{w-z}\det\big(I-\mathcal{K}_0-(w-z)c_1^w\otimes c_2^z\big)_{\bar{\ell}^2(M)}\notag\\
&=\left(\int_{\gamma_{\rho_1}}dz\int_{\gamma_{\rho_2}}dw\frac{F(z)G(w)}{w-z}-1\right)\det(\big(I-\mathcal{K}_0\big)_{\bar{\ell}^2(M)}\notag\\
&+\det\bigg(I-\mathcal{K}_0-\big(\int_{\gamma_{\rho_2}}G(w)c_1^w\,dw\big)\otimes \big(\int_{\gamma_{\rho_1}}F(z)c_2^z\,dz\big) \bigg)_{\bar{\ell}^2(M)},
\end{align}  
if $F(z)$ and $G(w)$ are integrable functions on their respective contours.
Combining this with (\ref{KtildeFredholm}) and (\ref{Kzwformula}) we arrive at the formula
\begin{align}\label{Ktildenew}
\tilde{K}_{L,R,M}(r,u;s,v)&=\frac 1{(2\pi \mathrm{i})^2}\int_{\gamma_{\rho_1}}dz\int_{\gamma_{\rho_2}}dw\frac{z^{u-d}}{w^{v+1-d}(w-z)}
\frac{\phi_{r,R}^-(z)\phi_{L,s}^+(w)}{\phi_{L,r}^-+z)\phi_{s,R}^-(w)}\\
&+\frac{\det\big(I-\mathcal{K}_0-a_{s,v}\otimes b_{r,u}\big)_{\bar{\ell}^2(M)}}{\det(\big(I-\mathcal{K})_0\big)_{\bar{\ell}^2(M)}}-1\notag\\
&=M_{L,R}(r,u;s,v)-\sum_{k=M}^\infty((I-\mathcal{K}_0)_M^{-1}a_{s,v})(k)b_{r,u}(k).\notag
\end{align}
This proves (\ref{KLRMfinal}). The dual particle system has the correlation kernel
\begin{equation*}
K_{L,R,M}^*(r,u;s,v)=\delta_{r,s}\delta_{u,v}-K_{L,R,M}(r,u;s,v).
\end{equation*}
Combining this with the relation
\begin{equation}\label{MMstar}
M_{L,R}(r,u;s,v)=-M_{L,R}^*(r,u;s,v)+p_{r,s}(u,v)+q_{s,r}(u,v)+\delta_{r,s}\delta_{u,v},
\end{equation}
where
\begin{equation}\label{qsr}
q_{s,r}(u,v)=\frac {\mathbb{I}_{s<r}}{2\pi \mathrm{i}}\int_{\gamma_1}\frac{z^{u-v}}{\phi_{s,r}(z)}\frac{dz}z,
\end{equation}
which follows from the residue theorem, we obtain (\ref{dualKLRMfinal}).
\end{proof}

\section{Dimer models on bipartite graphs} \label{sec:3}

\subsection{Kasteleyn's method} \label{sec3.1}

Let $\mathcal{G}=(V,E)$ be a graph where $V$ is the set of vertices, or points, and $E$ the set of edges. We assume that $\mathcal{G}$ is a planar, connected, bipartite graph with no cut points. To each edge $e\in E$ we associate a {\it weight} $\nu(e)>0$. A {\it dimer configuration},
or {\it perfect matching}, $C$ on $\mathcal{G}$ is a subset of $E$ such that each vertex in the graph belongs to exactly one edge in $C$. We denote by $\mathcal{C}$ the
set of all dimer configurations on the graph.
The weight of a dimer configuration $C$ is the product of the weights of the edges in $C$,
\begin{equation*}
\nu(C)=\prod_{e\in C}\nu(e).
\end{equation*}
The {\it partition function} is defined by
\begin{equation}\label{dimerZ}
Z=\sum_{C\in\mathcal{C}}\nu(C).
\end{equation}
We can now define a probablity measure on the set of all dimer coverings of $\mathcal{G}$ by
\begin{equation}\label{dimerprob}
\mathbb{P}[C]=\frac{\nu(C)}Z,
\end{equation}
for each dimer cover $C$ of $\mathcal{G}$.

Since $\mathcal{G}$ is bipartite we have a splitting of $V$ into two types of vertices, $V=\mathbf{B}\cup \mathbf{W}$, where the vertices in $\mathbf{B}$ are called {\it black} and those in $\mathbf{W}$ 
are {\it white}.
Any edge $e\in E$ connects a black vertex $b\in \mathbf{B}$ to a white vertex $w\in \mathbf{W}$, we write $e=bw$. Let $\{b_1,\dots,b_n\}$ and $\{w_1,\dots,w_n\}$ be some enumerations of
the black and white vertices respectively. Given such an enumeration it is clear that a dimer configuration $C$ can be written
\begin{equation*}
C=\{b_iw_{\sigma(i)}\,;\,1\le i\le n\},
\end{equation*}
for some permutation $\sigma\in S_n$. We write $C=C(\sigma)$. If a permutation $\sigma\in S_n$ does not correspond to a dimer configuration we set $C(\sigma)=\emptyset$
and $\nu(b_iw_{\sigma(i)})=0$.
We see that
\begin{equation}\label{Zpermanent}
Z=\sum_{\sigma\in S_n}\prod_{i=1}^n\nu(b_iw_{\sigma(i)}),
\end{equation}
so $Z$ is given by a permanent. Permanents are hard to work with and determinants are much nicer objects. Therefore, we want to introduce some signs into the sum in
(\ref{Zpermanent}) so that we get a determinant instead. A {\it Kasteleyn sign} is a function $s:E\to\mathbb{T}$, such that for any face in $\mathcal{G}$ with edges
$e_1,\dots,e_{2k}$ in cyclic order we have the relation
\begin{equation}\label{signrelation}
\frac{s(e_1)\dots s(e_{2k-1})}{s(e_2)\dots s(e_{2k})}=(-1)^{k+1}.
\end{equation}
If $bw$ is not an edge in the graph then we put $s(bw)=1$ by convention; the exact value is unimportant.
We assume that we have such a Kasteleyn sign. Here we will not discuss whether such a function exists, see remark \ref{rem:Kasteleyn}. In concrete cases this we can just define it and check the condition
(\ref{signrelation}). The {\it Kasteleyn operator} $\mathbb{K}:\mathbf{W}\to \mathbf{B}$ is the operator with kernel
\begin{equation}\label{Kasteleynoperator}
\mathbb{K}(b,w)=s(bw)\nu(bw).
\end{equation}
Given an enumeration of the vertices we get the {\it Kasteleyn matrix},
\begin{equation}\label{Kasteleynmatrix}
\mathbb{K}=(\mathbb{K}(b_i,w_j))_{1\le i,j\le n}.
\end{equation}
The role of the Kasteleyn signs and the Kasteleyn matrix $\mathbb{K}$ is to turn (\ref{Zpermanent}) into a formula involving the determinant of the matrix $\mathbb{K}$ instead. In fact we have
the following theorem, \cite{Kas}.
\begin{thma}\label{thm:Kasteleyn}
There is a complex number $S$ with $|S|=1$, independent of the choice of the weights $\{\nu(e)\}_{e\in E}$, so that
\begin{equation}\label{Kasteleyndeterminant}
\det \mathbb{K}=SZ.
\end{equation}
\end{thma}

Before we prove the theorem we state and prove the following lemma on the Kasteleyn signs, \cite{Ken}.

\begin{lemma}\label{lemma:signs}
Consider a planar, bipartite graph $\mathcal{G}$ with no cut points. If a cycle $e_1,\dots,e_{2k}$ of length $2k$ encloses $\ell$ points in the graph, and $s(e)$ is a 
Kasteleyn sign on $\mathcal{G}$, then
\begin{equation}\label{extendedsignrelation}
\frac{s(e_1)\dots s(e_{2k-1})}{s(e_2)\dots s(e_{2k})}=(-1)^{k+\ell+1}.
\end{equation}
\end{lemma}

\begin{proof}
We use induction on the number of enclosed points. If $\ell=0$, then (\ref{extendedsignrelation}) is just (\ref{signrelation}). Take a cycle with $\ell>0$ internal points. Assume
that the result is true if we have a smaller number of internal points. Since the graph is connected and has no cut points, there are two points $p,q$ in the cycle, and points
$r_1,\dots, r_{b-1}$ inside, such that $pr_1\dots r_{b-1}q$ is a path in the graph. Hence, we can split the given cycle $e_1,\dots,e_{2k}$ into two cycles
$e_1\dots e_af_1\dots f_b$ and $e_{a+1}\dots e_{2k}f_b\dots f_1$, where $f_1=pr_1$, $f_2=r_1r_2$, etc., perhaps after a cyclic renumbering of the edges in the original cycle,
which does not change (\ref{extendedsignrelation}). The two new cycles have a smaller number of internal points. If $a$ is even, then $b$ must be even, and by the induction assumption
\begin{equation}\label{signrelationproof1}
\frac{s(e_1)\dots s(e_{a-1})s(f_1)\dots s(f_{b-1})}{s(e_2)\dots s(e_{a})s(f_2)\dots s(f_b)}=(-1)^{(a+b)/2+\ell_1+1},
\end{equation}
and
\begin{equation}\label{signrelationproof2}
\frac{s(e_{a+1})\dots s(e_{2k-1})s(f_b)\dots s(f_{2})}{s(e_{a+2})\dots s(e_{2k})s(f_{b-1})\dots s(f_1)}=(-1)^{(2k-a+b)/2+\ell_2+1},
\end{equation}
where $\ell_i<\ell$, $i=1,2$ are the number of internal points in the two new cycles. Note that $\ell=\ell_1+\ell_2+b-1$. Taking the product of (\ref{signrelationproof1}) and
 (\ref{signrelationproof2}) gives
 \begin{equation*}
 \frac{s(e_1)\dots s(e_{2k-1})}{s(e_2)\dots s(e_{2k})}=(-1)^{k+b+\ell_1+\ell_2+2}=(-1)^{k+\ell+1}.
 \end{equation*}
 The case when $a$ is odd is completely analogous.
\end{proof}

We are now ready for the

\begin{proof} [Proof of theorem \ref{thm:Kasteleyn}]
We have that
\begin{equation}\label{determinant expansion}
\det \mathbb{K}=\sum_{\sigma\in S_n}\sgn{\sigma}\prod_{i=1}^n K(b_i,w_{\sigma(i)})
=\sum_{\sigma\in S_n}\sgn{\sigma}\prod_{i=1}^n s(b_iw_{\sigma(i)})\prod_{i=1}^n \nu(b_iw_{\sigma(i)}).
\end{equation}
Write
 \begin{equation*}
 S(\sigma)=\sgn{\sigma}\prod_{i=1}^n s(b_iw_{\sigma(i)}).
 \end{equation*}
 We want to prove that $S(\sigma_1)=S(\sigma_2)$ for all $\sigma_1,\sigma_2$ such that $C(\sigma_1),C(\sigma_2)\neq \emptyset$, since then (\ref{determinant expansion}) gives
 $SZ$, where $S$ is the common value of all such $S(\sigma)$.
 
 Draw the two dimer coverings corresponding to $\sigma_1$ and $\sigma_2$ simultaneously on the graph. This results in a set of double edges and loops. We can change one
 dimer covering into another by moving every second dimer in a loop into the next edge in the loop. By a succession of such operations, we can change $C(\sigma_1)$ into
 $C(\sigma_2)$. Hence, it is enough to show that $S(\sigma_1)=S(\sigma_2)$ if $C(\sigma_1)$ and $C(\sigma_2)$ differ by just one loop. Assume that
$C(\sigma_1)$ and $C(\sigma_2)$ differ by just one loop of length $2k$. Then $\sigma_1=\sigma_2\tau$, where the permutation $\tau$ equals the identity except for a $k$-cycle.
Let $e_1,\dots,e_{2k}$ be the edges in the loop with $e_2,\dots, e_{2k}$ covered in $C(\sigma_1)$, and $e_1,\dots, e_{2k-1}$ covered in $C(\sigma_2)$. Then,
\begin{align*}
\frac{S(\sigma_2)}{S(\sigma_1)}&=\frac{\sgn(\sigma_2)}{\sgn(\sigma_1)}\prod_{i=1}^n\frac{s(b_iw_{\sigma_2(i)})}{s(b_iw_{\sigma_1(i)})}=\sgn(\tau)\frac{s(e_1)\dots s(e_{2k-1})}{s(e_2)\dots s(e_{2k})}\\
&=(-1)^{k+1}(-1)^{k+\ell+1}=(-1)^{\ell},
\end{align*}
where $\ell$  is the number of internal points in the loop. Since all points inside the loop are covered by dimers, $\ell$ must be even and we are done.
\end{proof}

Consider the dimer model with the probability of a dimer covering $C$ given by (\ref{dimerprob}). We want to be able to compute the probability that a given set
of dimers is covered by dimers. Let 
\begin{equation*}
 \mathbb{I}_C(e)=\begin{cases} 1 & \text{if $e\in C$}\\ 0 &\text{if $e\notin C$} \end{cases}
 \end{equation*}
for all $e\in E$ and $C\in\mathcal{C}$. Then,
\begin{equation}\label{Zformula}
Z=\sum_{C\in\mathcal{C}}\prod_{e\in E}\nu(e)^{\mathbb{I}_C(e)},
\end{equation}
which we regard as a function of all the edge weights $(\nu(e))_{e\in E}$.

Let $e_1,\dots,e_r$ be a given set of edges in $E$ and let $P(e_1,\dots,e_r)$ be the probability that they are all present in a dimer cover. Then
\begin{align}\label{probderformula}
P(e_1,\dots,e_r)&=\frac 1Z\sum_{C\in\mathcal{C}}\prod_{i=1}^r\mathbb{I}_C(e_i)\prod_{e\in C}\nu(e)
=\frac 1Z\sum_{C\in\mathcal{C}}\prod_{i=1}^r\mathbb{I}_C(e_i)\nu(e_i)\prod_{\substack{e\in C:e\neq e_i\\1\le i\le r} }\nu(e)\\
&=\frac 1Z\prod_{i=1}^r\nu(e_i)\frac{\partial^r Z}{\partial \nu(e_1)\dots\partial\nu(e_r)},\notag
\end{align}
by (\ref{Zformula}). From theorem \ref{thm:Kasteleyn} we know that $Z=S^{-1}\det \mathbb{K}$, where $S$ is independent of $\{\nu(e)\}_{e\in E}$, and thus
\begin{equation}\label{probZcomp}
P(e_1,\dots,e_r)=\prod_{i=1}^r\nu(e_i)\frac 1{\det \mathbb{K}}\frac{\partial^r }{\partial \nu(e_1)\dots\partial\nu(e_r)}\det \mathbb{K}.
\end{equation}
Let $e=b_{i_k}w_{j_k}$, $1\le k\le r$, with some enumeration of the black and white vertices, and let $I=\{i_k\}_{k=1}^r$ and $J=\{j_k\}_{k=1}^r$. For a matrix
$A=(a_{ij})_{1\le i,j\le n}$ and subsets $I,J$ of $\{1,\dots, n\}$, we write $A(I,J)=(a_{ij})_{i\in I, j\in J}$. Also, we write $I^c=\{1,\dots, n\}\setminus I$. Then
\begin{equation}\label{partialdetK}
\frac{\partial^r }{\partial \nu(e_1)\dots\partial\nu(e_r)}\det \mathbb{K}=(-1)^{\sum_{i\in I}i+\sum_{j\in J} j} \det \mathbb{K}(I^c,J^c)\prod_{i=1}^rs(e_i),
\end{equation}
by the Laplace expansion, since $\nu(b_{i_k}w_{j_k})$ occurs as $s(b_{i_k}w_{j_k})\nu(b_{i_k}w_{j_k})$ precisely in row $i_k$ and column $j_k$. Combining 
(\ref{probZcomp}) and (\ref{partialdetK}), we obtain
\begin{align*}
P(e_1,\dots,e_r)&=\prod_{k=1}^r \mathbb{K}(b_{i_k}, w_{j_k})(-1)^{\sum_{i\in I}i+\sum_{j\in J} j} \,\frac{\det \mathbb{K}(I^c,J^c)}{\det \mathbb{K}}\\
&=\prod_{k=1}^r \mathbb{K}(b_{i_k}, w_{j_k})\det\big(\mathbb{K}^{-1}(J,I)\big)
\end{align*}
by the formula for the minors of the inverse matrix.

Thus we have proved, \cite{MPS}, \cite{Ken97},
\begin{thma}\label{thm:Kenyonformula}
In the setting defined above the probability that a given set of edges $e_i=b_iw_i$, $1\le i\le r$, are included in a dimer covering is given by
\begin{equation}\label{Kenyonformula}
P(e_1,\dots,e_r)=\prod_{i=1}^r \mathbb{K}(b_i,w_i)\det\big(\mathbb{K}^{-1}(w_i,b_j)\big)_{1\le i,j\le r}.
\end{equation}
Thus, the dimers form a determinantal point process with correlation kernel
\begin{equation}\label{Kastdetkernel}
L(e_i,e_j)=\mathbb{K}(b_i,w_i)\mathbb{K}^{-1}(w_i,b_j),
\end{equation}
if $e_i=b_iw_i$, $w_i\in \mathbf{W}$, $b_i\in \mathbf{B}$.
\end{thma}

Computing $\mathbb{K}^{-1}$ for large graphs is in general a very hard problem. Typically, we are interested in computing $\mathbb{K}^{-1}$ for a growing sequence of graphs and
we want to have a formula for $\mathbb{K}^{-1}$ that is suitable for asymptotic analysis. In some special models it is possible to write a double contour integral formula for the
inverse Kasteleyn matrix. We will not give such formulas in this paper, except for the two-periodic Aztec diamond, see (\ref{KinverseTPAz}). However, we will mention a few cases where 
such formulas are available and give references.
 If we have a dimer model on a torus then we can use Fourier analysis to compute $\mathbb{K}^{-1}$, but in this case theorem \ref{thm:Kasteleyn} is no
longer valid since we assumed that $\mathcal{G}$ is a planar graph. However, the theorem can be extended to the case of the torus and more generally to
graphs on higher genus surfaces, see \cite{CiRe}, \cite{Ken}.

\begin{remark} \label{rem:Kasteleyn}
{\rm The Kasteleyn method was originally developed for non-bipartite graphs, \cite{Kas}, \cite{CiRe}. The partition function is then given by the Pfaffian of an 
appropriately defined antisymmetric matrix, the {\it Kasteleyn matrix}. The elements are again weights of edges with a choice of signs. These signs are specified by choosing
an orientation of the edges in the graph with certain properties, a {\it Kasteleyn} or {\it Pfaffian orientation}, \cite{CiRe}, \cite{FeTi}, \cite{Kas}.}
\end{remark}

\subsection{The Aztec diamond} \label{sec:Aztec}

\begin{figure}
\begin{center}
\includegraphics[height=3cm,angle=90]{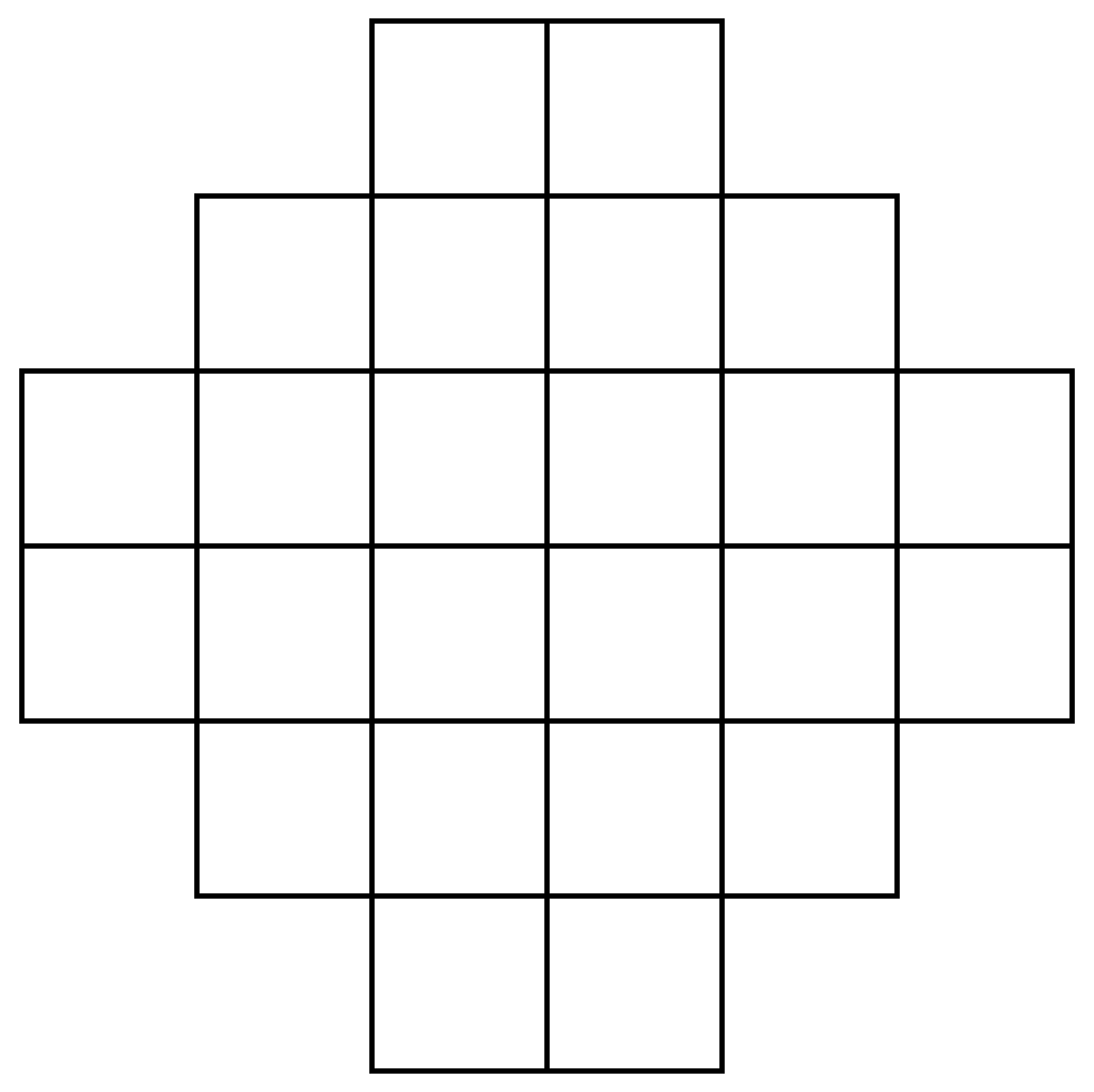} \hspace{10mm}
			\includegraphics[height=3cm,angle=90]{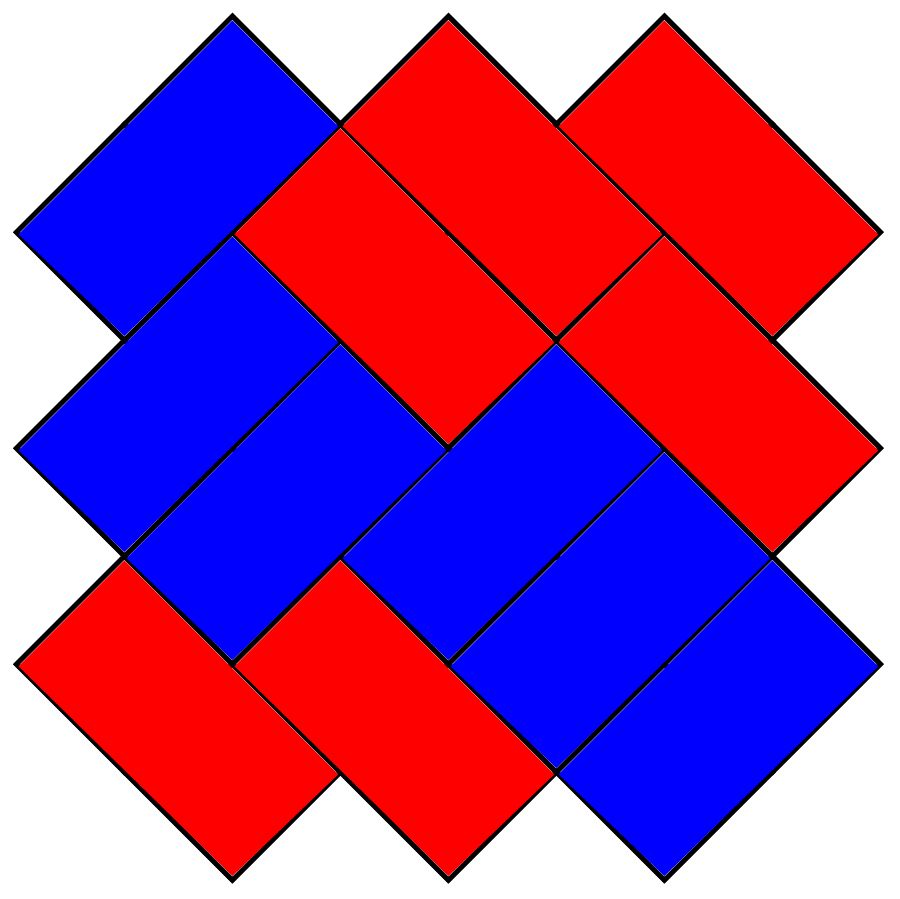} 
\caption{On the left we see a size 3 Aztec diamond shape. On the right we have a domino covering of this shape rotated by 45 degrees. This is a domino tiling of the Aztec diamond.}
\label{fig:Aztecdef}
\end{center}
\end{figure}

\begin{figure}
\begin{center}
\includegraphics[height=3cm,angle=90] {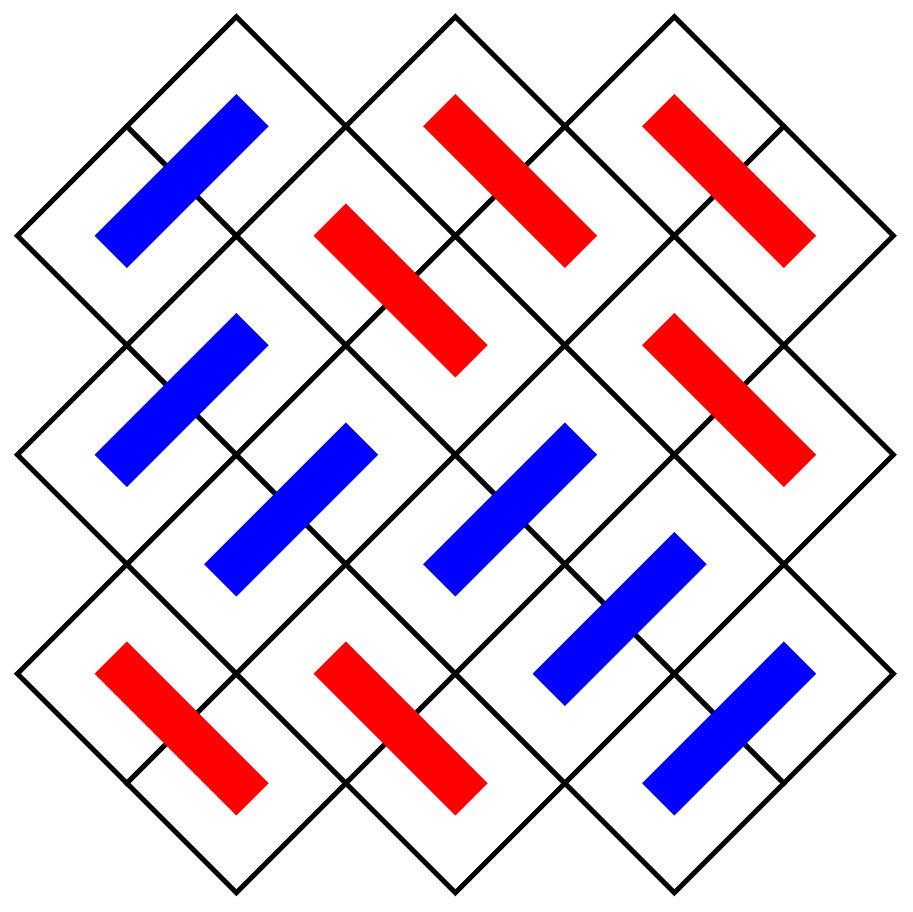}\hspace{10mm}
\includegraphics[height=3cm,angle=90]{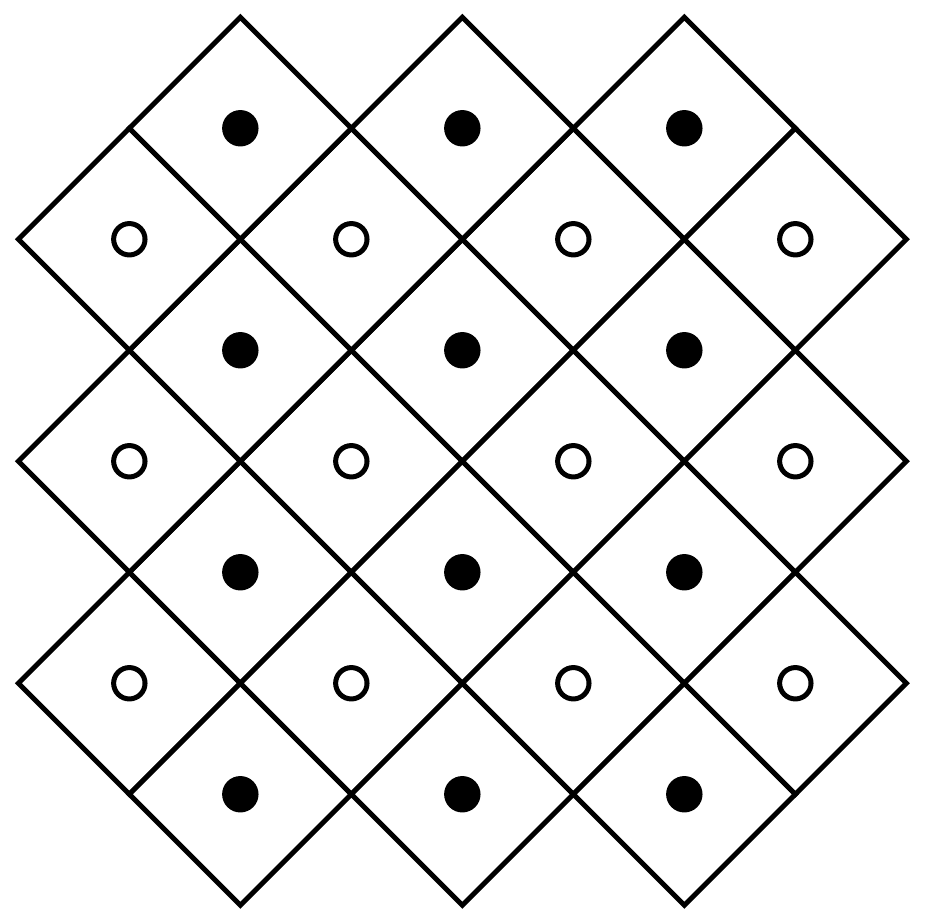}
\caption{To the left we have the dimer cover corresponding to the domino tiling in figure \ref{fig:Aztecdef}. To the right is the underlying Aztec diamond graph for the corresponding dimer model. }
\label{fig:Aztecdimer}
\end{center}
\end{figure}

The \emph{Aztec diamond shape} of size $n$ is the union of all unit lattice squares $[k,k+1]\times [\ell,\ell+1]$ inside the region $\{(x,y)\,;\,|x|+|y|\le n+1\}$. This region can be tiled by dominoes, i.e. size $1\times 2$ or $2\times 1$ rectangles, see figure \ref{fig:Aztecdef}. We are interested in random domino tilings of the Aztec diamond shape and we will refer to
this as the Aztec diamond (model). It was introduced studied systematically in \cite{EKLP1}, where it was shown that there are $2^{n(n+1)/2}$ possible domino tilings of an Aztec
diamond shape of size $n$. We can also think of this model as a dimer model on a certain graph that we call the \emph{Aztec diamond graph}, see figure \ref{fig:Aztecdimer}.
We can introduce weights by giving weight $1$ to horizontal dominoes, the red dominoes in figure \ref{fig:Aztecdef}, and weight $a$ to the vertical dominoes, the blue dominoes
in figure \ref{fig:Aztecdef}. This translates directly into a weighting of the Aztec diamond graph in the dimer model version. We can choose a Kasteleyn sign by putting the weight $1$ on horizontal edges, those in the direction southwest to northeast in figure \ref{fig:Aztecdimer}, and weight $ai$ on the vertical edges, those in the direction southeast to northwest
in figure \ref{fig:Aztecdimer}. There is a black and white checkerboard colouring of the unit squares in the Aztec diamond shape, which corresponds to the black and white colouring
of the vertices in the bipartite Aztec diamond graph, see the right picture in figure \ref{fig:Aztecdimer}. The four colours in random domino tilings that we see in figure \ref{fig:aztecsim}
correspond to the vertical and horizontal dominoes and the fact that a vertical (and horizontal) domino can cover in two different ways. A vertical domino can have black square at the top or at the bottom.

\subsection{The height function for the Aztec diamond} \label{sec:3.3}
To each tiling of the Aztec diamond we can associate a height function. The heights sit on the faces of the corresponding Aztec diamond graph. If we have a size $n$ Aztec diamond we consider the faces of the size $n+1$ Aztec diamond graph, see figure \ref{fig:Aztecheightdef}. We fix (arbitrarily) the height to be $0$ at the lower left face. The \emph{heights} are then fixed by the 
following rules for height differences between two adjacent faces. Recall that we have a black and white colouring of the vertices. The height difference is

\begin{enumerate}  
\item $+3\, (-3)$ if we cross a dimer with a white vertex to the right (left),
\item $+1\, (-1)$ if we do not cross a dimer and have a white vertex to the left (right).
\end{enumerate}

Note that the height on the outer faces is fixed, independent of the particular tiling. It is not hard to see that this definition is consistent if we go around a vertex. In fact, this definition
is a special case of a more general definition of a height function for dimer models on bipartite graphs, see \cite{FeTi}, \cite{Ken}.

In figure \ref{fig:Aztecheightsim} we see the height function of a random tiling of a uniform random tiling of the Aztec diamond. Note that the height above the frozen (solid) regions is completely flat, whereas it is rough above the liquid region. We can imagine from the picture that if we take a macroscopic limit we should have a limiting, non-random height function which is 
smooth except at the boundary between the liquid and frozen regions. This is indeed the case and can be proved in greater generality. This limiting function
is called the limit shape and is the solution of a certain variational problem, \cite{CKP}, \cite{KeOk1},\cite{OkLS}.

\begin{figure}
\begin{center}
\includegraphics[height=2.5in]{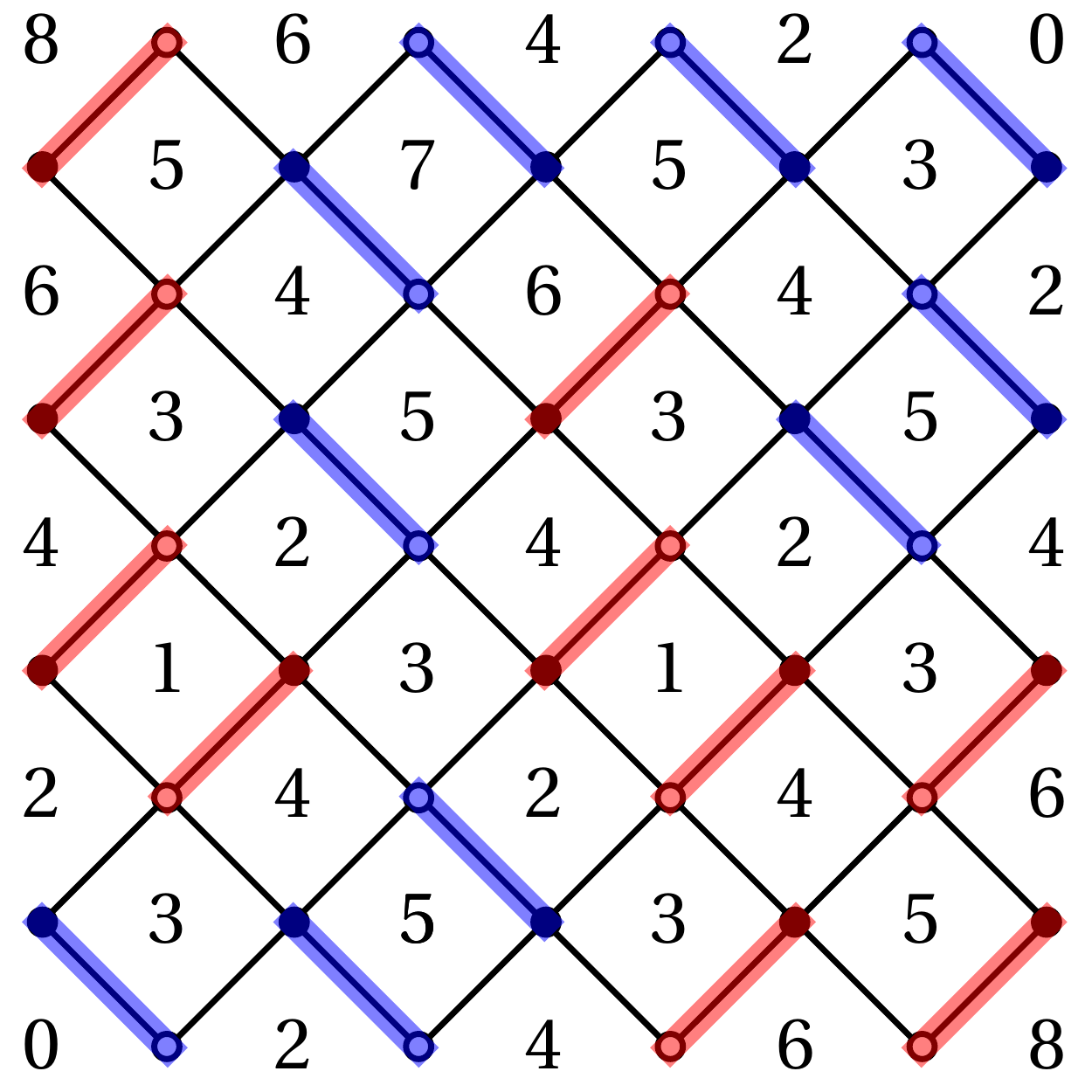}
\caption{The height function of a random tiling of an Aztec diamond.}
\label{fig:Aztecheightdef}
\end{center}
\end{figure}

\begin{figure}
\begin{center}
\includegraphics[height=3.5in]{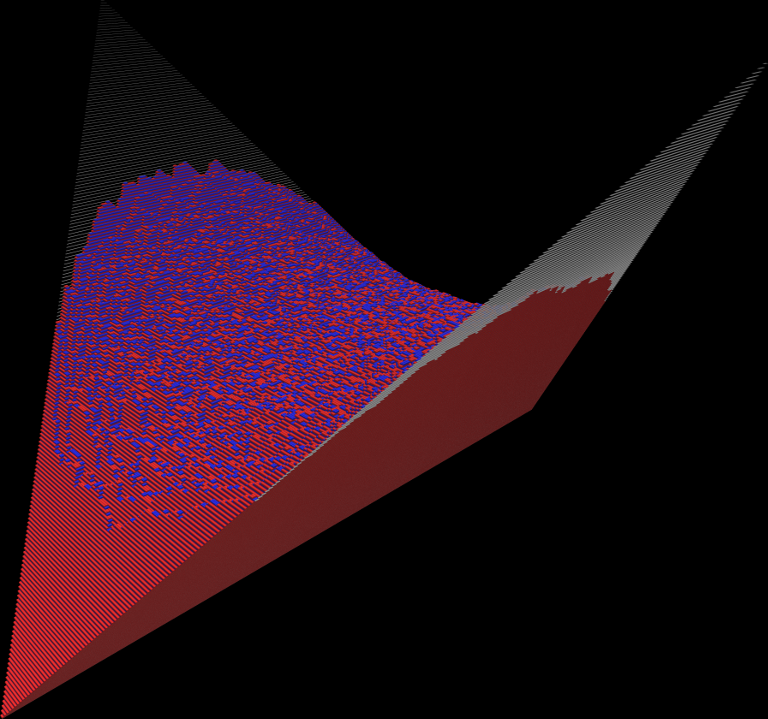}
\caption{The height function for a random tiling of a uniform Aztec diamond. (Picture by B. Young)}
\label{fig:Aztecheightsim}
\end{center}
\end{figure}

\section{Non-intersecting paths} \label{sec:4}

\subsection{The Lindstr\"om-Gessel-Viennot theorem} \label{sec:4.1}

We will consider probability measures associated with non-intersecting paths in a directed graph such as that in figure \ref{fig:Aztecdirectedgraph}, which is related to the
Aztec diamond.

Let $\mathcal{G}=(V,E)$ be a directed, acyclic graph with no multiple edges. For two vertices $u$ and $v$ the set of all directed paths $\pi$ in $\mathcal{G}$ from 
$u$ to $v$ is denoted by $\Pi(u,v)$. The set of all families of paths $(\pi_1,\dots,\pi_n)$, where $\pi_u$ goes from vertex $u_i$ to vertex $v_i$, $1\le i\le n$, is denoted by
$\Pi(\mathbf{u},\mathbf{v})$, with $\mathbf{u}=(u_1,\dots,u_n)$, $\mathbf{v}=(v_1,\dots,_n)$. We say that two directed paths {\it intersect} if they share a common vertex.
The families of paths in $\Pi(\mathbf{u},\mathbf{v})$ which do not have any intersections is denoted by $\Pi_{\text{n.i.}}(\mathbf{u},\mathbf{v})$.

A {\it weight function} $w$ is a map $w:E\to\mathbb{C}$. We extend the weight function multiplicatively to a weight function on paths by
\begin{equation*}
w(\pi)=\prod_{e\in\pi} w(e),
\end {equation*}
and to a family of paths $\pi_1,\dots,\pi_n$, by
\begin{equation*}
w(\pi_1,\dots,\pi_n)=\prod_{j=1}^n w(\pi_j).
\end {equation*}
Let $F$ be a set of families of paths. Then, the weight of $F$ is defined by
\begin{equation*}
W(F)=\sum_{(\pi_1,\dots,\pi_n)\in F}w(\pi_1,\dots,\pi_n).
\end {equation*}
For any $u,v\in V$, we set
\begin{equation}\label{niptransitionweight}
p(u,v)=W(\Pi(u,v))=\sum_{\pi\in\Pi(u,v)}w(\pi).
\end{equation}
We can think of $p(u,v)$ as the {\it transition weight} from vertex $u$ to vetex $v$.

We say that $\mathbf{u}$ and $\mathbf{v}$ are $\mathcal{G}$-{\it compatible} if whenever $\alpha<\alpha'$ and $\beta>\beta'$ every path from $u_{\alpha}$ to $v_{\beta}$ 
intersects every path from $u_{\alpha'}$ to $v_{\beta'}$. Typically the graph $\mathcal{G}$ is a plane graph and it is easy to check this condition.

\begin{thma}\label{thm:LGV} (Lindstr\"om-Gessel-Viennot). Let $\mathcal{G}$ be a directed, acyclic graph and $\mathbf{u}=(u_1,\dots,u_n)$, $\mathbf{v}=(v_1,\dots,_n)$
two $n$-tuples of verices which are $\mathcal{G}$-compatible. Then,
\begin{equation}\label{LGVdeterminant}
W(\Pi_{\text{n.i.}}(\mathbf{u},\mathbf{v}))=\det\big(p(u_i,v_j)\big)_{1\le i,j\le n}.
\end{equation}
\end{thma}
For a proof see e.g. \cite{JoHouch} or \cite{Ste}.

\begin{remark} \label{rem:Dyson}
{\rm Instead of non-intersecting paths in a directed, acyclic graph we can consider non-colliding stochastic processes. This goes back to the work of Karlin and McGregor, \cite{KaMc}. Under suitable assumptions it is possible to prove the analogue of theorem \ref{thm:LGV} in this case and it is called the Karlin-McGregor formula. In
particular we can apply this formula to study \emph{non-colliding Brownian motions}. We start independent Brownian motions $B_j(t)$, $1\le j\le n$ at $x_1,\dots, x_n$ at time $t=0$,
and condition them never to collide for all $t>0$. This process has the same distribution as the eigenvalue process in \emph{Hermitian Dyson Brownian motion}, which is defined as follows. Let $H(t)$ be an $n\times n$ Hermitian matrix whose elements $H_{ij}(t)$, $1\le i\le j\le n$ are independent Brownian motions, real for $i=j$ and complex for $i<j$. If we take 
$H_{ij}(0)=x_j$ and project to the eigenvalue process we get a process with the same distribution as the non-colliding Brownian motions. In particular, if $x_1=\dots=x_n=0$, at each time the points are distributed as the eigenvalues of a rescaled $n\times n$ GUE matrix, see (\ref{GUE}). This indicates a strong connection between properties of non-intersecting paths and
random matrix theory. For example, the top path is analogous to the path of the largest eigenvalue.
}
\end{remark}

\subsection{Non-interesecting paths and the Aztec diamond} \label{sec4.2}
A random domino tiling of an Aztec diamond can be described in terms of non-intersecting paths and associated particles which form a determinantal point process. We describe this connection and how it leads to a formula for the correlation kernel. 

\subsubsection{Particle description of the Aztec diamond}\label{sec4.2.1}
When applying theorem \ref{thm:LGV} we often get probability measures given by products of determinants such as those discussed in section \ref{sec:2.2}. As an example,
we will consider the Aztec diamond of size $n$. The directed graph that we will use is as in figure \ref{fig:Aztecdirectedgraph}.

\begin{figure}
\begin{center}
\includegraphics[height=2.5in]{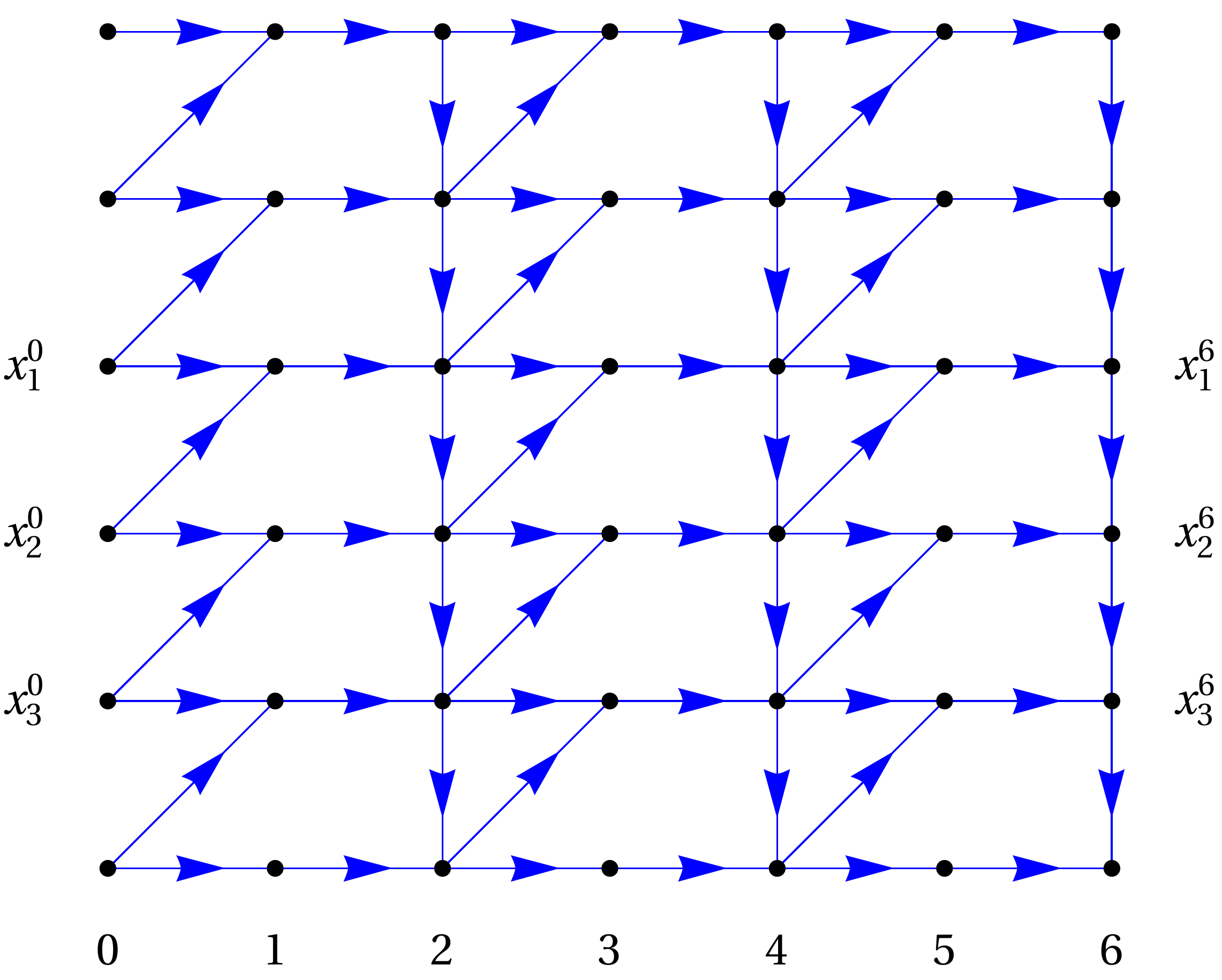}
\caption{The directed graph for the Aztec diamond}
\label{fig:Aztecdirectedgraph}
\end{center}
\end{figure}

The vertex set is $V=\{0,\dots,2n\}\times\mathbb{Z}$, and we have directed edges from $(2j, k)$ to $(2j+1,k)$, $(2j, k)$ to $(2j+1,k+1)$, $0\le j<n-1$, from 
$(2j, k)$ to $(2j,k-1)$, $0<j\le n$, and from $(2j, k)$ to $(2j-1,k)$, $1\le j \le n$. We have weight $a>0$ on edges from $(2j, k)$ to $(2j+1,k+1)$ and from
$(2j, k)$ to $(2j,k-1)$. All other edges have weight 1. Consider non-intersecting paths from $u_i=(0,x_i^0)$ to $v_i=x_i^{N}$, where $x_i^0=x_i^{N}=1-i$,
$1\le i\le M$, where $N=2n$ and $M\ge n$. A path from $u_i$ to $v_i$ intersects the line $\{r\}\times\mathbb{Z}$ at a point $(r,x_i^r)$. A family of non-intersecting paths 
$\pi_i$ from $u_i$ to $v_i$, $1\le i\le M$ has the probability
\begin{equation*}
\mathbb{P}[\Pi_{\text{n.i.}}(\mathbf{u},\mathbf{v})]=\frac 1{Z_{N,M}}w(\pi_1,\dots,\pi_M).
\end{equation*}
By theorem \ref{thm:LGV} the normalization constant is given by
\begin{equation*}
Z_{N,M}=\det\big(p(u_i,v_j)\big)_{1\le i,j\le M}.
\end{equation*}
Under this probability measure, the points $(r,x_i^r)$, $1\le r<N$, $1\le i\le M$, form a point process. We see that there is a bijective correspondence between the
non-intersecting paths and the particle configuration. The transition functions $p_{r,r+1}(x,y)$ from vertices on line $r$ to line $r+1$ are given by
\begin{align}\label{Aztectransition}
p_{2i,2i+1}(x,y)&=a\delta_{y-x,1}+\delta_{y-x,0}\\
p_{2i-1,2i}(x,y)&=a^{-(y-x)}\mathbb{I}_{y-x\le 0}.
\notag
\end{align}
It follows from theorem \ref{thm:LGV} that the probability of a particular configuration of the point process is given by
\begin{equation}\label{Aztecproductmeasure}
p_{N,M}(\mathbf{x})=\frac 1{Z_{N,M}}\prod_{r=0}^{N-1}\det\big(p_{r,r+1}(x_j^r,x_k^{r+1})\big)_{1\le j,k\le M},
\end{equation}
where $\mathbf{x}=(x^1,\dots,x^{N-1})$, $x^r=(x_1^r,\dots,x_M^r)$.

This model is in bijection with random tilings of the size $n$ Aztec diamond. We will not explain the mapping in detail but refer to figures \ref{fig:Aztecpathsparticles1} and
 \ref{fig:Aztecpathsparticles2}. See \cite{JoAz} for details.
The particles describing the non-intersecting paths can be thought of as particles sitting on the dominoes in the tiling, see the right figure in 
figure \ref{fig:Aztecpathsparticles2}.

\begin{figure}
\begin{center}
\includegraphics[height=2in]{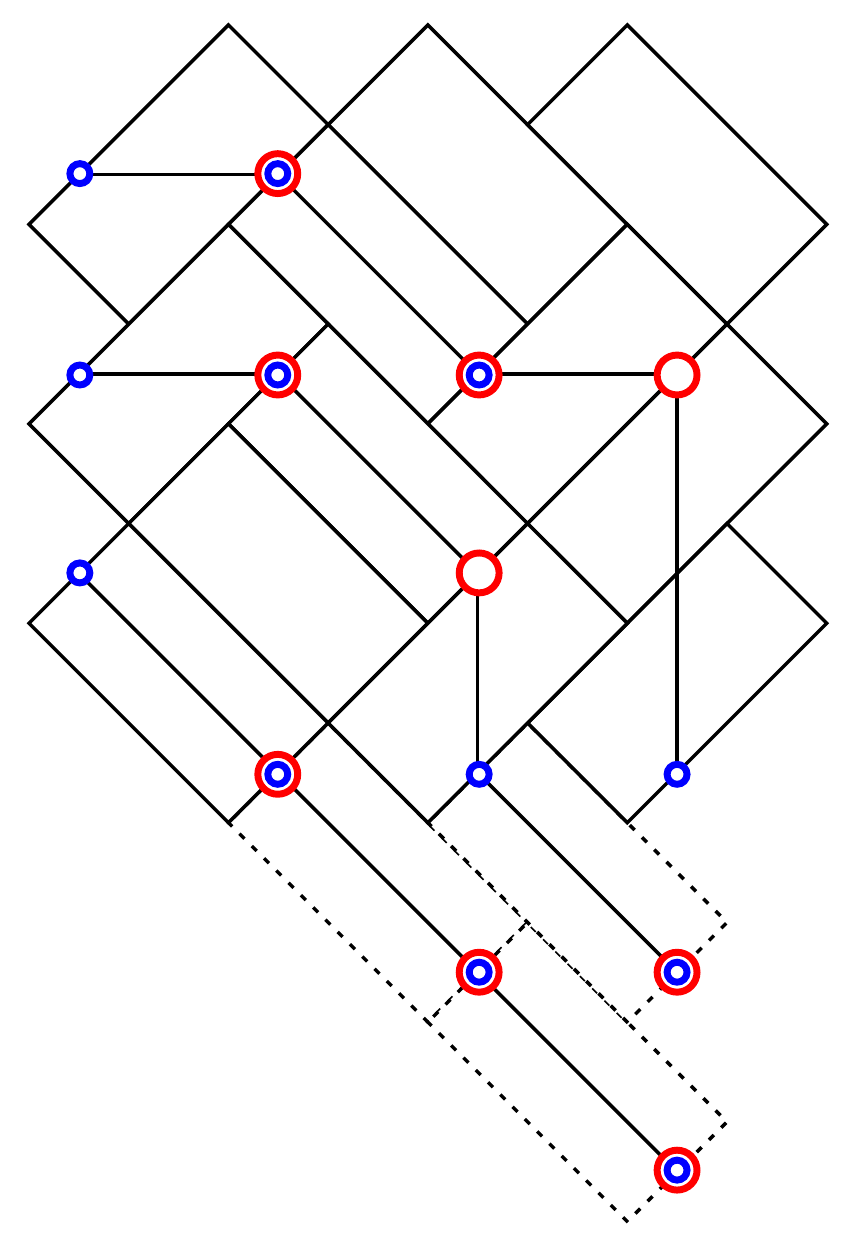}
\hspace{2mm}
\includegraphics[height=2in]{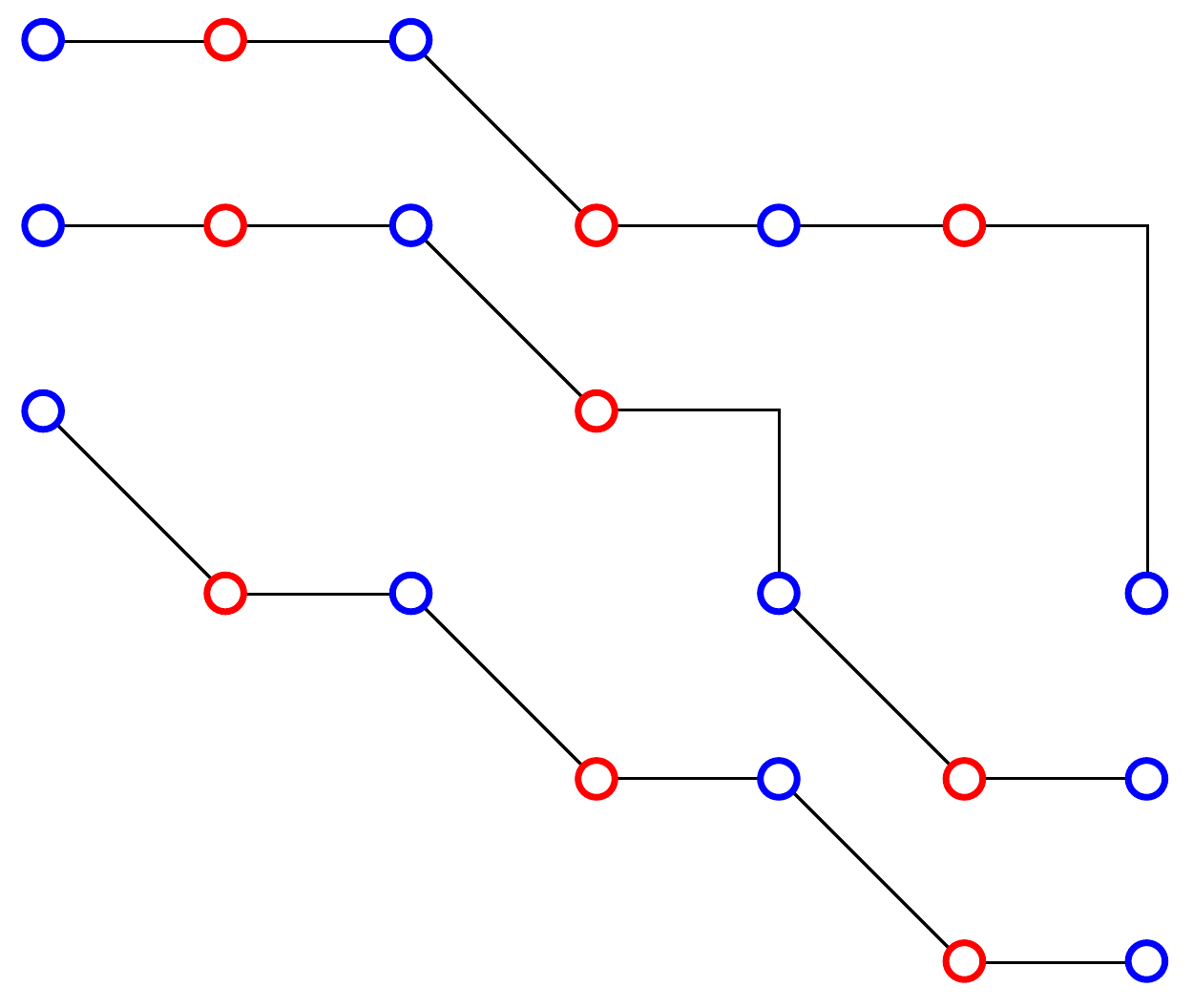}
\caption{To the left we see the non-intersecting paths describing a tiling of the Aztec diamond. We can think of the paths as going from a blue to a red point. These points can sit
on top of each other as in the picture. To the right we have rotated the picture by 45 degrees and separated the blue and red particles, which leads to a new set of non-intersecting paths. If a red particle sits on top of a blue particle the path from the red to the blue particle is a horizontal line segment.}
\label{fig:Aztecpathsparticles1}
\end{center}
\end{figure}

\begin{figure}
\begin{center}
\includegraphics[height=2in]{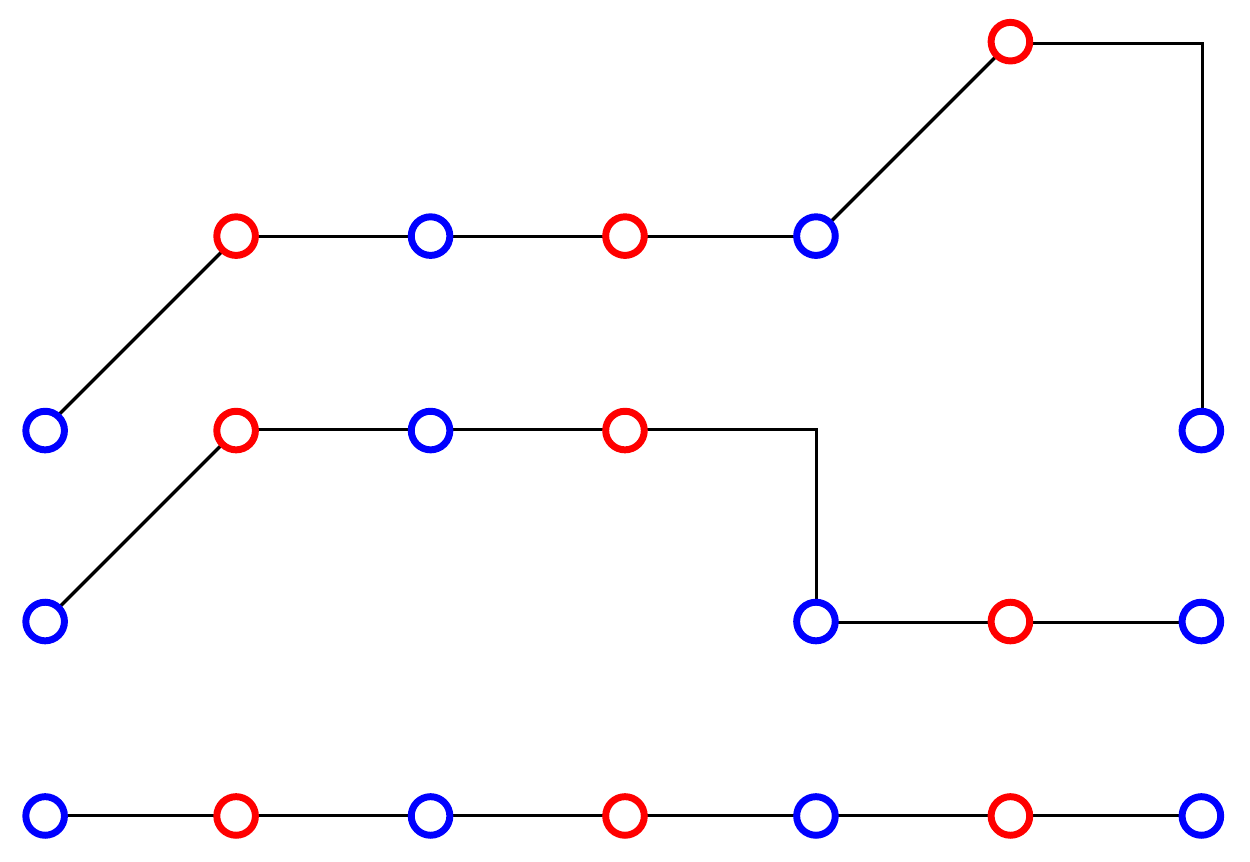}
\hspace{2mm}
\includegraphics[height=2in]{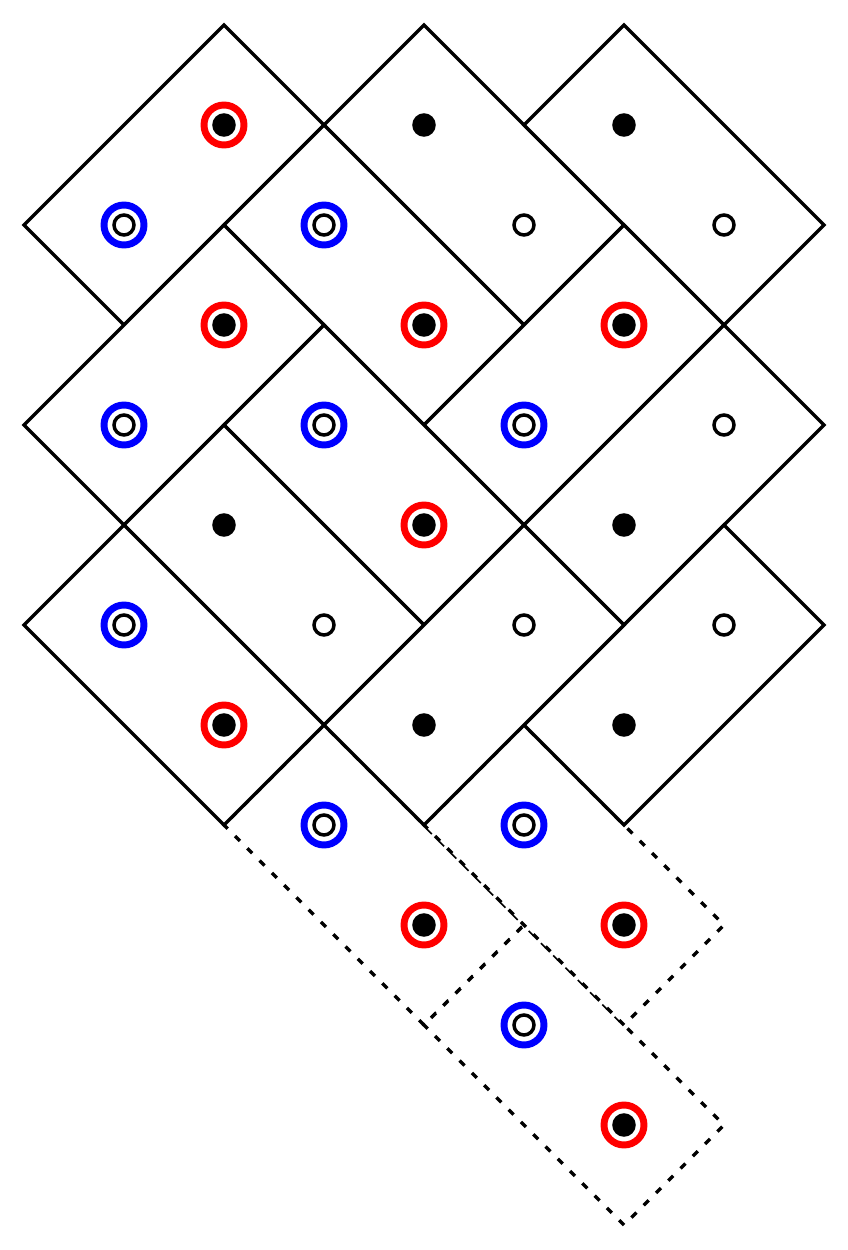}
\caption{To the left the paths to the right in figure \ref{fig:Aztecdirectedgraph} have been modified by shifting the steps between even and odd columns up by one. The first column
is column zero and the last column $2n$, where $n$ is the size of the Aztec diamond. In the right picture we have moved the red particles half a unit to the left, and the blue
particles half a unit to the right in the left figure in \ref{fig:Aztecdirectedgraph}. In this way the particles sit on the dominos, or on the vertices in the corresponding dimer model.}
\label{fig:Aztecpathsparticles2}
\end{center}
\end{figure}

This gives us a particle picture of a domino tiling of the Aztec diamond. It follows from theorem \ref{productmeasurethm} that this particle process is a determinantal point process with correlation kernel given by (\ref{generalK}). If we let
\begin{equation}\label{phiAztec}
\phi_{2i}(z)=az+1,\quad \phi_{2i-1}(z)=\frac 1{1-a/z},
\end{equation}
then $p_{r,r+1}(x,y)=\hat{\phi}_r(y-x)$. In this case, all paths from $(0,x_i^0)$ to $(N,x_i^N)$ for $i>n$ will be frozen, the same for all possible tilings, as can be seen by examining paths going from left to right in fig \ref{fig:Aztecdirectedgraph}. Thus, this example fits perfectly into the infinite Toeplitz case discussed in section \ref{sec:2.2.2}. It is straightforward to check
that
\begin{equation}\label{phiAztecWH}
\phi_{2r+\epsilon_1,2s+\epsilon_2}(z)=\phi_{2r+\epsilon_1,2s+\epsilon_2}^+(z)\phi_{2r+\epsilon_1,2s+\epsilon_2}^-(z),
\end{equation}
where
\begin{equation}\label{phiAzpm}
\phi_{2r+\epsilon_1,2s+\epsilon_2}^+(z)=(1+az)^{s-r+\epsilon_2-\epsilon_1}\,,\,\,
\phi_{2r+\epsilon_1,2s+\epsilon_2}^-(z)=\frac 1{(1-a/z)^{s-r}},
\end{equation}
for $\epsilon_1,\epsilon_2\in\{0,1\}$ and $2r+\epsilon_1<2s+\epsilon_2$. It follows from (\ref{KLRcontour}) that the particle process defined above has a correlation kernel
(note that $L=0,R=N=2n$),
\begin{align}\label{OneAzteckernel}
K_n^{\text{OneAz}}(2r+\epsilon_1,u;2s+\epsilon_2,v)&=-\frac{\mathbb{I}_{2r+\epsilon_1<2s+\epsilon_2}}{2\pi \mathrm{i}}\int_{\gamma_1}z^{u-v}
\frac{(1+az)^{s-r+\epsilon_2-\epsilon_1}}{(1-a/z)^{s-r}}\frac{dz}{z}\\
&+\frac 1{(2\pi \mathrm{i})^2}\int_{\gamma_{\rho_1}}dz\int_{\gamma_{\rho_2}}dw\frac{z^{u-1}}{w^v(w-z)}\frac{(1-a/w)^{n-s}(1+aw)^{s+\epsilon_2}}
{(1-a/z)^{n-r}(1+az)^{r+\epsilon_1}},\notag
\end{align}
where $a<\rho_1<\rho_2<1/a$. If we make the change of variables $z\to -1/z$, $w\to -1/w$ a computation gives the following alternative formula
\begin{align}\label{OneAzteckernel2}
(-1)^{u-v}K_n^{\text{OneAz}}(2r+\epsilon_1,u;2s+\epsilon_2,v)&=-\frac{\mathbb{I}_{2r+\epsilon_1<2s+\epsilon_2}}{2\pi \mathrm{i}}\int_{\gamma_1}z^{v-u}
\frac{(1+az)^{r-s}}{(1-a/z)^{r-s+\epsilon_1-\epsilon_2}}\frac{dz}{z}\\
&+\frac 1{(2\pi \mathrm{i})^2}\int_{\gamma_{\rho_1}}dz\int_{\gamma_{\rho_3}}dw\frac{w^{v-1}}{z^u(z-w)}\frac{(1+aw)^{n-s}(1-a/w)^{s+\epsilon_2}}
{(1+az)^{n-r}(1-a/z)^{r+\epsilon_1}},\notag
\end{align}
where $a<\rho_3<\rho_1<1/a$.

\begin{remark}{\rm It is possible to write a formula for the inverse Kasteleyn matrix for the Aztec diamond using $K_N^{\text{OneAz}}$. This was first done in the case $a=1$ in \cite{Helf}. The probability of seeing particles at certain position in the Aztec diamond, as in figure \ref{fig:Aztecpathsparticles2}, can be expressed both in terms of $K^{\text{OneAz}}_n$
and in terms of the inverse Kasteleyn matrix. Here one uses the fact that we have a particle at the white vertex $w$ if and only if the edge from $w$ to $w+(1,1)$ or $w+(1,-1)$ is covered by a dimer. Using this, it is possible to guess the inverse Kasteleyn matrix from $K^{\text{OneAz}}_n$, see \cite{CJY}. A more systematic approach is given in 
\cite{CY}. By associating a dimer model on so called rail-yard graphs to the Schur process, defined below, it is also possible to compute the inverse Kasteleyn matrix for the Aztec diamond, see \cite{BBCCR}.
}
\end{remark}
\begin{remark} {\rm The kernel (\ref{OneAzteckernel}) can also be expressed in terms of Krawtchouk polynomials, and a special case gives the Krawtchouk
kernel for the Krawtchouk ensemble a discrete orthogonal polynomial ensemble with binomial weight, see \cite{JoAz}.}
\end{remark}

\subsection{The Schur process} \label{sec4.3}

For basic facts about partitions and Schur polynomials see the Appendix.
It is not hard to see that the measure in (\ref{Aztecproductmeasure}) can be expressed in terms of Schur polynomials. Note that we can write the transition functions in
(\ref{Aztectransition})
\begin{equation*}
p_{2i,2i+1}(x,y)=e_{y-x}(a),\quad p_{2i-1,2i}(x,y)=h_{x-y}(a).
\end{equation*}
Thus , the measure $p_{N,M}$ in (\ref{Aztecproductmeasure}) becomes ($N=2n$)
\begin{equation}\label{AztecSchur}
p_{N,M}(\mathbf{x})=\frac 1{Z_{N,M}}\prod_{r=0}^{n-1}\det\big(e_{x_k^{2r+1}-x_j^{2r}}(a)\big)_{1\le j,k\le M}\det\big(h_{x_k^{2r+1}-x_j^{2r+2}}(a)\big)_{1\le j,k\le M},
\end{equation}
for any $M\ge n$. We can define partitions $\lambda_j^{(r)}=x_j^{2r-1}+j-1$, $\mu_j^{(r)}=x_j^{2r}+j-1$. From the nature of the paths it follows that 
$\lambda^{(r)}=(\lambda_1^{(r)},\lambda_2^{(r)},\dots )$ and $\mu{^(r)}=(\mu_1^{(r)},\mu_2^{(r)},\dots )$ are indeed partitions such that $\lambda^{(r)}_1\le n$, 
$\ell(\lambda^{(r)})\le n$, and the same for $\mu^{(r)}$. Also, the non-intersection condition gives the interlacing structure $\mu^{(r)}\prec \lambda^{(r+1)}$ and ${\mu^{(r)}}'\prec {\lambda^{(r+1)}}'$.
Thus, on the sequence of partitions satisfying
\begin{equation}\label{Aztecpartitionorder}
\emptyset=\mu^{(0)}\prec\lambda^{(1)}\succ\mu^{(1)}\prec\lambda^{(2)}\succ\dots\prec\lambda^{(n)}\succ\mu^{(n)}=\emptyset,
\end{equation}
by the Jacobi-Trudi identities (\ref{JacobiTrudi1}) and (\ref{JacobiTrudi2}), we obtain the probability measure
\begin{equation}\label{Aztecpartitionmeasure}
p(\lambda^{(1)}, \mu^{(1)},\dots,\mu^{(n-1)},\lambda^{(n)})=\frac 1{Z}\prod_{r=0}^{n-1}s_{{\lambda^{(r+1)}}'/{\mu^{(r)}}'}(a)s_{\lambda^{(r+1)}/\mu^{(r+1)}}(a).
\end{equation}
This is a special case of the Schur process that we next define more generally.

Let $\mathbf{a}^k=(a_1^k,\dots,a_{m}^k)$, $L\le k<R$, with $0<a_j^k<1$ be given together with two sequences
\begin{equation}\label{cdseq}
c_k\in\{h,e\},\quad d_k\in\{-1,1\}, \,\,L\le k<R.
\end{equation}
Introduce the notation
\begin{align}\label{Shefunctions}
S_{h,1}(\lambda,\mu;\mathbf{a})=s_{\mu/\lambda}(\mathbf{a})&,\quad S_{h,-1}(\lambda,\mu;\mathbf{a})=s_{\lambda/\mu}(\mathbf{a}),\\
S_{e,1}(\lambda,\mu;\mathbf{a})=s_{\mu'/\lambda'}(\mathbf{a})&,\quad S_{e,-1}(\lambda,\mu;\mathbf{a})=s_{\lambda'/\mu'}(\mathbf{a}).\notag
\end{align}
The {\it Schur process}, \cite{OkRe1}, \cite{BoGo}, is the probability measure on the sequence of partitions $\{\lambda^{(k)}\}_{L<k<R}$ defined by
\begin{equation}\label{Schurprocess}
p(\lambda^{(L+1)},\dots,\lambda^{(R-1)})=\frac 1Z\prod_{k=L}^{R-1}S_{c_k,d_k}(\lambda^{(k)},\lambda^{(k+1)};\mathbf{a}^k),
\end{equation}
where $\lambda^{(L)}=\lambda^{(R)}=\emptyset$.
We can think of $k$ in (\ref{Schurprocess}) as a "time parameter" and this leads to thinking of the probability measure as a process with certain transition
functions. Set $x_j^k=\lambda_j^{(k)}-j+d$, for some choice of $d\in\mathbb{Z}$, and define
\begin{equation}\label{phiSchur}
\phi_{h,\pm 1,\mathbf{a}}(z)=\prod_{i=1}^m\frac 1{1-a_iz^{\pm 1}},\quad \phi_{e,\pm 1,\mathbf{a}}(z)=\prod_{i=1}^m(1+a_iz^{\pm 1}),
\end{equation}
We see that 
\begin{equation*}
S_{c_k,d_k}(\lambda^{(k)},\lambda^{(k+1)};\mathbf{a}^k)=\det\big(\hat{\phi}_{c_k,d_k,\mathbf{a}^k}(x_h^{k+1}-x_j^k)\big)_{1\le i,j\le M},
\end{equation*}
if $M\ge \max(\ell(\lambda^{(k)}),\ell(\lambda^{(k+1)}),\lambda^{(k)}_1,\lambda^{(k+1)}_1)$ for all $k$.   Thus, the Schur process can be written
\begin{equation}\label{Schurprocessphi}
p(\lambda^{(L+1)},\dots,\lambda^{(R-1)})=\frac 1Z\prod_{k=L}^{R-1}\det\big(\hat{\phi}_{c_k,d_k,\mathbf{a}^k}(x_h^{k+1}-x_j^k)\big)_{1\le i,j\le M}
\end{equation}
as a measure on the particle configuration $x_j^k$ instead. Since we can let $M\to\infty$ we see that this fits into the framework of section \ref{sec:2.2.2} and hence
we get a determinantal point process with kernel given by (\ref{KLRcontour}) with
\begin{equation}\label{phipmSchur}
\phi_{r,s}^{+}(z)=\prod_{\{k\,;\,r\le k<s,d_k=+1\}}\phi_{c_k,+1,\mathbf{a}^k}(z),\quad
\phi_{r,s}^{-}(z)=\prod_{\{k\,;\,r\le k<s,d_k=-1\}}\phi_{c_k,-1,\mathbf{a}^k}(z).
\end{equation}
We see that the measure (\ref{AztecSchur}) that we got from the particle version of the Aztec diamond is a special case of the Schur process with
$L=0$, $R=2n$, $\mathbf{a}^k=a\in\mathbb{R}$ for all $k$, and with $c_{2r}=e$, $c_{2r-1}=h$, $d_{2r}=1$ and $d_{2r-1}=-1$.

It is an interesting and important fact for Schur processes that we can compute the partition function, or normalization constant, $Z$ in (\ref{Schurprocess}) as a function of the parameters in the model. By the arguments
in section (\ref{sec:2.2.2}) the partition function $Z$ is given by a Toeplitz determinant with symbol
\begin{equation*}
f(z)=\prod_{k=L}^{R-1}\phi_{c_k,d_k,\mathbf{a}^k}(z).
\end{equation*}
This is a Toeplitz determinant with a rational symbol and there is an exact formula for it if $M$ is large enough, or we can let $M\to\infty$ and apply the strong Szeg\H{o}
limit theorem. In symmetric function theory this relates to Cauchy type identities for Schur polynomials. Let
\begin{align*}
I_1=\{k\,;\,L\le k<R,(c_k,d_k)&=(h,1)\},\quad I_2=\{k\,;\,L\le k<R,(c_k,d_k)=(e,1)\}\\
I_3=\{k\,;\,L\le k<R,(c_k,d_k)&=(h,-1)\},\quad I_4=\{k\,;\,L\le k<R,(c_k,d_k)=(e,-1).
\end{align*}
A computation gives
\begin{equation}\label{partitionfunctionSchur}
Z=\prod_{k_1\in I_1}\prod_{k_2\in I_2}\prod_{k_3\in I_3}\prod_{k_4\in I_4}\prod_{1\le i,j\le m}
\frac{(1+a_i^{k_1}a_j^{k_4})(1+a_i^{k_2}a_j^{k_3})}{(1-a_i^{k_1}a_j^{k_3})(1-a_i^{k_2}a_j^{k_4})}.
\end{equation}

\begin{remark} {\rm The \emph{Schur measure} is a special case of the Schur process that was introduced earlier, \cite{Ok1}. It is a probability measure on partitions defined by
\begin{equation}\label{Schurmeasure}
p(\lambda)=\frac 1{Z}s_{\lambda}(a_1,\dots,a_m)s_{\lambda}(b_1,\dots,b_m),
\end{equation}
where  $a_i,b_i\in (0,1)$ are parameters.
This fits into the Schur process framework (\ref{Schurprocess}) above by taking $L=0$, $R=2$, $\lambda^{(0)}=\lambda^{(2)}=\emptyset$, $\lambda^{(1)}=\lambda$,
$c_0=c_1=h$, $d_0=1$, $d_1=-1$, $a^0=(a_1,\dots,a_m)$ and $a^1=(b_1,\dots,b_m)$. By the general formalism (\ref{KLRcontour}) becomes the correlation kernel
\begin{equation}\label{Schurmeasurekernel}
K^{\text{Schur}}(u,v)=\frac 1{(2\pi \mathrm{i})^2}\int_{\gamma_{\rho_1}}dz\int_{\gamma_{\rho_2}}dw\frac{z^{u+m-1}}{w^{v+m}(w-z)}\prod_{i=1}^m\frac{(1-a_iz)(w-b_i)}{(1-a_iw)(z-b_i)},
\end{equation}
where $\max(b_i)<\rho_1<\rho_2<\min(1/a_i)$. This is a correlation kernel
for the determinantal point process $\{x_j\}$, where $x_j=\lambda_j+1-j$, $j\ge 1$. The normalization constant $Z$ is
\begin{equation*}
Z=\prod_{1\le i,j\le m}\frac 1{1-a_ib_j},
\end{equation*}
by (\ref{partitionfunctionSchur}).}
\end{remark}

\begin{remark}\label{rem:Schur}
{\rm The Schur measure arises also in the corner growth model, last-passage percolation, and the totally asymmetric simple exclusion process (TASEP), see
\cite{JoHouch}, \cite{JoSh}, \cite{BDS}, \cite{Corw}. A special case gives the Meixner ensemble, a discrete orthogonal polynomial ensemble with the negative binomial weight.
The Airy process discussed in section \ref{subsec:Airypoint} is an important limit process in local random growth models, directed
random polymers and interacting particle systems. In this context there has been many important generalizations beyond the determinantal point process context starting with
\cite{TrWi} on the asymmetric simple exclusion process (ASEP) and \cite{Oc} on a positive temperature directed polymer. For example, the Schur process has been generalized to a much wider class of processes called Macdonald processes, \cite{BorCorw}, \cite{BoGo}. We will not attempt to summarize this here but instead refer to the review papers
\cite{BoGo}, \cite{BorPet}, \cite{Corw}, \cite{KrKr}, \cite{Quas} and references therein. See also the books \cite{BDS} and \cite{Rom} for the connection to the longest
increasing subsequence problem for random permutations.}

\end{remark}

\begin{remark} {\rm It is possible to introduce dynamics on Schur processes which leads to a picture that also includes various interacting particle systems, see \cite{BoFe}, \cite{Bor}. This
is also related to simulations of random tiling and dimer models, see e.g. \cite{EKLP2}, \cite{JPS}, \cite{Pr}, \cite{BBBCCV}.}
\end{remark}

Lozenge or rhombus tilings of certain regions, e.g. a hexagon can also be described by non-intersecting paths, see e.g. figure 1 in \cite{FFN}. These models do not seem to fit
immediately into the Schur process framework. In the next section we will consider the closely related random skew plane partitions which are given by a Schur process. In fact
this is the context in which Schur processes were first introduced, see \cite{OkRe1}. Random lozenge tilings of a hexagon can be described in terms of interlacing particle
systems which are given by a measure similar to Schur processes, see section \ref{sec:5}.

\subsection{Random skew plane partitions} \label{sec4.4}
Let $\mu\subseteq\langle a^b\rangle$ be a partition that is contained in the rectangular partition of size $a\times b$ viewed as a Young diagram. A {\it skew plane partition}
is a map
\begin{equation*}
\pi:\langle a^b\rangle\setminus\mu\ni (i,j)\mapsto\pi_{i,j}\in\{0,1,2,\dots\}
\end{equation*}
that is monotone
\begin{equation}\label{monotone}
\pi_{i,j}\ge\pi_{i+k,j+\ell}
\end{equation}
for $k,\ell\ge 0$. Placing $\pi_{i,j}$ cubes over the square $(i,j)$ in $\langle a^b\rangle\setminus\mu$ gives a 3-dimensional object, a {\it skew 3D partition} also denoted by
$\pi$. We can also think of this as a type of rhombus or lozenge tiling, see figure 1 in \cite{OkRe3}. Let
\begin{equation}\label{sizePP}
|\pi|=\sum_{(i,j)\in\langle a^b\rangle\setminus\mu}\pi_{i,j}
\end{equation}
denote the {\it volume} of the 3D partition. For $-a<k<b$ we define the partitions
\begin{equation}\label{partitionPP}
\lambda^{(k)}=(\pi_{i,k+i})_{i\ge 1},
\end{equation}
and also let $\lambda^{(-a)}=\lambda^{(b)}=\emptyset$. 

Write the partition $\mu$ as $\mu=\langle c_1^{\ell_1}c_2^{\ell_2}\dots c_m^{\ell_m}\rangle$, where $0=c_0<c_1<\dots<c_m$ and $\ell_i\ge 1$. Set
\begin{equation*}
L_r=\ell_r+\dots+\ell_m,\quad 1\le r\le m,
\end{equation*}
and $L_{k+1}=0$, so that $L_1=\ell(\mu)$. For the partition $\mu$ we have {\it inner corners} at $(L_j+1,c_{j-1}+1)$, which occur at the
{\it inner times} $t_j=c_{j-1}-L_j$, $1\le j\le m+1$. We also have {\it outer corners} at $(L_j+1,c_{j}+1)$, which occur at the
{\it outer times} $t_j=c_{j}-L_j$, $1\le j\le m$. Note that
\begin{equation}\label{timesequence}
-a=\tau_0<t_1<\tau_1<t_2<\dots<\tau_m<t_{m+1}=b.
\end{equation}
It follows from (\ref{monotone}) and (\ref{partitionPP}) that we have the interlacing structure
\begin{equation}\label{lambdaorder1}
\lambda^{(\tau_{r-1}+i)}\prec\lambda^{(\tau_{r-1}+i+1)}
\end{equation}
for $0\le i\le t_r-\tau_{r-1}-1$, and
\begin{equation}\label{lambdaorder2}
\lambda^{(t_{r}+i)}\succ\lambda^{(t_{r}+i+1)},
\end{equation}
for $0\le i\le \tau_{r}-t_r-1$, $1\le r\le m+1$. Define the two sets of increasing times $I$, and decreasing times $J$,
\begin{equation}\label{Idef}
I=\bigcup_{r=1}^{m+1}\{\tau_{r-1},\dots, t_r-1\},
\end{equation}
\begin{equation}\label{Jdef}
J=\bigcup_{r=1}^{m+1}\{t_r,\dots, \tau_r-1\}.
\end{equation}

Let $0<q_k<1$ be given parameters, and let $\lambda^{(k)}$ be the sequence (\ref{partitionPP}) coming from a skew plane partition. We can introduce a probability
measure on skew plane partitions by
\begin{equation}\label{PPmeasure}
\mathbb{P}\big[\{\lambda^{(k)}\}_{k=-a+1}^{b-1}\big]=\frac 1Z\prod_{k=-a+1}^{b-1}q_k^{|\lambda^{(k)}|}.
\end{equation}
We see from (\ref{lambdaorder1}) to (\ref{Jdef}) that we can write this probability measure as a measure on all sequences of partitions 
$\{\lambda^{(k)}\}_{k=-a+1}^{b-1}$ as
\begin{equation}\label{PPmeasure2}
\mathbb{P}\big[\{\lambda^{(k)}\}_{k=-a+1}^{b-1}\big]=\frac 1Z\left(\prod_{k=-a+1}^{b-1}q_k^{|\lambda^{(k)}|}\right)\prod_{k\in I}\mathbb{I}_{\lambda^{(k)}\prec\lambda^{(k+1)}}
\prod_{k\in J}\mathbb{I}_{\lambda^{(k)}\succ\lambda^{(k+1)}}.
\end{equation}
We want to fit this into a Schur process so that we can apply the general theory to get a determinantal point process and a formula for the correlation kernel. This leads to
the following theorem,\cite{OkRe3}.
\begin{thma}\label{thm:PP}
Consider the probability measure (\ref{PPmeasure2}) on sequences of partitions $\{\lambda^{(k)}\}_{k=-a+1}^{b-1}$ and set $x_j^k=\lambda^{(k)}_i-j+1$. If we think of 
$x_j^k$, $j\ge 1$, $-a<k<b$, as positions of particles we get a determinantal point process with correlation kernel given by
\begin{equation}\label{PPkernel}
K(r,u;s,v)=\frac 1{(2\pi \mathrm{i})^2}\int_{\gamma_{\rho_2}}dz\int_{\gamma_{\rho_2}}dw\frac{z^{u-1}}{w^v(w-z)}\frac{\prod_{k=-a}^{r-1}(1-a_k^+z)\prod_{k=s}^{b-1}(1-a_k^-/w)}
{\prod_{k=r}^{b-1}(1-a_k^-w)\prod_{k=a}^{s-1}(1-a_k^+/z)},
\end{equation}
where $q_{-a}\dots q_r<\rho_2<\rho_1<q_{-a}\dots q_{s-1}$ if $r\ge s$, and $\rho_2>q_{-a}\dots q_r$, $\rho_1<q_{-a}\dots q_{s-1}$ and $\rho_1<\rho_2$ if $r<s$.
Here
\begin{align*}
a_k^+&=(q_{-a}\dots q_k)^{-1}, \quad a_k^-=0, \,\, \text{if $k\in I$},\\
a_k^+&=0, \quad a_k^-=q_{-a}\dots q_k, \,\, \text{if $k\in J$}.
\end{align*}
\end{thma}

\begin{proof}
Let $0<\beta\le 1$ and consider the modified measure
\begin{equation}\label{PPmeasuremodified}
\frac 1{Z(\beta)}\left(\prod_{k=-a+1}^{b-1}q_k^{|\lambda^{(k)}|}\right)\prod_{k\in I}\beta^{|\lambda^{(k+1)}|-|\lambda^{(k)}|}
\prod_{k\in I}\mathbb{I}_{\lambda^{(k)}\prec\lambda^{(k+1)}}
\prod_{k\in J}\mathbb{I}_{\lambda^{(k)}\succ\lambda^{(k+1)}}.
\end{equation}
Let
\begin{align*}
b_k&=(q_{-a}\dots q_k)^{-1}\beta\quad\text{if $k\in I$},\\
b_k&=q_{-a}\dots q_k\,\,\, \quad\quad\quad\text{if $k\in J$}.
\end{align*}
Consider the Schur process
\begin{align}\label{PPSchurprocessmodified}
&\frac 1{Z(\beta)}\prod_{k\in I}s_{\lambda^{(k+1)}/\lambda^{(k)}}(b_k)\prod_{k\in J}s_{\lambda^{(k)}/\lambda^{(k+1)}}(b_k)\\
&=\frac 1{Z(\beta)}\prod_{k\in I}b_k^{|\lambda^{(k+1)}|-|\lambda^{(k)}|}\prod_{k\in J}b_k^{|\lambda^{(k)}|-|\lambda^{(k+1)}|}
\prod_{k\in I}\mathbb{I}_{\lambda^{(k)}\prec\lambda^{(k+1)}}
\prod_{k\in J}\mathbb{I}_{\lambda^{(k)}\succ\lambda^{(k+1)}},\notag
\end{align}
where we used (\ref{Schurinterlacing}).
Now, as is straightforward to check,
\begin{equation*}
\prod_{k\in I}b_k^{|\lambda^{(k+1)}|-|\lambda^{(k)}|}\prod_{k\in J}b_k^{|\lambda^{(k)}|-|\lambda^{(k+1)}|}
=\prod_{k=-a+1}^{b-1}q_k^{|\lambda^{(k)}|}\prod_{k\in I}\beta^{|\lambda^{(k+1)}|-|\lambda^{(k)}|},
\end{equation*}
(recall that $|\lambda^{(-a)}|=|\lambda^{(b)}|=0$). Thus, (\ref{PPmeasuremodified}) and (\ref{PPSchurprocessmodified}) are identical, and it follows that we get a
determinantal point process provided $\beta$ is so small that all $b_k$ are $<1$. This is the reason that we cannot take $\beta=1$ immediately. Note, however, that
the measure is an analytical function of $\beta$ for $0<\beta<1+\epsilon$ with $\epsilon$ small enough. Here we use the explicit form of $Z(\beta)$ in (\ref{partitionfunctionSchur}).
Hence, if we can write the formulas approproately, we can use analyticity to take the limit $\beta\to 1$.
Using the notation (\ref{phiSchur}), we see that (\ref{PPSchurprocessmodified}) can be written
\begin{equation*}
\frac 1{Z(\beta)}\prod_{k\in I}\det\big(\hat{\phi}_{h,+1,b_k}(x_i^{k+1}-x_j^k)\big)\prod_{k\in J}\det\big(\hat{\phi}_{h,-1,b_k}(x_i^{k+1}-x_j^k)\big),
\end{equation*}
where $\phi_{h,\pm 1,b_k}(z)=1/(1-b_k z^{\pm 1})$, so that
\begin{align*}
\phi^-_{h,1,b_k}(z)&=1,\quad \phi^+_{h,1,b_k}(z)=\frac 1{1-b_kz},\\
\phi^-_{h,-1,b_k}(z)&=\frac 1{1-b_k/z},\quad \phi^+_{h,-1,b_k}(z)=1.
\end{align*}
The correlation kernel is given by (\ref{KLRcontour}) and we see that 
\begin{align*}
\phi_{r,s}^+(z)&=\prod_{k\in I\cap [r,s-1]}\frac 1{1-b_kz}=\prod_{k=r}^{s-1}\frac 1{1-b_k^+z},\\
\phi_{r,s}^-(z)&=\prod_{k\in J\cap [r,s-1]}\frac 1{1-b_k/z}=\prod_{k=r}^{s-1}\frac 1{1-b_k^-/z},
\end{align*}
where $b_k^+=b_k$, $k\in I$, $b_k^+=0$, $k\in J$, and $b_k^-=b_k$, $k\in J$, $b_k^-=0$, $k\in I$. Thus, by (\ref{KLRcontour}), the correlation kernel is given by
\begin{align}\label{PPkernelmod}
K(r,u;s,v)&=-\frac{\mathbb{I}_{r<s}}{2\pi \mathrm{i}}\int_{\gamma_1}\frac{z^{u-v}}{\prod_{k=r}^{s-1}(1-b_k^+z)\prod_{k=r}^{s-1}(1-b_k^-/z)}\frac {dz}{z}\\
&+\frac 1{(2\pi \mathrm{i})^2}\int_{\gamma_{\rho_2}}\int_{\gamma_{\rho_3}}\frac{z^{u-1}}{w^v(w-v)}
\frac{\prod_{k=a}^{r-1}(1-b_k^+z)\prod_{k=s}^{b-1}(1-b_k^-/w)}{\prod_{k=a}^{s-1}(1-b_k^+w)\prod_{k=r}^{b-1}(1-b_k^-/z)},
\notag
\end{align}
where
\begin{equation*}
\max_{r\le k<b}(b_k^-)<\rho_2<\rho_3<\min_{a\le k<s}(1/b_k^+),
\end{equation*}
i.e.
\begin{equation}\label{radiirelation1}
q_{-a}\dots q_r<\rho_2<\rho_3<\frac 1{\beta}q_{-a}\dots q_{s-1}.
\end{equation}
For the first integral, we also require that
\begin{equation}\label{radiirelation2}
q_{-a}\dots q_r<1<\frac 1{\beta}q_{-a}\dots q_{s-1}
\end{equation}
for $r<s$.
Assume $r\le s$. Then $q_{-a}\dots q_r<q_{-a}\dots q_{s-1}$, and we can take the limit $\beta\to 1$in (\ref{PPkernelmod}) and obtain the result of the theorem. If $r<s$, then we
take $q_{-a}\dots q_r<\rho_1<\rho_2<\frac 1{\beta}q_{-a}\dots q_{s-1}$ and we move the $w$-contour $\gamma_{\rho_3}$ in (\ref{PPkernelmod}) to
$\gamma_{\rho_1}$, which gives
\begin{equation*}
K(r,u;s,v)=\frac 1{(2\pi \mathrm{i})^2}\int_{\gamma_{\rho_2}}dz\int_{\gamma_{\rho_2}}dw\frac{z^{u-1}}{w^v(w-z)}\frac{\prod_{k=-a}^{r-1}(1-b_k^+z)\prod_{k=s}^{b-1}(1-b_k^-/w)}
{\prod_{k=r}^{b-1}(1-b_k^-w)\prod_{k=a}^{s-1}(1-b_k^+/z)}.
\end{equation*}
The contribution from the pole at $w=z$ exactly cancels the single integral in the right side of (\ref{PPkernelmod}).
This holds provided $\rho_2>q_{-a}\dots q_r$, $\rho_1<q_{-a}\dots q_{s-1}$ and $\rho_1<\rho_2$. We can now let $\beta\to 1$ and we have proved the theorem.
\end{proof}

It follows from (\ref{sizePP}) and (\ref{partitionPP}) that if we take $q_k=q\in(0,1)$ for all $k$, then the probability measure (\ref{PPmeasure}) on partitions satisfying the ordering condition becomes
\begin{equation*}
\mathbb{P}\big[\{\lambda^{(k)}\}_{k=-a+1}^{b-1}\big]=\frac 1Z q^{|\pi|},
\end{equation*}
which is called the $q^{\text{vol}}-measure$ on skew 3D partitions.

\subsection{The Double Aztec diamond} \label{sec4.5}
Consider a random tiling of the double Aztec diamond as shown in figure \ref{fig:tpn200}. The Double Aztec diamond shape is shown in figure \ref{fig:DoubleAztecShape}.
A tiling of the Double Aztec diamond can be described by the same type of non-intersecting paths, here called the {\it outlier paths},
as in the ordinary Aztec diamond, see figure \ref{fig:Outlierpaths}. We call the associated particles the {\it outlier particles} and they are indicated in the figure. After the same mapping as for the standard Aztec diamond we obtain the paths in figure \ref{fig:InlierOutlierpaths}.

\begin{figure}
\begin{center}
\includegraphics[height=2.5in]{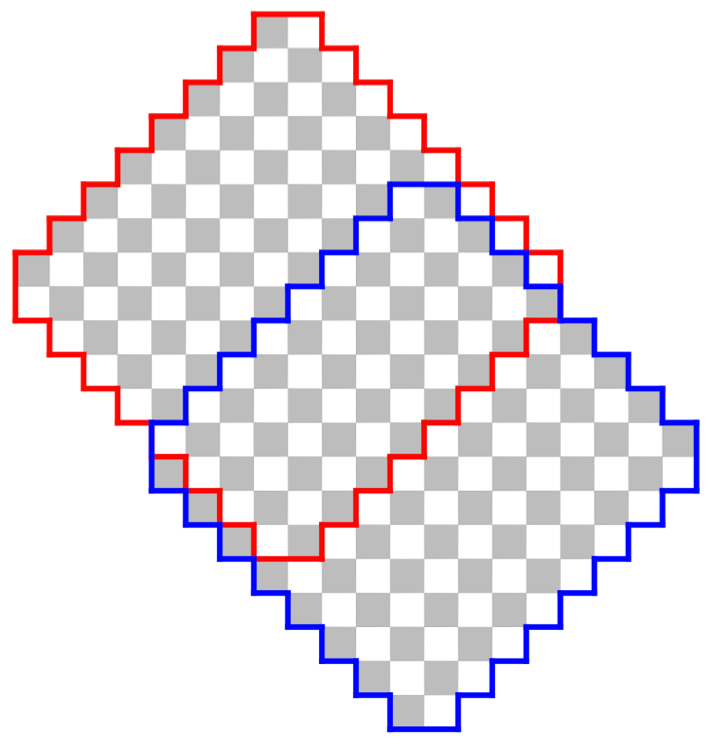}
\caption{The double Aztec shape.}
\label{fig:DoubleAztecShape}
\end{center}
\end{figure}

\begin{figure}
\begin{center}
\includegraphics[height=3.5in]{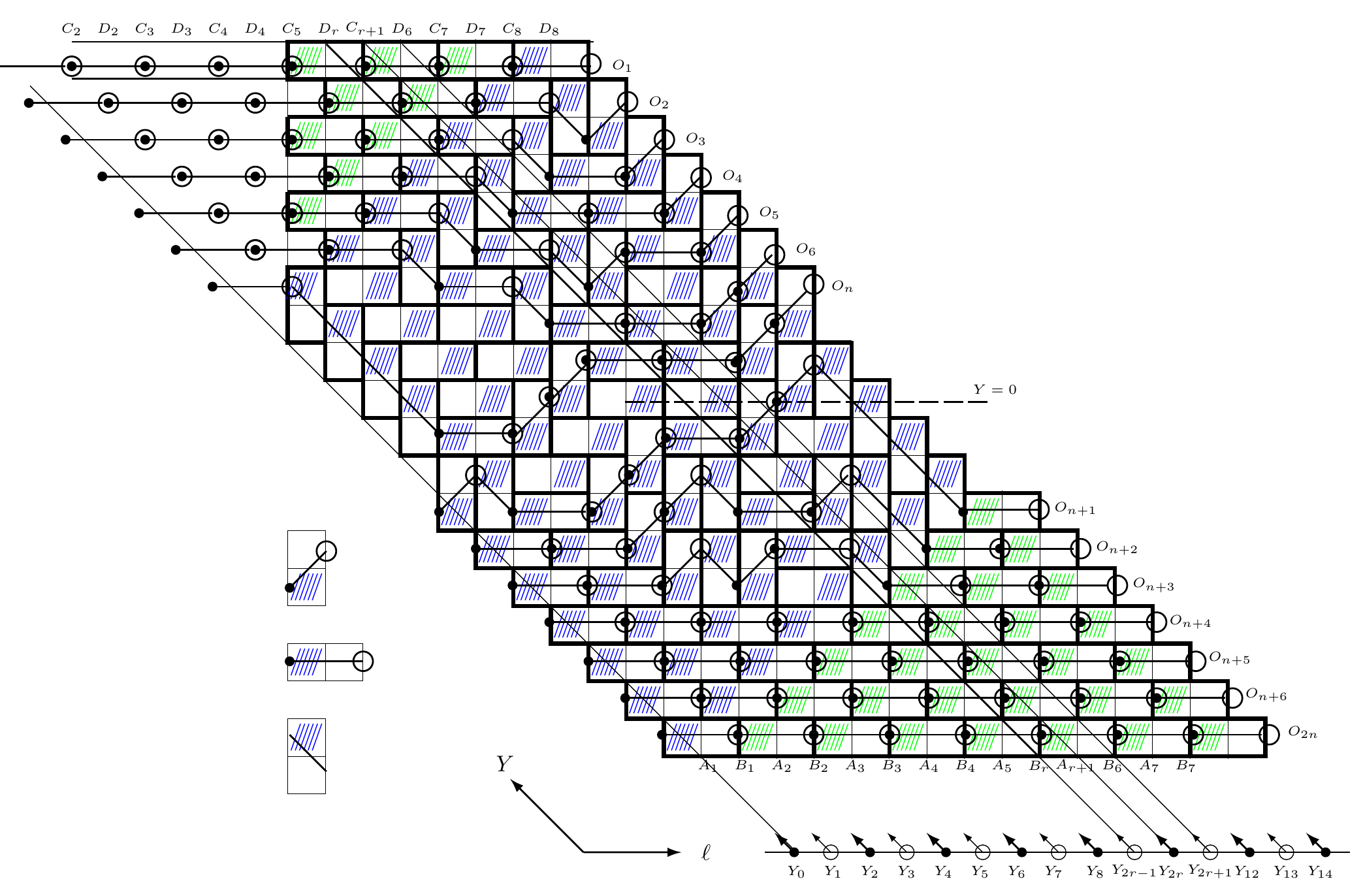}
\caption{Paths describing a random tiling of the double Aztec diamond together with the outlier particles.}
\label{fig:Outlierpaths}
\end{center}
\end{figure}

\begin{figure}
\begin{center}
\includegraphics[height=3in]{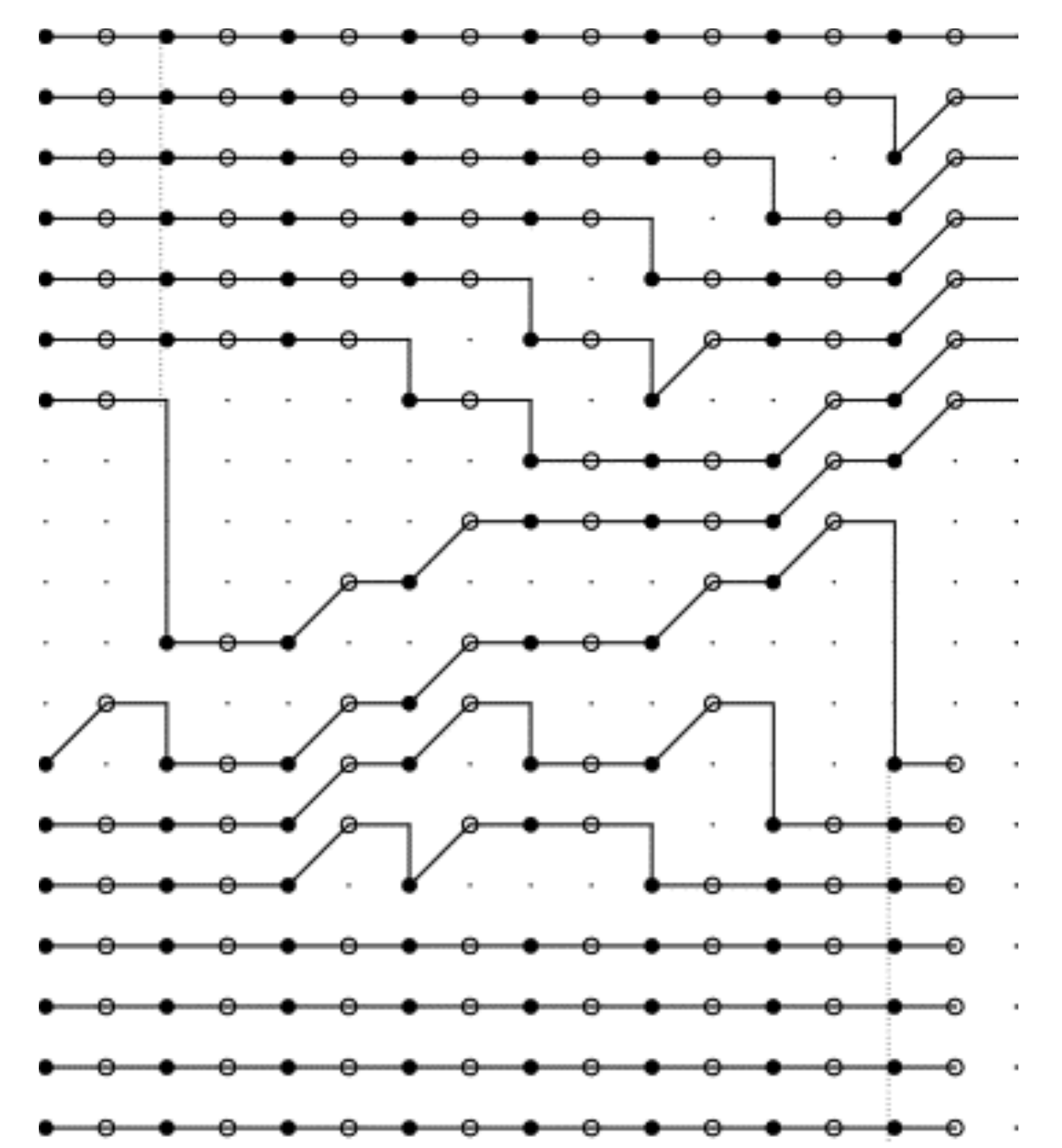}
\hspace{2mm}
\includegraphics[height=3in]{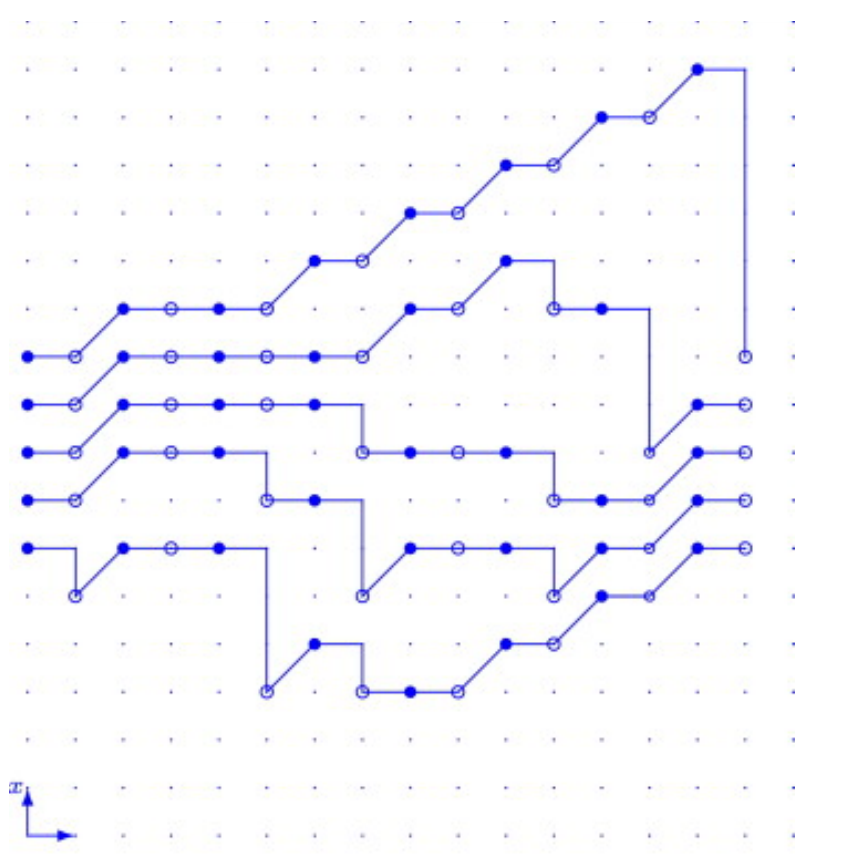}
\caption{To the left the modified outlier paths and particles in the double Aztec diamond, and to the right the inlier paths and particles. If we superimpose the two pictures we
see that the outlier particles are dual to the inlier particles.}
\label{fig:InlierOutlierpaths}
\end{center}
\end{figure}

This system of non-intersecting paths does not fit into the Schur process framework. However, by considering the corresponding dual particle system, called the {\it inlier particles} and
the corresponding non-intersecting paths, the {\it inlier paths}, see figure \ref{fig:InlierOutlierpaths}, we get a model that we can analyze using the results from section \ref{sec:2.2.3}.

Consider again the transition functions (\ref{Aztectransition}) and the measure (\ref{Aztecproductmeasure}) with $N=2n+1$, and $M=2m+1$ inlier particles on each vertical line.
Then, with $\epsilon_1,\epsilon_2\in\{0,1\}$,
\begin{equation}\label{DAtransition}
p_{2r+\epsilon_1,2s+\epsilon_2}(x,y)=\frac {\mathbb{I}_{2r+\epsilon_1<2s+\epsilon_2}}{2\pi \mathrm{i}}\int_{\gamma_1}\phi_{2r+\epsilon_1,2s+\epsilon_2}(z)\frac{dz}{z},
\end{equation}
where $\phi_{2r+\epsilon_1,2s+\epsilon_2}$ is given by (\ref{phiAztecWH}) and (\ref{phiAzpm}). We have initial and final points $x_j^0=x_j^N=m+1-j$, $1\le j\le 2m+1=M$. Thus,
we get a kernel of the form (\ref{finiteMkernel}) with $L=0, R=N=2n+1$. The correlation kernel for the outlier kernel is then given by formula (\ref{dualKLRMfinal}) in
theorem \ref{thm:KMfinite}. Note that the inlier particles are the particles in this theorem whereas the outlier particles are the dual particles. From (\ref{Kzero}), (\ref{abformulas})
and (\ref{phiAzpm}) we see that
\begin{equation}\label{DAKzero}
\mathcal{K}_0(j,k)=\frac 1{(2\pi \mathrm{i})^2}\int_{\gamma_{\sigma_1}}d\omega\int_{\gamma_{\sigma_2}}d\zeta \frac{\omega^k}{\zeta^{j+1}(\zeta-\omega)}
\frac{(1+a\omega)^n(1-a/\omega)^{n+1}}{(1+a\zeta)^n(1-a/\zeta)^{n+1}},
\end{equation}
and
\begin{align}\label{DAabformulas}
a_{2s+\epsilon_2,v}(j)&=\frac 1{(2\pi \mathrm{i})^2}\int_{\gamma_{\rho_2}}dw\int_{\gamma_{\sigma_2}}d\zeta \frac{w^{m-v}}{\zeta^{j+1}(\zeta-w)}
\frac{(1+aw)^s(1-a/w)^{n-s+1-\epsilon_2}}{(1+a\zeta)^n(1-a/\zeta)^{n+1}}\\
b_{r,u}(k)&=\frac 1{(2\pi \mathrm{i})^2}\int_{\gamma_{\rho_1}}dz\int_{\gamma_{\sigma_1}}d\omega \frac{\omega^k}{z^{m+1-u}(\omega-z)}
\frac{(1+a\omega)^n(1-a/\omega)^{n+1}}{(1+az)^r(1-a/z)^{n-r+1-\epsilon_1}},\notag
\end{align}
where
\begin{equation}\label{radii}
a<\rho_3<\rho_1<\sigma_1<\sigma_2<\rho_2<1/a.
\end{equation}
Furthermore from (\ref{Mstar}) we obtain
\begin{equation}\label{DAMstar}
M_{0,2n+1}^*(2r+\epsilon_1,u;2s+\epsilon_1,v)=\frac 1{(2\pi \mathrm{i})^2}\int_{\gamma_{\rho_1}}dz\int_{\gamma_{\rho_3}}dw \frac{z^{u-m-1}}{w^{v-m}(z-w)}
\frac{(1+aw)^s(1-a/w)^{n-s+1-\epsilon_2}}{(1+az)^r(1-a/z)^{n-r+1-\epsilon_1}},
\end{equation}
where $a<\rho_3<\rho_1<1/a$, and from (\ref{qsr}),
\begin{equation}\label{DAqsr}
q_{2r+\epsilon_1,2s+\epsilon_2}(u,v)=
\frac {\mathbb{I}_{2r+\epsilon_1<2s+\epsilon_2}}{2\pi \mathrm{i}}\int_{\gamma_1}\frac{(1+az)^{s-r}}{(1-a/z)^{s-r+\epsilon_2-\epsilon_1}}z^{u-v}\frac{dz}{z}.
\end{equation}
It follows from (\ref{OneAzteckernel2}) that
\begin{align}\label{MstaroneArel}
&-q_{2r+\epsilon_1,2s+\epsilon_2}(u,v)+M_{0,2n+1}^*(2r+\epsilon_1,u;2s+\epsilon_1,v)\\
&=(-1)^{u-v}K_{n+1}^{\text{OneAz}}\big(2(n-r+1)-\epsilon_1,m+1-u;2(n-s+1)-\epsilon_2,m+1-v\big)\notag
\end{align}
Combining this with (\ref{finiteMkernel}) we have established the following theorem.

\begin{thma}\label{thm:DA}
The outlier particles in the Double Aztec diamond form a determinantal point process with correlation kernel given by
\begin{align}\label{DAkernel}
&K_{n,m}^{\text{DoubleAz}}(2r+\epsilon_1,u;2s+\epsilon_2,v)\\
&=(-1)^{u-v}K_{n+1}^{\text{OneAz}}\big(2(n-r+1)-\epsilon_1,m+1-u;2(n-s+1)-\epsilon_2,m+1-v\big)\notag\\
&+\sum_{k=2m+1}^\infty\big((I-\mathcal{K}_0)^{-1}_{2m+1}a_{2s+\epsilon_2,v}\big)(k)b_{2r+\epsilon_1,u}(k),\notag
\end{align}
where $\mathcal{K}_0$, $a_{2s+\epsilon_2,v}$ and $b_{2r+\epsilon_1,u}$ are given by (\ref{DAKzero}) and (\ref{DAabformulas}), $K_{n+1}^{\text{OneAz}}$ by
(\ref{OneAzteckernel2}) or (\ref{OneAzteckernel}), and $(I-\mathcal{K}_0)^{-1}_{2m+1}$ is the inverse acting on the space $\ell^2(\{2m+1,2m+2,\dots\})$.
\end{thma}

This theorem was established in \cite{AJvM}.

\begin{remark} {\rm There is a dimer model corresponding to the Double Aztec diamond analogously to the case of the Aztec diamond. It is possible to write a formula for this 
dimer model using $K_{n,m}^{\text{DoubleAz}}$, see \cite{ACJvM}.}
\end{remark}



\section{Interlacing particle systems} \label{sec:5}
Let $y^r=(y_1^r,\dots,y_r^r)\in\mathbb{Z}^r$, $1\le r\le n$. We say that $y^1,\dots, y^n$ is a (discrete) \emph{interlacing particle system} (or a \emph{Gelfand-Tsetlin pattern}) if
\begin{equation}\label{yinterlacing}
y_i^{r+1}\ge y_i^r>y_{i+1}^{r+1},
\end{equation}
for $1\le i<r$, $1\le r<n$. If we think of $y^r$ as the positions of $r$ particles on row $r$, where rows go upwards, then condition (\ref{yinterlacing}) says that
the particles in row $r$ interlace those of row $r+1$, see figure \ref{figTilHalfHexT}.  We will consider random uniform interlacing particle system with a fixed top row
\begin{equation}\label{yn}
y^n=(y_1^n,\dots,y_n^n)=(x_1,\dots, x_n),
\end{equation}
where $x_1,\dots,x_n$ are given. All configurations $y^1,\dots, y^n$ satistying (\ref{yinterlacing}) and (\ref{yn}) are equally probable. 
The lozenge tiling interpretation corresponding to figure \ref{figTilHalfHexT} is shown in figure \ref{figTilHalfHex}. They are related by a simple map.
In the random lozenge tiling interpretation this
corresponds to fixing lozenges in the top row and then considering uniform tilings within a region with these lozenges fixed. In this way we can for example get a uniform
random tiling of a hexagon, compare figure \ref{figEquivIntPart}.

\begin{figure}[H]
\centering
\begin{tikzpicture}[xscale=1/3,yscale=1/3]
\draw (1,0) \lozbt;
\draw (2,0) \lozbt;
\draw (3,0) \lozbt;
\draw (4,0) \lozyt;
\draw (4,0) \lozrt;
\draw (5,0) \lozrt;
\draw (6,0) \lozrt;
\draw (7,0) \lozrt;
\draw (8,0) \lozrt;
\draw (0,1) \lozbt;
\draw (1,1) \lozyt;
\draw (1,1) \lozrt;
\draw (2,1) \lozrt;
\draw (4,1) \lozbt;
\draw (5,1) \lozbt;
\draw (6,1) \lozyt;
\draw (6,1) \lozrt;
\draw (7,1) \lozrt;
\draw (8,1) \lozrt;

\draw (-1,2) \lozbt;
\draw (0,2) \lozyt;
\draw (1,2) \lozbt;
\draw (2,2) \lozbt;
\draw (3,2) \lozyt;
\draw (3,2) \lozrt;
\draw (4,2) \lozrt;
\draw (6,2) \lozbt;
\draw (7,2) \lozyt;
\draw (7,2) \lozrt;
\draw (8,2) \lozrt;
\draw (-2,3) \lozbt;
\draw (-1,3) \lozyt;
\draw (0,3) \lozbt;
\draw (1,3) \lozyt;
\draw (1,3) \lozrt;
\draw (3,3) \lozbt;
\draw (4,3) \lozbt;
\draw (5,3) \lozbt;
\draw (6,3) \lozyt;
\draw (7,3) \lozbt;
\draw (8,3) \lozyt;
\draw (8,3) \lozrt;
\draw (-3,4) \lozyt;
\draw (-3,4) \lozrt;
\draw (-1,4) \lozyt;
\draw (-1,4) \lozrt;
\draw (1,4) \lozbt;
\draw (2,4) \lozbt;
\draw (3,4) \lozyt;
\draw (3,4) \lozrt;
\draw (4,4) \lozrt;
\draw (6,4) \lozyt;
\draw (6,4) \lozrt;
\draw (8,4) \lozyt;
\draw (8,4) \lozrt;
\draw (-4,5) \lozyt;
\draw (-3,5) \lozbt;
\draw (-2,5) \lozyt;
\draw (-1,5) \lozbt;
\draw (0,5) \lozbt;
\draw (1,5) \lozyt;
\draw (1,5) \lozrt;
\draw (3,5) \lozbt;
\draw (4,5) \lozyt;
\draw (4,5) \lozrt;
\draw (6,5) \lozyt;
\draw (6,5) \lozrt;
\draw (8,5) \lozbt;
\draw (9,5) \lozyt;
\draw (-5,6) \lozyt;
\draw (-4,6) \lozyt;
\draw (-4,6) \lozrt;
\draw (-2,6) \lozyt;
\draw (-2,6) \lozrt;
\draw (-1,6) \lozrt;
\draw (1,6) \lozbt;
\draw (2,6) \lozyt;
\draw (2,6) \lozrt;
\draw (4,6) \lozyt;
\draw (4,6) \lozrt;
\draw (6,6) \lozbt;
\draw (7,6) \lozyt;
\draw (7,6) \lozrt;
\draw (9,6) \lozyt;
\draw (-6,7) \lozyt;
\draw (-5,7) \lozyt;
\draw (-4,7) \lozyt;
\draw (-4,7) \lozrt;
\draw (-2,7) \lozbt;
\draw (-1,7) \lozbt;
\draw (0,7) \lozyt;
\draw (0,7) \lozrt;
\draw (2,7) \lozbt;
\draw (3,7) \lozyt;
\draw (4,7) \lozyt;
\draw (4,7) \lozrt;
\draw (5,7) \lozrt;
\draw (7,7) \lozyt;
\draw (7,7) \lozrt;
\draw (9,7) \lozyt;

\filldraw (3.5,1)  circle (3pt);


\filldraw (0.5,2) circle (3pt);
\filldraw (5.5,2) circle (3pt);


\filldraw (-0.5,3) circle (3pt);

\filldraw (2.5,3) circle (3pt);

\filldraw (6.5,3) circle (3pt);

\filldraw (-1.5,4) circle (3pt);

\filldraw (0.5,4) circle (3pt);

\filldraw (5.5,4) circle (3pt);

\filldraw (7.5,4) circle (3pt);
\filldraw (-3.5,5) circle (3pt);

\filldraw (-1.5,5) circle (3pt);

\filldraw (2.5,5) circle (3pt);

\filldraw (5.5,5) circle (3pt);

\filldraw (7.5,5) circle (3pt);

\filldraw (-4.5,6) circle (3pt);

\filldraw (-2.5,6) circle (3pt);

\filldraw (0.5,6) circle (3pt);

\filldraw (3.5,6) circle (3pt);

\filldraw (5.5,6) circle (3pt);

\filldraw (8.5,6) circle (3pt);

\filldraw (-5.5,7) circle (3pt);

\filldraw (-4.5,7) circle (3pt);

\filldraw (-2.5,7) circle (3pt);

\filldraw (1.5,7) circle (3pt);

\filldraw (3.5,7) circle (3pt);

\filldraw (6.5,7) circle (3pt);

\filldraw (8.5,7) circle (3pt);
\draw (-6.5,8) circle (3pt);

\draw (-5.5,8) circle (3pt);

\draw (-4.5,8) circle (3pt);

\draw (-0.5,8) circle (3pt);

\draw (2.5,8) circle (3pt);

\draw (3.5,8) circle (3pt);

\draw (6.5,8) circle (3pt);

\draw (8.5,8) circle (3pt);

\end{tikzpicture}
\caption{\label{figTilHalfHexT}An example of an interlacing particle system and an associated tiling. The unfilled circles represent the
fixed particles.}
\end{figure}
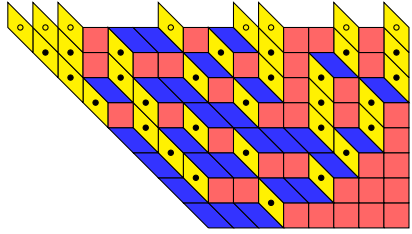

\begin{figure}[H]
\centering
\begin{tikzpicture}[xscale=1/3,yscale=1/3]
\draw (1,0) \lozb;
\draw (2,0) \lozb;
\draw (3,0) \lozb;
\draw (4,0) \loza;
\draw (4,0) \lozc;
\draw (5,0) \lozc;
\draw (6,0) \lozc;
\draw (7,0) \lozc;
\draw (8,0) \lozc;
\draw (.5,{sqrt(3)/2}) \lozb;
\draw (1.5,{sqrt(3)/2}) \loza;
\draw (1.5,{sqrt(3)/2}) \lozc;
\draw (2.5,{sqrt(3)/2}) \lozc;
\draw (4.5,{sqrt(3)/2}) \lozb;
\draw (5.5,{sqrt(3)/2}) \lozb;
\draw (6.5,{sqrt(3)/2}) \loza;
\draw (6.5,{sqrt(3)/2}) \lozc;
\draw (7.5,{sqrt(3)/2}) \lozc;
\draw (8.5,{sqrt(3)/2}) \lozc;
\draw (0,{sqrt(3)}) \lozb;
\draw (1,{sqrt(3)}) \loza;
\draw (2,{sqrt(3)}) \lozb;
\draw (3,{sqrt(3)}) \lozb;
\draw (4,{sqrt(3)}) \loza;
\draw (4,{sqrt(3)}) \lozc;
\draw (5,{sqrt(3)}) \lozc;
\draw (7,{sqrt(3)}) \lozb;
\draw (8,{sqrt(3)}) \loza;
\draw (8,{sqrt(3)}) \lozc;
\draw (9,{sqrt(3)}) \lozc;
\draw (-.5,{(1.5)*sqrt(3)}) \lozb;
\draw (.5,{(1.5)*sqrt(3)}) \loza;
\draw (1.5,{(1.5)*sqrt(3)}) \lozb;
\draw (2.5,{(1.5)*sqrt(3)}) \loza;
\draw (2.5,{(1.5)*sqrt(3)}) \lozc;
\draw (4.5,{(1.5)*sqrt(3)}) \lozb;
\draw (5.5,{(1.5)*sqrt(3)}) \lozb;
\draw (6.5,{(1.5)*sqrt(3)}) \lozb;
\draw (7.5,{(1.5)*sqrt(3)}) \loza;
\draw (8.5,{(1.5)*sqrt(3)}) \lozb;
\draw (9.5,{(1.5)*sqrt(3)}) \loza;
\draw (9.5,{(1.5)*sqrt(3)}) \lozc;
\draw (-1,{2*sqrt(3)}) \loza;
\draw (-1,{2*sqrt(3)}) \lozc;
\draw (1,{2*sqrt(3)}) \loza;
\draw (1,{2*sqrt(3)}) \lozc;
\draw (3,{2*sqrt(3)}) \lozb;
\draw (4,{2*sqrt(3)}) \lozb;
\draw (5,{2*sqrt(3)}) \loza;
\draw (5,{2*sqrt(3)}) \lozc;
\draw (6,{2*sqrt(3)}) \lozc;
\draw (8,{2*sqrt(3)}) \loza;
\draw (8,{2*sqrt(3)}) \lozc;
\draw (10,{2*sqrt(3)}) \loza;
\draw (10,{2*sqrt(3)}) \lozc;
\draw (-1.5,{(2.5)*sqrt(3)}) \loza;
\draw (-.5,{(2.5)*sqrt(3)}) \lozb;
\draw (.5,{(2.5)*sqrt(3)}) \loza;
\draw (1.5,{(2.5)*sqrt(3)}) \lozb;
\draw (2.5,{(2.5)*sqrt(3)}) \lozb;
\draw (3.5,{(2.5)*sqrt(3)}) \loza;
\draw (3.5,{(2.5)*sqrt(3)}) \lozc;
\draw (5.5,{(2.5)*sqrt(3)}) \lozb;
\draw (6.5,{(2.5)*sqrt(3)}) \loza;
\draw (6.5,{(2.5)*sqrt(3)}) \lozc;
\draw (8.5,{(2.5)*sqrt(3)}) \loza;
\draw (8.5,{(2.5)*sqrt(3)}) \lozc;
\draw (10.5,{(2.5)*sqrt(3)}) \lozb;
\draw (11.5,{(2.5)*sqrt(3)}) \loza;
\draw (-2,{3*sqrt(3)}) \loza;
\draw (-1,{3*sqrt(3)}) \loza;
\draw (-1,{3*sqrt(3)}) \lozc;
\draw (1,{3*sqrt(3)}) \loza;
\draw (1,{3*sqrt(3)}) \lozc;
\draw (2,{3*sqrt(3)}) \lozc;
\draw (4,{3*sqrt(3)}) \lozb;
\draw (5,{3*sqrt(3)}) \loza;
\draw (5,{3*sqrt(3)}) \lozc;
\draw (7,{3*sqrt(3)}) \loza;
\draw (7,{3*sqrt(3)}) \lozc;
\draw (9,{3*sqrt(3)}) \lozb;
\draw (10,{3*sqrt(3)}) \loza;
\draw (10,{3*sqrt(3)}) \lozc;
\draw (12,{3*sqrt(3)}) \loza;
\draw (-2.5,{(3.5)*sqrt(3)}) \loza;
\draw (-1.5,{(3.5)*sqrt(3)}) \loza;
\draw (-.5,{(3.5)*sqrt(3)}) \loza;
\draw (-.5,{(3.5)*sqrt(3)}) \lozc;
\draw (1.5,{(3.5)*sqrt(3)}) \lozb;
\draw (2.5,{(3.5)*sqrt(3)}) \lozb;
\draw (3.5,{(3.5)*sqrt(3)}) \loza;
\draw (3.5,{(3.5)*sqrt(3)}) \lozc;
\draw (5.5,{(3.5)*sqrt(3)}) \lozb;
\draw (6.5,{(3.5)*sqrt(3)}) \loza;
\draw (7.5,{(3.5)*sqrt(3)}) \loza;
\draw (7.5,{(3.5)*sqrt(3)}) \lozc;
\draw (8.5,{(3.5)*sqrt(3)}) \lozc;
\draw (10.5,{(3.5)*sqrt(3)}) \loza;
\draw (10.5,{(3.5)*sqrt(3)}) \lozc;
\draw (12.5,{(3.5)*sqrt(3)}) \loza;

\filldraw (4,{sqrt(3)/2}) circle (3pt);
\filldraw (1.5,{sqrt(3)}) circle (3pt);
\filldraw (6.5,{sqrt(3)}) circle (3pt);
\filldraw (1,{(1.5)*sqrt(3)}) circle (3pt);
\filldraw (4,{(1.5)*sqrt(3)}) circle (3pt);
\filldraw (8,{(1.5)*sqrt(3)}) circle (3pt);
\filldraw (.5,{2*sqrt(3)}) circle (3pt);
\filldraw (2.5,{2*sqrt(3)}) circle (3pt);
\filldraw (7.5,{2*sqrt(3)}) circle (3pt);
\filldraw (9.5,{2*sqrt(3)}) circle (3pt);
\filldraw (-1,{(2.5)*sqrt(3)}) circle (3pt);
\filldraw (1,{(2.5)*sqrt(3)}) circle (3pt);
\filldraw (5,{(2.5)*sqrt(3)}) circle (3pt);
\filldraw (8,{(2.5)*sqrt(3)}) circle (3pt);
\filldraw (10,{(2.5)*sqrt(3)}) circle (3pt);
\filldraw (-1.5,{3*sqrt(3)}) circle (3pt);
\filldraw (.5,{3*sqrt(3)}) circle (3pt);
\filldraw (3.5,{3*sqrt(3)}) circle (3pt);
\filldraw (6.5,{3*sqrt(3)}) circle (3pt);
\filldraw (8.5,{3*sqrt(3)}) circle (3pt);
\filldraw (11.5,{3*sqrt(3)}) circle (3pt);
\filldraw (-2,{(3.5)*sqrt(3)}) circle (3pt);
\filldraw (-1,{(3.5)*sqrt(3)}) circle (3pt);
\filldraw (1,{(3.5)*sqrt(3)}) circle (3pt);
\filldraw (5,{(3.5)*sqrt(3)}) circle (3pt);
\filldraw (7,{(3.5)*sqrt(3)}) circle (3pt);
\filldraw (10,{(3.5)*sqrt(3)}) circle (3pt);
\filldraw (12,{(3.5)*sqrt(3)}) circle (3pt);
\draw (-2.5,{4*sqrt(3)}) circle (3pt);
\draw (-1.5,{4*sqrt(3)}) circle (3pt);
\draw (-.5,{4*sqrt(3)}) circle (3pt);
\draw (3.5,{4*sqrt(3)}) circle (3pt);
\draw (6.5,{4*sqrt(3)}) circle (3pt);
\draw (7.5,{4*sqrt(3)}) circle (3pt);
\draw (10.5,{4*sqrt(3)}) circle (3pt);
\draw (12.5,{4*sqrt(3)}) circle (3pt);

\draw (4.5,-.4); 

\end{tikzpicture}
\caption{The lozenge tiling corresponding to figure \ref{figTilHalfHexT}. The lozenges with
the unfilled circles are
fixed and the rest are chosen uniformly at random.}
\label{figTilHalfHex}
\end{figure}
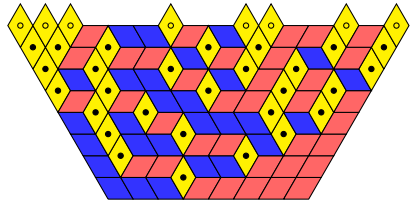

\begin{figure}[H]
\centering
\begin{tikzpicture}[xscale=1/2,yscale=1/2]

\draw (0,0) \lozb;
\draw (1,0) \lozb;
\draw (2,0) \loza;
\draw (2,0) \lozc;
\draw (3,0) \lozc;
\draw (-.5,{sqrt(3)/2}) \lozb;
\draw (.5,{sqrt(3)/2}) \lozb;
\draw (1.5,{sqrt(3)/2}) \loza;
\draw (2.5,{sqrt(3)/2}) \lozb;
\draw (3.5,{sqrt(3)/2}) \loza;
\draw (3.5,{sqrt(3)/2}) \lozc;
\draw (-1,{sqrt(3)}) \lozb;
\draw (0,{sqrt(3)}) \loza;
\draw (0,{sqrt(3)}) \lozc;
\draw (2,{sqrt(3)}) \lozb;
\draw (3,{sqrt(3)}) \loza;
\draw (4,{sqrt(3)}) \lozb;
\draw (5,{sqrt(3)}) \loza;
\draw (-1.5,{(1.5)*sqrt(3)}) \loza;
\draw (-1.5,{(1.5)*sqrt(3)}) \lozc;
\draw (.5,{(1.5)*sqrt(3)}) \loza;
\draw (.5,{(1.5)*sqrt(3)}) \lozc;
\draw (1.5,{(1.5)*sqrt(3)}) \lozc;
\draw (3.5,{(1.5)*sqrt(3)}) \loza;
\draw (3.5,{(1.5)*sqrt(3)}) \lozc;
\draw (5.5,{(1.5)*sqrt(3)}) \loza;
\draw (-1,{2*sqrt(3)}) \lozb;
\draw (0,{2*sqrt(3)}) \loza;
\draw (1,{2*sqrt(3)}) \lozb;
\draw (2,{2*sqrt(3)}) \loza;
\draw (2,{2*sqrt(3)}) \lozc;
\draw (4,{2*sqrt(3)}) \loza;
\draw (4,{2*sqrt(3)}) \lozc;
\draw (-1.5,{(2.5)*sqrt(3)}) \lozc;
\draw (.5,{(2.5)*sqrt(3)}) \loza;
\draw (.5,{(2.5)*sqrt(3)}) \lozc;
\draw (2.5,{(2.5)*sqrt(3)}) \loza;
\draw (2.5,{(2.5)*sqrt(3)}) \lozc;
\draw (4.5,{(2.5)*sqrt(3)}) \lozb;
\draw (-1,{3*sqrt(3)}) \lozc;
\draw (1,{3*sqrt(3)}) \lozb;
\draw (2,{3*sqrt(3)}) \loza;
\draw (3,{3*sqrt(3)}) \lozb;
\draw (4,{3*sqrt(3)}) \lozb;
\draw (-.5,{(3.5)*sqrt(3)}) \lozc;
\draw (.5,{(3.5)*sqrt(3)}) \lozc;
\draw (2.5,{(3.5)*sqrt(3)}) \lozb;
\draw (3.5,{(3.5)*sqrt(3)}) \lozb;

\filldraw (2,{(.5)*sqrt(3)}) circle (3pt);
\filldraw (1.5,{sqrt(3)}) circle (3pt);
\filldraw (3.5,{sqrt(3)}) circle (3pt);
\filldraw (0,{(1.5)*sqrt(3)}) circle (3pt);
\filldraw (3,{(1.5)*sqrt(3)}) circle (3pt);
\filldraw (5,{(1.5)*sqrt(3)}) circle (3pt);
\filldraw (-1.5,{2*sqrt(3)}) circle (3pt);
\filldraw (.5,{2*sqrt(3)}) circle (3pt);
\filldraw (3.5,{2*sqrt(3)}) circle (3pt);
\filldraw (5.5,{2*sqrt(3)}) circle (3pt);
\filldraw (0,{(2.5)*sqrt(3)}) circle (3pt);
\filldraw (2,{(2.5)*sqrt(3)}) circle (3pt);
\filldraw (4,{(2.5)*sqrt(3)}) circle (3pt);
\filldraw (0.5,{3*sqrt(3)}) circle (3pt);
\filldraw (2.5,{3*sqrt(3)}) circle (3pt);
\filldraw (2,{(3.5)*sqrt(3)}) circle (3pt);

\draw [dotted] (5,{(0.5)*sqrt(3)}) --++ (3.8,0);
\draw (10,{(0.55)*sqrt(3)}) node {row $1$};
\draw [dotted] (11.2,{(0.5)*sqrt(3)}) --++ (3.8,0);
\draw [dotted] (5.5,{sqrt(3)}) --++ (3.3,0);
\draw (10,{(1.05)*sqrt(3)}) node {row $2$};
\draw [dotted] (11.2,{sqrt(3)}) --++ (3.3,0);
\draw [dotted] (6,{(1.5)*sqrt(3)}) --++ (2.8,0);
\draw (10,{(1.55)*sqrt(3)}) node {row $3$};
\draw [dotted] (11.2,{(1.5)*sqrt(3)}) --++ (2.8,0);
\draw [dotted] (6.5,{2*sqrt(3)}) --++ (2.3,0);
\draw (10,{(2.05)*sqrt(3)}) node {row $4$};
\draw [dotted] (11.2,{2*sqrt(3)}) --++ (2.3,0);
\draw [dotted] (6,{(2.5)*sqrt(3)}) --++ (2.3,0);
\draw (9.5,{(2.55)*sqrt(3)}) node {row $5$};
\draw [dotted] (10.7,{(2.5)*sqrt(3)}) --++ (2.3,0);
\draw [dotted] (5.5,{3*sqrt(3)}) --++ (2.3,0);
\draw (9,{(3.05)*sqrt(3)}) node {row $6$};
\draw [dotted] (10.2,{3*sqrt(3)}) --++ (2.3,0);
\draw [dotted] (5,{(3.5)*sqrt(3)}) --++ (2.3,0);
\draw (8.5,{(3.55)*sqrt(3)}) node {row $7$};
\draw [dotted] (9.7,{(3.5)*sqrt(3)}) --++ (2.3,0);
\draw [dotted] (4.5,{4*sqrt(3)}) --++ (2.3,0);
\draw (8,{(4.05)*sqrt(3)}) node {row $8$};
\draw [dotted] (9.2,{4*sqrt(3)}) --++ (2.3,0);

\draw (16,0) \lozb;
\draw (17,0) \lozb;
\draw (18,0) \loza;
\draw (18,0) \lozc;
\draw (19,0) \lozc;
\draw (15.5,{sqrt(3)/2}) \lozb;
\draw (16.5,{sqrt(3)/2}) \lozb;
\draw (17.5,{sqrt(3)/2}) \loza;
\draw (18.5,{sqrt(3)/2}) \lozb;
\draw (19.5,{sqrt(3)/2}) \loza;
\draw (19.5,{sqrt(3)/2}) \lozc;
\draw (15,{sqrt(3)}) \lozb;
\draw (16,{sqrt(3)}) \loza;
\draw (16,{sqrt(3)}) \lozc;
\draw (18,{sqrt(3)}) \lozb;
\draw (19,{sqrt(3)}) \loza;
\draw (20,{sqrt(3)}) \lozb;
\draw (21,{sqrt(3)}) \loza;
\draw (14.5,{(1.5)*sqrt(3)}) \loza;
\draw (14.5,{(1.5)*sqrt(3)}) \lozc;
\draw (16.5,{(1.5)*sqrt(3)}) \loza;
\draw (16.5,{(1.5)*sqrt(3)}) \lozc;
\draw (17.5,{(1.5)*sqrt(3)}) \lozc;
\draw (17.5,{(1.5)*sqrt(3)}) \lozc;
\draw (19.5,{(1.5)*sqrt(3)}) \loza;
\draw (19.5,{(1.5)*sqrt(3)}) \lozc;
\draw (21.5,{(1.5)*sqrt(3)}) \loza;
\draw (14,{2*sqrt(3)}) \loza;
\draw (15,{2*sqrt(3)}) \lozb;
\draw (16,{2*sqrt(3)}) \loza;
\draw (17,{2*sqrt(3)}) \lozb;
\draw (18,{2*sqrt(3)}) \loza;
\draw (18,{2*sqrt(3)}) \lozc;
\draw (20,{2*sqrt(3)}) \loza;
\draw (20,{2*sqrt(3)}) \lozc;
\draw (22,{2*sqrt(3)}) \loza;
\draw (13.5,{(2.5)*sqrt(3)}) \loza;
\draw (14.5,{(2.5)*sqrt(3)}) \loza;
\draw (14.5,{(2.5)*sqrt(3)}) \lozc;
\draw (16.5,{(2.5)*sqrt(3)}) \loza;
\draw (16.5,{(2.5)*sqrt(3)}) \lozc;
\draw (18.5,{(2.5)*sqrt(3)}) \loza;
\draw (18.5,{(2.5)*sqrt(3)}) \lozc;
\draw (20.5,{(2.5)*sqrt(3)}) \lozb;
\draw (21.5,{(2.5)*sqrt(3)}) \loza;
\draw (22.5,{(2.5)*sqrt(3)}) \loza;
\draw (13,{3*sqrt(3)}) \loza;
\draw (14,{3*sqrt(3)}) \loza;
\draw (15,{3*sqrt(3)}) \loza;
\draw (15,{3*sqrt(3)}) \lozc;
\draw (17,{3*sqrt(3)}) \lozb;
\draw (18,{3*sqrt(3)}) \loza;
\draw (19,{3*sqrt(3)}) \lozb;
\draw (20,{3*sqrt(3)}) \lozb;
\draw (21,{3*sqrt(3)}) \loza;
\draw (22,{3*sqrt(3)}) \loza;
\draw (23,{3*sqrt(3)}) \loza;
\draw (12.5,{(3.5)*sqrt(3)}) \loza;
\draw (13.5,{(3.5)*sqrt(3)}) \loza;
\draw (14.5,{(3.5)*sqrt(3)}) \loza;
\draw (15.5,{(3.5)*sqrt(3)}) \loza;
\draw (15.5,{(3.5)*sqrt(3)}) \lozc;
\draw (16.5,{(3.5)*sqrt(3)}) \lozc;
\draw (18.5,{(3.5)*sqrt(3)}) \lozb;
\draw (19.5,{(3.5)*sqrt(3)}) \lozb;
\draw (20.5,{(3.5)*sqrt(3)}) \loza;
\draw (21.5,{(3.5)*sqrt(3)}) \loza;
\draw (22.5,{(3.5)*sqrt(3)}) \loza;
\draw (23.5,{(3.5)*sqrt(3)}) \loza;

\filldraw (18,{(.5)*sqrt(3)}) circle (3pt);
\filldraw (17.5,{sqrt(3)}) circle (3pt);
\filldraw (19.5,{sqrt(3)}) circle (3pt);
\filldraw (16,{(1.5)*sqrt(3)}) circle (3pt);
\filldraw (19,{(1.5)*sqrt(3)}) circle (3pt);
\filldraw (21,{(1.5)*sqrt(3)}) circle (3pt);
\filldraw (14.5,{2*sqrt(3)}) circle (3pt);
\filldraw (16.5,{2*sqrt(3)}) circle (3pt);
\filldraw (19.5,{2*sqrt(3)}) circle (3pt);
\filldraw (21.5,{2*sqrt(3)}) circle (3pt);
\draw (14,{(2.5)*sqrt(3)}) circle (3pt);
\filldraw (16,{(2.5)*sqrt(3)}) circle (3pt);
\filldraw (18,{(2.5)*sqrt(3)}) circle (3pt);
\filldraw (20,{(2.5)*sqrt(3)}) circle (3pt);
\draw (22,{(2.5)*sqrt(3)}) circle (3pt);
\draw (13.5,{3*sqrt(3)}) circle (3pt);
\draw (14.5,{3*sqrt(3)}) circle (3pt);
\filldraw (16.5,{3*sqrt(3)}) circle (3pt);
\filldraw (18.5,{3*sqrt(3)}) circle (3pt);
\draw (21.5,{3*sqrt(3)}) circle (3pt);
\draw (22.5,{3*sqrt(3)}) circle (3pt);
\draw (13,{(3.5)*sqrt(3)}) circle (3pt);
\draw (14,{(3.5)*sqrt(3)}) circle (3pt);
\draw (15,{(3.5)*sqrt(3)}) circle (3pt);
\filldraw (18,{(3.5)*sqrt(3)}) circle (3pt);
\draw (21,{(3.5)*sqrt(3)}) circle (3pt);
\draw (22,{(3.5)*sqrt(3)}) circle (3pt);
\draw (23,{(3.5)*sqrt(3)}) circle (3pt);
\draw (12.5,{4*sqrt(3)}) circle (3pt);
\draw (13.5,{4*sqrt(3)}) circle (3pt);
\draw (14.5,{4*sqrt(3)}) circle (3pt);
\draw (15.5,{4*sqrt(3)}) circle (3pt);
\draw (20.5,{4*sqrt(3)}) circle (3pt);
\draw (21.5,{4*sqrt(3)}) circle (3pt);
\draw (22.5,{4*sqrt(3)}) circle (3pt);
\draw (23.5,{4*sqrt(3)}) circle (3pt);

\end{tikzpicture}
\caption{Left: Equivalent interlaced particle configuration of a tiling of regular hexagon of size 4.
Right: Equivalent interlaced particle configuration with added fixed
lozenges/particles. The unfilled circles represent the fixed
particles, and this choice gives a uniform lozenge tiling of the regular hexagon.}
\label{figEquivIntPart}
\end{figure}
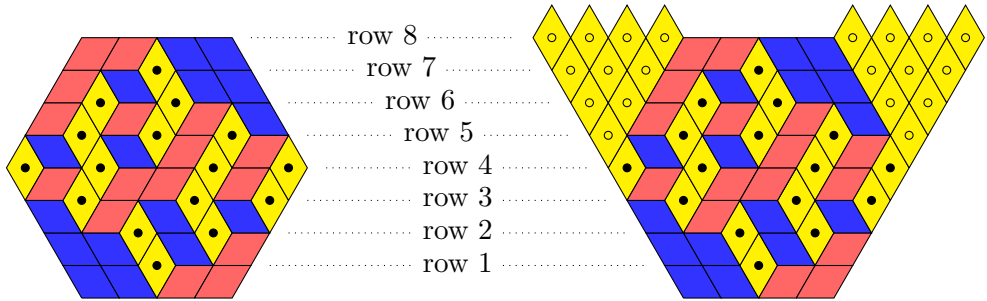

This uniform interlacing particle system is a determinantal point process with a kernel that we can compute. We have the following theorem.

\begin{thma}\label{thm:interlacing}
Consider a system of interlacing particles $y^1,\dots, y^n$ satisfying (\ref{yinterlacing}) and (\ref{yn}) with $x_1,\dots,x_n$ fixed, and choose such an interlacing particle configuration 
uniformly. The resulting particle system forms a determinantal point process. Let $\Gamma_u$ be a contour that contains $\{x_j\,;\,x_j\ge u\}$, but no other $x_j$, and let 
$\gamma$ be a contour that contains $\Gamma_u$ and $\{v-n+s,v-n+s+s,\dots,v\}$. Then a correlation kernel for the process is given by
\begin{align}\label{interlacingkernelfinal}
&K_n^{\text{Int}}(r,u;s,v)=-\mathbb{I}_{r<s}\mathbb{I}_{u\le v}\frac{(v-u+1)_{s-r-1}}{(s-r-1)!}\\
&+\frac{(n-s)!}{(n-r-1)!}\frac 1{(2\pi \mathrm{i})^2}\int_{\gamma}dw\int_{\Gamma_u}dz\frac{\prod_{k=u+r-n+1}^{u-1}(z-k)}{\prod_{k=v+s-n}^{v}(w-k)}\prod_{i=1}^n\frac{w-x_i}{z-x_i}\frac 1{w-z}.
\end{align}
\end{thma}

\begin{proof}
To the discrete interlacing particle system we can associate a sequence of partitions by putting
\begin{equation}\label{lambday}
\lambda_i^{(r)}=
\begin{cases}
y_i^r-x_n+i-n, &\text{for $1\le i\le r$,}\\
0 &\text{for $r< i\le n$,}
\end{cases}
\end{equation}
$1\le r\le n$. We can think of $y_i^r=x_n+n-i$, $r<i\le n$, corresponding to $\lambda_i^{(r)}=0$, as fixed "virtual particles" that have been added. It is convenient to let
$\lambda^{(0)}=\emptyset$ and correspondingly
\begin{equation}\label{yzero}
y_i^0=x_n+n-i,\quad 1\le i\le n.
\end{equation}
Note that
\begin{equation}\label{lambdan}
\lambda_i^{(n)}=x_i-x_n+i-n.
\end{equation}
By (\ref{Partitioninterlacing}), the interlacing condition translates into
\begin{equation}\label{lambdainterlacing}
\lambda^{(n)}\succ\lambda^{(n-1)}\succ\dots\succ\lambda^{(1)}\succ\lambda^{(0)}=\emptyset.
\end{equation}
Using (\ref{Schurinterlacing}),
we see that the uniform probability measure we are considering with a fixed top row (i.e. fixed $\lambda^{(n)}$) can be written as
\begin{equation}\label{interlacingmeasure}
\frac 1Z\prod_{r=0}^{n-1}s_{\lambda^{(r+1)}/\lambda^{(r)}}(1).
\end{equation}
Since we do not both start and end with an empty partition, this is not quite a Schur measure and we cannot apply the infinite Toeplitz matrix formalism. Since
$\ell(\lambda)=n$, the Jacob-Trudi identity (\ref{JacobiTrudi1}) shows that (\ref{interlacingmeasure}) can be written
\begin{equation}\label{interlacingPmeasure}
\frac 1Z\prod_{r=0}^{n-1}\det\big(h_{\lambda_i^{(r+1)}-\lambda_j^{(r)}-i+j}(1)\big)_{1\le i,j\le n}=
\frac 1Z\prod_{r=0}^{n-1}\det\big(p_{r,r+1}(y_i^r,y_j^{r+1})\big)_{1\le i,j\le n},
\end{equation}
where
\begin{equation}\label{prhvu}
p_{r,r+1}(u,v)=h_{v-u}(1).
\end{equation}
It follows from the general formalism in section \ref{sec:2.2} that the interlacing particle system is a determinantal point process with a correlation kernel given by
(\ref{generalK}). In order to givea more useful expression for this kernel we will use the Cramer's rule idea in section \ref{sec:2.2.1}. From (\ref{prhvu}) we get
\begin{equation}\label{prshvu}
p_{r,s}(u,v)=\mathbb{I}_{r<s}h_{v-u}(1^{s-r}).
\end{equation}
Given a function $f:\mathbb{Z}\mapsto\mathbb{C}$ let $\Delta_xf(x)=f(x+1)-f(x)$ be the finite difference operator with respect to the variable $x$.
\begin{lemma}\label{lem:Deltaformula}
We have that
\begin{equation}\label{Deltaformula}
\Delta_r^sh_r(1^N)=h_{r+s}(1^{N-s}).
\end{equation}
\end{lemma}
\begin{proof}
A simple computation shows that
\begin{equation*}
\sum_{r\in\mathbb{Z}}\Delta_rh_r(1^N)z^r=\sum_{r\in\mathbb{Z}}\Delta_rh_{r+1}(1^{N-1})z^r,
\end{equation*}
and repeated use gives (\ref{Deltaformula}).
\end{proof}

Let $A$ be the matrix
\begin{equation}\label{interlacingA}
A=\big(p_{0,n}(y_i^0,y_j^n)\big)_{1\le i,j\le n}=\big(h_{y_j^n-y_i^0}(1^n)\big)_{1\le i,j\le n},
\end{equation}
and let $A[j,s,v]$ be the same matrix but with column $j$ replaced by
\begin{equation*}
\left(\begin{matrix} h_{v+s-n-y_1^0}(1^n)\\
\vdots\\
h_{v+s-n-y_n}^0(1^n)
\end{matrix}\right).
\end{equation*}
By lemma \ref{Deltaformula},
\begin{equation*}
\Delta_v^{n-s}h_{v+s-n-y_i^0}(1^n)=h_{v-y_iô}(1^s),
\end{equation*}
and the argument in section \ref{sec:2.2.2} using Cramer's rule shows that the correlation kernel is given by
\begin{equation}\label{interlacingkernel}
K_n^{\text{Int}}(r,u;s,v)=-\mathbb{I}_{r<s}h_{v-u}(1^{s-r})+\tilde{K}_n(r,u;s,v)
\end{equation}
with
\begin{equation}\label{interlacingKtilde}
\tilde{K}_n(r,u;s,v)=\Delta_v^{n-s}\sum_{j=1}^nh_{y_j^n-u}(1^{n-r})\frac{\det A[j,s,v]}{\det a}.
\end{equation}
Now, by (\ref{yn}), (\ref{yzero}) and (\ref{ClassicalSchur}),
\begin{align*}
\det A=\det\big(h_{x_j-i-x_n-n}(1^n)\big)_{1\le i,j\le n}&=
\lim_{q\to 1}\det\big(h_{x_j-i-x_n-n}(1,q,\dots,q^{n-1})\big)_{1\le i,j\le n}\\
&=\lim_{q\to 1}\frac{\prod_{1\le i,j\le n}\bigg(q^{x_j-x_n}-q^{x_i-x_n}\bigg)}{\det\big(q^{(i-1)(n-j)}\big)_{1\le i,j\le n}}.
\end{align*}
It follows that
\begin{equation*}
\frac{\det A[j,s,v]}{\det a}=\prod_{\substack{i=1\\i\neq j}}^n\frac{v+s-n-x_i}{x_j-x_i},
\end{equation*}
and from (\ref{interlacingKtilde}) we obtain
\begin{equation}\label{interlacingKtilde2}
\tilde{K}_n(r,u;s,v)=\Delta_v^{n-s}\sum_{j=1}^nh_{y_j^n-u}(1^{n-r})\prod_{\substack{i=1\\i\neq j}}^n\frac{v+s-n-x_i}{x_j-x_i}.
\end{equation}
The formula $h_k(1^n)=\mathbb{I}_{k\ge 0}(k+1)_{n-1}/(n-1)!$ gives
\begin{equation}\label{hformula}
h_{x_j-u}(1^{n-r})=\mathbb{I}_{x_j\ge u}\frac 1{(n-r-1)!}\prod_{k=u+r-n+1}^{u-1}(x_j-k).
\end{equation}
For any function $f:\mathbb{Z}\mapsto\mathbb{C}$,
\begin{equation}\label{Deltaf}
(\Delta_v^{n-s}f)(v+s-n)=(n-s)!\sum_{\ell=v-n+s}^{u-1}\frac{f(\ell)}{\prod_{\substack{j=v-n+s\\ j \neq \ell}}^v(\ell-j)}.
\end{equation}
Combining (\ref{interlacingKtilde2}) to (\ref{Deltaf}), we find
\begin{equation}\label{interlacingKtilde3}
\tilde{K}_n(r,u;s,v)=\frac{(n-s)!}{(n-r-1)!}\sum_{j=1}^n\sum_{\ell=v-n+s}^{u-1}\mathbb{I}_{x_j\ge u}\frac{\prod_{k=u+r-n+1}^{u-1}(x_j-k)}
{\prod_{\substack{k=v-n+s\\ j \neq \ell}}^v(\ell-k)}
\prod_{\substack{i=1\\i\neq j}}^n\frac{\ell-x_i}{x_j-x_i}.
\end{equation}
It follows from (\ref{interlacingKtilde3}) and the residue theorem that
\begin{equation*}
\tilde{K}_n(r,u;s,v)=\frac{(n-s)!}{(n-r-1)!}\frac 1{(2\pi \mathrm{i})^2}\int_{\gamma}dw\int_{\Gamma_u}dz\frac{\prod_{k=u+r-n+1}^{u-1}(z-k)}
{\prod_{k=v+s-n}^{v}(w-k)}\prod_{i=1}^n\frac{w-x_i}{z-x_i}\frac 1{w-z}.
\end{equation*}
Also,
\begin{equation*}
p_{r,s}(u,v)=\mathbb{I}_{r<s}h_{v-u}(1^{s-r})=\mathbb{I}_{r<s}\mathbb{I}_{u\le v}\frac{(v-u+1)_{s-r-1}}{(s-r-1)!},
\end{equation*}
which completes the proof of the theorem by (\ref{interlacingkernel}).
\end{proof}

This theorem was first proved in \cite{Pe1} using a somewhat different approach. The proof given here is that in \cite{DuMeI}. A corresponding result for continuous
interlacing particle systems was studied in \cite{Me}, see also \cite{Def}.

\begin{remark} {\rm As shown in figure \ref{figTilHalfHex} the discrete interlacing particle process can be translated into a random tiling model. This model then
corresponds to a dimer model on a hexagonal graph, and it is possible to express the inverse Kasteleyn matrix for this dimer model using the kernel
$K_n^{\text{Int}}$, see \cite{Pe1}.}
\end{remark}




\section{Scaling limits} \label{sec:7}
Let $K_n(r_1,u_1;r_2,u_2)$ be a correlation kernel for some discrete model which is a determinantal point process like the ones that have been discussed in the previous
sections. Before taking a scaling limit it is often necessary to consider a conjugate kernel
\begin{equation}\label{conjukernel}
K^\ast_n(r_1,u_1;r_2,u_2)=\frac{g_n(r_1,u_1)}{g_n(r_2,u_2)}K_n(r_1,u_1;r_2,u_2),
\end{equation}
with some non-zero function $g_n(r,u)$, compare (\ref{conjugatekernel}). 

There are two types of scaling limits that are natural to consider in random tiling (and related) models. 

1) \emph{Continuous scaling limits}. We let
\begin{equation}\label{scaling}
r_i=[cn+\alpha n^{\gamma}\rho_i],\quad u_i=[dn+\beta n^\delta\xi_i],
\end{equation}
$i=1,2$, where $c,d,\alpha,\beta$ are constants, $\gamma,\delta$ are \emph{scaling exponents} and $\rho_i,\xi_i\in\mathbb{R}$ the new variables in the scaling limit, the
\emph{scaling variables}.
The constants $c,d$ determine the asymptotic point around which we take the limit, i.e. the point we zoom in around. The scaling exponents determine the scales in which we zoom in. Depending on the geometry of the asymptotic shape around the point where we zoom in, we can expect different scaling limits. The limiting kernel, if the limit exists, is then
given by
\begin{equation}\label{scalinglimit}
\mathcal{K}(\rho_1,\xi_1;\rho_2,\xi_2)=\lim_{n\to\infty}K^\ast_n(r_1,u_1;r_2,u_2),
\end{equation}
with $r_i,u_i$ given by (\ref{scaling}).

2) \emph{Discrete/continuous scaling limits}. In some geometric situations the natural scaling limit is to keep one variable discrete and just rescale the other variable. We then get
what we can call a discrete/continuous scaling limit. We let
\begin{equation}\label{scaling2}
u_i=[dn+\beta n^\delta\xi_i],
\end{equation}
and consider the limit
\begin{equation}\label{scalinglimit2}
\mathcal{K}(r_1,\xi_1;r_2,\xi_2)=\lim_{n\to\infty}K^\ast_n(r_1,u_1;r_2,u_2),
\end{equation}
with $u_i$ given by (\ref{scaling2}).

\begin{remark} {\rm Note that a scaling limit of a determinantal point process is not necessarily a determinantal point process. In \cite{CJY}, theorem 2.7, a scaling of a domino
process is considered where the limit is what is called a \emph{thickened determinantal point process} in the paper. It may be more appropriate to call it a \emph{compound determinantal
point process}.}
\end{remark}

We will now consider several examples of natural scaling limits that can be obtained in various geometric situations in tiling models. We expect these limits to be universal scaling
limits. They define natural limiting point processes.

\subsection{Continuous scaling limits} \label{sec:7.1}

\subsubsection{The Airy point process and the Airy process}\label{subsec:Airypoint}

The typical edge scaling limit at a boundary between a liquid and a solid phase is the \emph{extended Airy point process}. The limiting kernel is the \emph{extended Airy kernel},
 \begin{equation}\label{extAirykernel}
K_{\text{extAi}}(\rho_1,\xi_1;\rho_2,\xi_2)=-\phi_{\rho_1,\rho_2}(\xi_1,\xi_2)+\tilde{K}_{\text{extAi}}(\rho_1,\xi_1;\rho_2,\xi_2),
\end{equation}
where
 \begin{equation*}
\phi_{\rho_1,\rho_2}(\xi_1,\xi_2)=\frac{\mathbb{I}_{\rho_1<\rho_2}}{\sqrt{4\pi (\rho_2-\rho_1)}}\exp\bigg(-\frac{(\xi_1-\xi_2)^2}{4(\rho_2-\rho_1)}-\frac 12(\rho_2-\rho_1)
(\xi_1+\xi_2)+\frac 1{12}(\rho_2-\rho_1)^3\bigg),
\end{equation*}
and
\begin{equation}\label{KtildeextAiry}
\tilde{K}_{\text{extAi}}(\rho_1,\xi_1;\rho_2,\xi_2)=\int_0^\infty e^{-\lambda(\rho_1-\rho_2)}\Ai(\xi_1+\lambda)\Ai(\xi_2+\lambda)\,d\lambda.
\end{equation}
This kernel is also given by a double contour integral formula. Let $\mathcal{C}_1$ be the contour consisting of lines from $\infty e^{-\pi i/3}$ to $0$ and from
$0$ to $\infty e^{\pi i/3}$, and let $\mathcal{C}_2$ be the contour consisting of lines from $\infty e^{-2\pi i/3}$ to $0$ and from
$0$ to $\infty e^{2\pi i/3}$. Then,
\begin{equation*}
\tilde{K}_{\text{extAi}}(\rho_1,\xi_1;\rho_2,\xi_2)=\frac {e^{\frac 13(\rho_2^3-\rho_1^3)-\xi_2\rho_2+\xi_1\rho_1}}{(2\pi \mathrm{i})^2}\int_{\mathcal{C}_1}dz\int_{\mathcal{C}_2}\frac {dw}{z-w}e^{\frac 13(z^3-w^3)-\rho_1z^2+\rho_2 w^2-(\xi-\rho_1^2)z^2+(\xi_2-\rho_2^2)w}.
\end{equation*}
This can be seen by inserting the integral formula for the Airy function in (\ref{KtildeextAiry}). A special case of the extended Airy kernel is the \emph{Airy kernel},
\begin{equation}\label{Airykernel}
K_{\Ai}(\xi_1,\xi_2)=K_{\text{extAi}}(\rho_1,\xi_1;\rho_1,\xi_2)=\int_0^\infty \Ai(\xi_1+\lambda)\Ai(\xi_2+\lambda)\,d\lambda.
\end{equation}
Let $\tau_1<\dots<\tau_m$. The kernel $K_{\text{extAi}}$ is the correlation kernel for a determinantal point process on $\{\tau_1,\dots,\tau_m\}\times\mathbb{R}$.
The point process on $\{\tau_i\}\times\mathbb{R}$ has a last particle with the distribution
\begin{equation}\label{TWdistribution}
F_2(s)=\det(I-K_{\Ai})_{L^2(s,\infty)},
\end{equation}
the \emph{Tracy-Widom distribution}.

The extended Airy kernel appears, for example, in an appropriate scaling limit at typical points at the liquid-solid boundary in the Aztec diamond, i.e. at the boundary between the disordered and the ordered regions. From lemma 3.1 in \cite{JoAz} we have the following result.

\begin{thma}\label{thm:Azteckernelasymptotics}
Consider the correlation kernel for the Aztec diamond particle process as given by (\ref{OneAzteckernel}) or (\ref{OneAzteckernel2}). Take $c=1+1/\sqrt{2}$,
$d=1/\sqrt{2}$, $\alpha=2^{-1/6}$, $\beta=-2^{-5/6}$, $\gamma=2/3$ and $\delta=1/3$ in (\ref{scaling}). Then,
\begin{equation}\label{extAirylimit}
\lim_{n\to\infty}(\sqrt{2}-1)^{u_1-u_2+r_2-r_1}e^{\xi_1\rho_1-\xi_2\rho_2-\rho_1^3/3+\rho_2^3/3 }K^{\text{OneAz}}_n(r_1,u_1;r_2,u_2)=
K_{\text{extAi}}(\rho_1,\xi_1;\rho_2,\xi_2).
\end{equation}
\end{thma}

This is proved using a saddle point argument in the contour integral formula for $K^{\text{OneAz}}_n$. The exponents $\gamma=2/3$ and $\delta=1/3$ are called the
\emph{KPZ-scaling exponents}, since they also appear in random growth models and a basic model for this type of local random growth is the
Kardar-Parisi-Zhang stochastic partial differential equation, see remark \ref{rem:Schur} for references.

We can use theorem \ref{thm:Azteckernelasymptotics} to get a limit of the random boundary curve itself, see figure \ref{fig:Aztecboundaryprocess}.

\begin{figure}
\begin{center}
\includegraphics[height=2in]{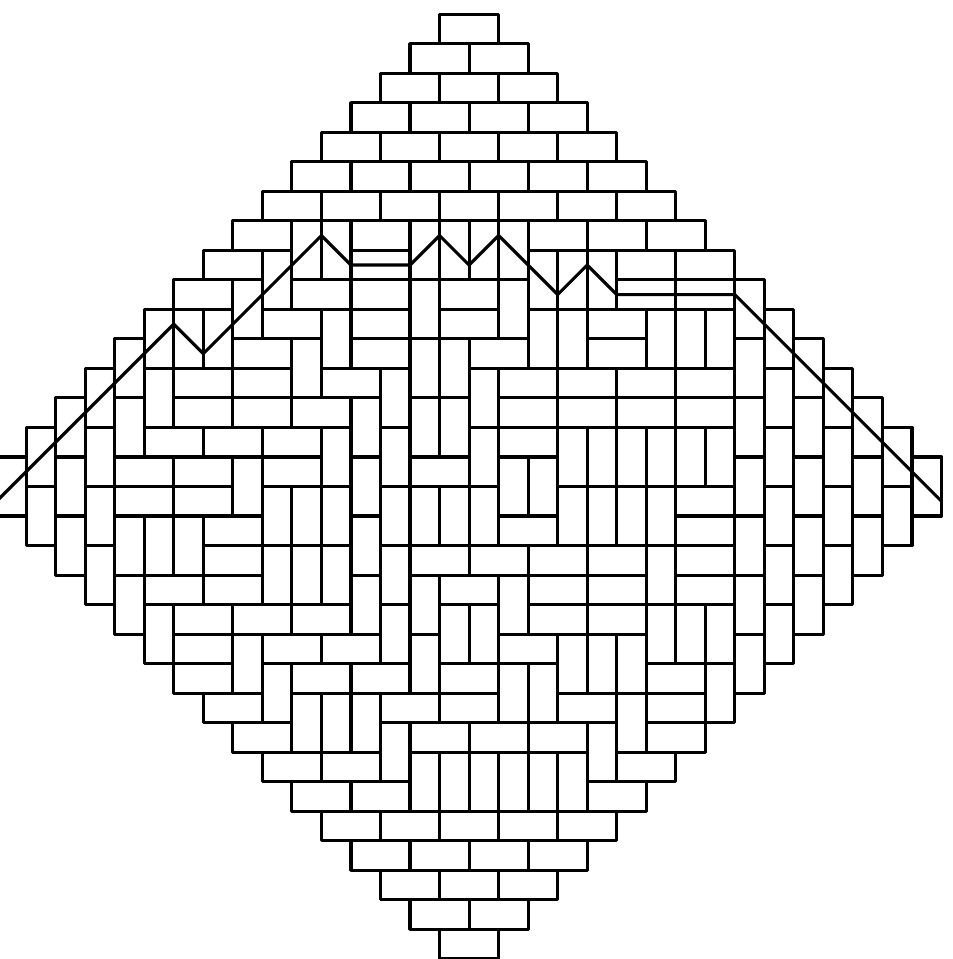}
\caption{Boundary process in the Aztec diamond}
\label{fig:Aztecboundaryprocess}
\end{center}
\end{figure}

There is a stationary stochastic process $\tau\mapsto\mathcal{A}(\tau)$, the \emph{Airy process}, such that for any $m\ge 1$ and all $s_j\in\mathbb{R}$
\begin{align*}
&\mathbb{P}\big[\mathcal{A}(\tau_1)\le s_1,\dots,\mathcal{A}(\tau_m)\le s_m\big]\\
&=\mathbb{P}\big[\text{the particles on $\{\tau_j\}\times\mathbb{R}$
are below $s_j$, $1\le j\le m$}\big]\\
&=\det(I-L)_{L^2(\{\tau_1,\dots,\tau_m\}\times\mathbb{R})},
\end{align*}
where
\begin{equation*}
L(\tau_i,\xi_i;\tau_j,\xi_j)=\mathbb{I}_{(s_i,\infty)}(\xi_i)K_{\text{extAi}}(\tau_i,\xi_i;\tau_j,\xi_j)\mathbb{I}_{(s_i,\infty)}(\xi_i).
\end{equation*}
From (\ref{TWdistribution}) we see that the one-point distribution is the Tracy-Widom distribution. The Airy process has continuous sample paths and locally looks like Brownian 
motion, see \cite{Hagg}, \cite{CoHa}, but it is not a Markov process. For uniform random tilings of the Aztec diamond we have the following theorem,\cite{JoAz}.

\begin{thma}\label{thm:Azasymptotics}
Consider an $n\times n$-Aztec diamond. Let $X_n(t)$ be the boundary process as in figure \ref{fig:Aztecboundaryprocess}, where we use the coordinate system in which the
Axtec diamond shape is given by the unit integer squares in $\{(x,y)\,;\,|x|+|y|\le n+1\}$. Then,
\begin{equation}\label{Airyprocesslimit}
\frac{X_n(2^{-1/6}n^{2/3}\tau)-n/\sqrt{2}}{2^{-5/6}n^{1/3}}\to\mathcal{A}(\tau)-\tau^2,
\end{equation}
as $n\to\infty$, in the sense of convergence of finite dimensional distributions.
\end{thma}
Thus, in an appropriate scaling limit, the boundary between the liquid and the solid regions in the Aztec diamond is described by the Airy process. This is what we expect
to be the universal edge behaviour at typical points of the liquid-solid boundary in random tiling models.

The extended Airy kernel point process and the Airy process can also be obtained as limits of Dyson's Brownian motion, see remark \ref{rem:Dyson}, so here we see a connection to random matrix theory.
The Tracy-Widom distribution is then the asymptotic largest eigenvalue distribution for large Hermitian random matrices.

\subsubsection{The Pearcey process}\label{subsec:Pearcey}

The boundary to a liquid region can have cusp points and if we zoom in at such a point we get a new limit. The Pearcey process is expected to appear when we have a cusp
whose symmetry line is not parallel to a side of the region and where we have only one type of tiles inside the cusp. There is also a special type of cusp points, see
section \ref{subsec:CuspAiry}. Let $\mathcal{C}$ be the contour consisting of the straight lines from $\infty e^{\pi i/4}$ to $0$, $0$ to $\infty e^{3\pi i/4}$,
$\infty e^{-3\pi i/4}$ to $0$ and $0$ to $\infty e^{-\pi i/4}$. Then, the \emph{Pearcey kernel} is
\begin{align}\label{Pearceykernel}
K_{\text{Pearcey}}(\rho_1,\xi_1;\rho_2,\xi_2)&=-\frac{\mathbb{I}_{\rho_1<\rho_2}}{\sqrt{2(\rho_2-\rho_1)}}e^{-\frac{(\xi_1-\xi_2)^2}{2(\rho_2-\rho_1)}}\notag\\
&+\frac 1{(2\pi \mathrm{i})^2}\int_\mathcal{C} dz\int_{-i\infty}^{i\infty}dw\, e^{\frac 14(z^4-w^4)+\frac12(\rho_2 w^2-\rho_1 z^2)+\xi_1z-\xi_2w}\frac 1{w-z}.
\end{align}
The Pearcey particle process is a determinantal point process with kernel (\ref{Pearceykernel}) on a set of lines $\{\tau_1,\dots,\tau_m\}\times\mathbb{R}$. This type of limit can 
be obtained for example in skew plane partitions that we discussed above in section \ref{sec4.4}, see figure 2 in \cite{OkRe3}. We will not present an actual limit theorem, see \cite{OkRe3}.
In this case the scaling exponents are $\gamma=1/2$ and $\delta=1/4$.

\subsubsection{The Tacnode process}\label{subsec:Tacnode}
A special type of geometric situation occurs when the boundary of the liquid region has a tacnode, i.e. we have a quadratic tangency at a point of a line from both sides,
see figure \ref{fig:tpn200}. If we have the same types of tiles in the openings as in figure  \ref{fig:tpn200}, we expect a limiting process called the \emph{tacnode process}.
The correlation kernel defining the tacnode process is defined as follows. Let $\lambda>0$ and $\sigma\in\mathbb{R}$ be two parameters. The parameter $\lambda$ is
the ratio of the curvatures of the two asymptotic boundary curves that are tangent at the point we are looking at, and $\sigma$ is a kind of "pressure" between the two sides
of the tangent line. The value $\lambda=1$ corresponds top the symmetric case. Set
\begin{align*}
\Ai^{(s)}(x)&=e^{\frac 23s^3+sx}\Ai(s^2+x),\\
K_{\Ai}^{(\alpha,\beta)}(x,y)&=\int_0^\infty\Ai^{(\alpha)}(x+t)\Ai^{(\beta)}(y+t)\,dt,\\
\mathcal{B}_{\tau,\xi}^{\lambda}(x)&=\int_0^\infty\Ai^{(\tau)}(\xi+(1+1/\sqrt{\lambda})^{1/3}t)\Ai(x+t)\,dt,\\
b_{\tau,\xi}^{\lambda}(x)&=\lambda^{1/6}\Ai^{(\lambda^{1/3}\tau)}(-\lambda^{1/6}+(1+\sqrt{\lambda})^{1/3}x).
\end{align*}
Write $\tilde{\sigma}=\lambda^{1/6}(1+\sqrt{\lambda})^{2/3}\sigma$ and define
\begin{align*}
L_{\text{tac}}^{\lambda,\sigma}(\rho_1,\xi_1;\rho_2,\xi_2)&=K_{\Ai}^{(-\tau_1,\tau_2)}(\sigma+\xi_1,\sigma+\xi_2)\\
&+(1+1/\sqrt{\lambda})^{1/3}\langle \mathcal{B}_{\rho_2,\sigma+\xi_2}^{\lambda}-b_{\rho_2,\sigma+\xi_2}^{\lambda},
(I-\chi_{\tilde{\sigma}}K_\Ai\chi_{\tilde{\sigma}})^{-1}\mathcal{B}_{-\rho_1,\sigma+\xi_1}^{\lambda}\rangle_{L^2(\tilde{\sigma},\infty)},
\end{align*}
where $\chi_{\tilde{\sigma}}(x)=\mathbb{I}_{x>\tilde{\sigma}}$ is an indicator function. The \emph{tacnode kernel} is then given by
\begin{align}\label{tacnodekernel}
K_{\text{tacnode}}^{\lambda,\sigma}(\rho_1,\xi_1;\rho_2,\xi_2)&=-\frac{\mathbb{I}_{\rho_1<\rho_2}}{\sqrt{4(\rho_2-\rho_1)}}e^{-\frac{(\xi_1-\xi_2)^2}{4(\rho_2-\rho_1)}}\\
&+L_{\text{tac}}^{\lambda,\sigma}(\rho_1,\xi_1;\rho_2,\xi_2)+
\lambda^{1/6}L_{\text{tac}}^{\lambda^{-1},\lambda^{2/3}\sigma}(\lambda^{1/3}\rho_1,-\lambda^{1/6}\xi_1;\lambda^{1/3}\rho_2,-\lambda^{1/6}\xi_2).
\end{align}
The tacnode process is again a determinantal point process on a set of lines $\{\tau_1,\dots,\tau_m\}\times\mathbb{R}$. Using the formula in theorem \ref{thm:DA} this limit can be proved in the double Aztec diamond model, \cite{AJvM}. We will not go into the details. The scaling exponents are the KPZ-scaling exponents in this case.

\begin{remark} {\rm The tacnode kernel has appeared in several different forms and it is not trivial that they agree with each other. The tacnode process was first analysed in \cite{DKZ} for non-colliding Brownian motions using a $4\times 4$ matrix Riemann-Hilbert approach, and
in \cite{AFvM} starting from a non-intersecting path ensemble. The form of the limit is completely different in the two approaches. In \cite{JoNIBM} the tacnode situation
for non-colliding Brownian motions was analyzed using an approach not based on the Riemann-Hilbert formulation, and the kernel was given essentially in the form given above.
The form given above comes from \cite{FeVe} who generalized \cite{JoNIBM} to the non-symmetric situation. The double Aztec diamond was analyzed in \cite{AJvM} and it follows 
from that paper that the form of the tacnode kernel given in \cite{AFvM} agrees with the one above. The agreement with the Riemann-Hilbert form of the kernel was shown in \cite{Del}.}
\end{remark}

\begin{remark} {\rm As discussed in section \ref{sec:4} many of the random tiling models can be thought of as ensembles of non-intersecting paths. If we keep the picture of random paths we should in the limit get an infinite ensemble of non-colliding paths called a \emph{line ensemble}. For the so called \emph{Airy line ensemble}, see \cite{CoHa}.
This is the line ensemble corresponding to the extended Airy kernel point process and it has a top path which is an Airy process.}
\end{remark}

\subsection{Discrete/continuous scaling limits} \label{sec:7.2}

In this section we consider various possible scaling limits in random tiling models of the form (\ref{scaling2}). For rhombus tilings the particles are the rhombi, whereas for domino
tilings we consider the particles as defined in section \ref{sec4.2.1}. In all the cases of interest for the discrete/continuous scaling limits there is an interlacing particle
structure which also survives in the limit. We will not discuss the actual scaling limits or the necessary asymptotic analysis which in some cases is rather involved.

\begin{figure}
\begin{center}
\includegraphics[height=2in]{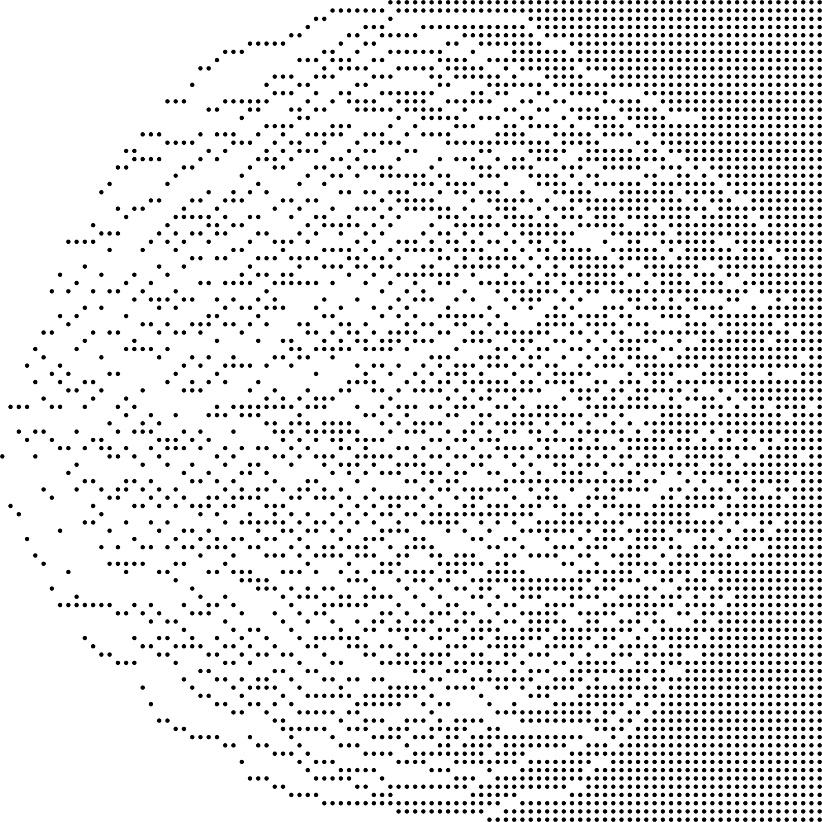}
\caption{Interlacing particles at a tangency point in the Aztec diamond. The GUE-minor process is the limit process if we consider the interlacing particles on succesive vertical lines in the left side of the picture. We get the same result for the dual particle system in the right side of the picture.}
\label{fig:GUEminor}
\end{center}
\end{figure}

\subsubsection{The GUE-minor or -corner process}\label{subsec:GUEcorner}

When the liquid region is tangent to the boundary of the tiling region and we have different type of frozen tiles on the two sides, we expect to get a scaling limit called the \emph{GUE-minor process} or \emph{GUE-corner process} in a suitable discrete/continuous scaling limit. Consider the particle picture of the Aztec diamond and remove one
type of particles, e.g. keep just the blue particles, see figure \ref{fig:GUEminor}. In the vicinity of the tangent point to the boundary we then get an interlacing particle process which is determinantal. In a scaling limit like
(\ref{scalinglimit2}) with $\delta=1/2$ in (\ref{scaling2}), we get an interlacing particle process on $\mathbb{Z}^+\times\mathbb{R}$. We keep the discrete
structure in the direction orthogonal to the boundary and the rescaling is done in the direction parallel to the boundary. Let $r_1,r_2\in\mathbb{Z}^+$ and
$\xi_1,\xi_2\in\mathbb{R}$.
The correlation kernel for the limiting determinantal point process is the \emph{GUE-minor kernel},
\begin{align}\label{GUEminorkernel}
K_{\text{GUEminor}}(r_1,\xi_1;r_2,\xi_2)&=-\mathbb{I}_{r_1>r_2}\mathbb{I}_{\xi_1\ge\xi_2}2^{r_1-r_2}\frac{(\xi_1-\xi_2)^{r_1-r_2-1}}{(r_1-r_2-1)!}\notag\\
&+\frac{2}{(2\pi \mathrm{i})^2}\int_{\gamma_r}dz\int_{L}\frac {dw}{w-z}\frac{w^{r_2}}{z^{r_1}}e^{w^2-z^2+2\xi_1z-2\xi_2w}.
\end{align}
Here $L$ is a contour $t\mapsto s+it$, $t\in \mathbb{R}$, where $s>r$.
It is also possible to express the kernel in terms of Hermite polynomials, see \cite{JoNo}.

The GUE-minor process also appears in random matrix theory. Consider an $N\times N$ GUE matrix, i.e. an Hermitian matrix $H=(H_{ij})$ taken randomly according
to the probability measure
\begin{equation}\label{GUE}
\frac 1{Z_N}e^{-\Tr H^2}dH,
\end{equation}
where $dH=\prod_{1\le i<j\le N} d(\Re H_{ij})d(\Im H_{ij})\prod_{i=1}^NdH_{ii}$. If we consider the eigenvalues of the principal submatrices $(H_{ij})_{1\le o,j\le n}$, $1\le n\le N$, they will interlace and form a determinantal point process on $\{1,\dots, N\}\times\mathbb{R}$ with correlation kernel given by (\ref{GUEminorkernel}). This is another case
where we see a close relation to random matrix theory. Note that this means in particular that the single particle on the first line has Gaussian fluctuations.

\subsubsection{The Cusp-Airy process}\label{subsec:CuspAiry}

The boundary curve of the liquid region can have a cusp with a symmetry line parallel to the sides of the region and such that inside the cusp there are two types
of tiles, one type on each side of the symmetry line, see figure 1 in \cite{KeOk2}.

In this case when we zoom in at the cusp and take a scaling limit like (\ref{scalinglimit2}) with
$\delta=1/3$ we get a limiting correlation kernel called the \emph{Cusp-Airy kernel} defined as follows. Let $\mathcal{L}_L$ be a contour consisting of lines from
$\infty e^{-2\pi i/3}$ to $y$ and $y$ to $\infty e^{2\pi i/3}$, and $\mathcal{L}_R$ be a contour consisting of lines from
$\infty e^{-\pi i/3}$ to $x$ and $x$ to $\infty e^{\pi i/3}$. Also, let $0<\rho_1<\rho_2$, $\rho_2<x$ and $\rho_2<-y$.
For
$r_1,r_2\in\mathbb{Z}$ and $\xi_1,\xi_2\in\mathbb{R}$,
\begin{align}\label{CuspAirykernel}
K_{\text{CuspAiry}}(r_1,\xi_1;r_2,\xi_2)&=-\mathbb{I}_{r_1<r_2}\mathbb{I}_{\xi_1\le\xi_2}\frac{(\xi_2-\xi_1)^{r_2-r_1-1}}{(r_2-r_1-1)!}\notag\\
&+\frac{1}{(2\pi \mathrm{i})^2}\int_{\mathcal{L}_L+\gamma_{\rho_2}}dz\int_{\mathcal{L}_R+\gamma_{\rho_1}}\frac {dw}{w-z}\frac{w^{r_1}}{z^{r_2}}e^{\frac 13(w^3-z^3)-\xi_1w+\xi_2z}.
\end{align}
This kernel can be established as a limit of a model of interlacing particles, or rhombus tilings, as discussed in section \ref{sec:5} as was done in \cite{DJMe}.
This type of situation was first studied in \cite{OkRe2} under the name cuspidal turning point in skew plane partitions but the kernel given there has the wrong contours and no proof was given.

\subsubsection{The Discrete Tacnode process}\label{subsec:DiscreteTacnode}

\begin{figure}
\begin{center}
\includegraphics[height=2in]{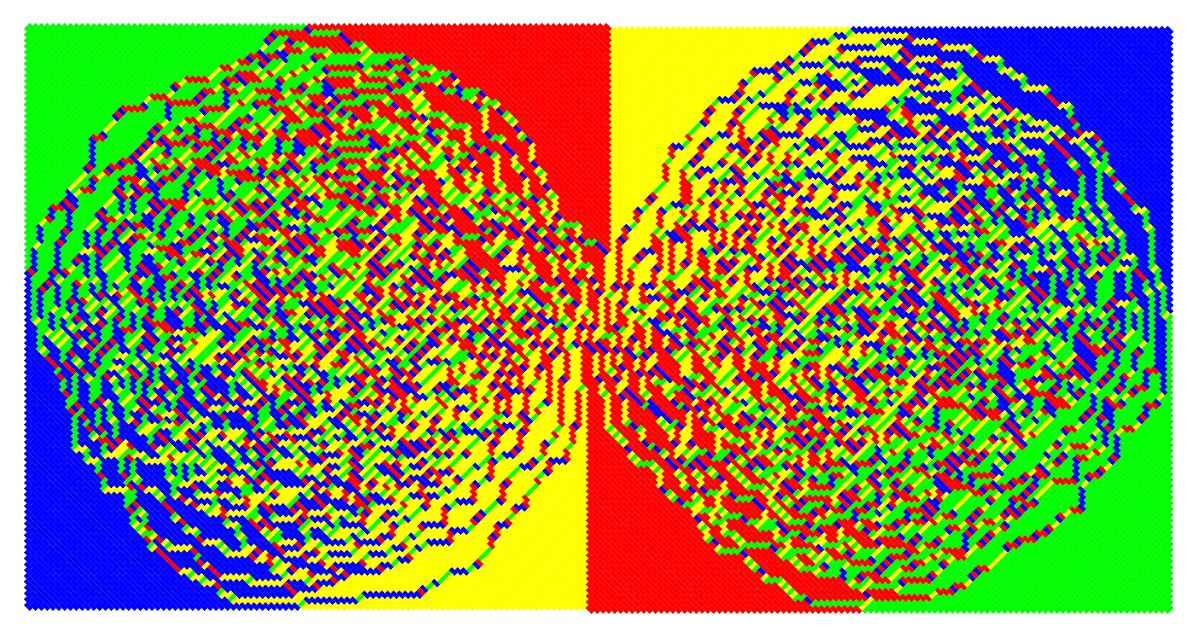}
\caption{A random tiling of the Double Aztec diamond in a situation giving a Discrete Tacnode.}
\label{fig:DiscTacTilings}
\end{center}
\end{figure}

\begin{figure}
\begin{center}
\includegraphics[height=2in]{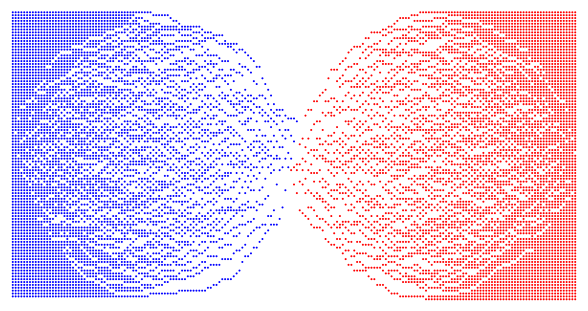}
\caption{The particle picture corresponding to figure \ref{fig:DiscTacTilings}.}
\label{fig:DiscTacParticles}
\end{center}
\end{figure}

In some situations the boundary curve of the liquid region can have a tacnode with a symmetry line parallel to a side of the tiling region and with different types of
tiles on the two sides of the symmetry line, see figure \ref{fig:DiscTacTilings} for an example in the Double Aztec diamond. Consider the Double Aztec diamond and look at in terms of the particle picture.
If we keep just one type of particles for the dominos, see figure \ref{fig:DiscTacParticles}, we get a two-sided interlacing particle system with a certain relation between the size of the overlap and the number of particles in the overlap. In the Double Aztec diamond we have an overlap of size $r$ and $r$ particles in the overlap region. This means that the number of particles
on each line in the two-sided interlacing pattern varies as $\dots,r+2,r+1,r,\dots, r,r+1,r+2,\dots$ with the number of $r$'s equal to $r+1$, see figure \ref{fig:twosidedinterlacing}.
We can think of this as two overlapping GUE-minor type processes.

In this case a scaling limit of the
form (\ref{scalinglimit2}) with $\delta=1/2$ can be computed using theorem \ref{thm:DA} and changing to the right type of particles by using the inverse Kasteleyn matrix for
the Double Aztec diamond, see \cite{ACJvM}.  This gives a limit that we can call the \emph{Discrete Tacnode kernel}. We will not give present the kernel here but refer to
\cite{ACJvM}. The formula involves a resolvent just as the Tacnode kernel above. In the paper the kernel is called the Tacnode GUE-minor kernel since it is an extension
of the GUE-minor kernel. For a different perspective on the kernel see \cite{AvM}. This type of two-sided interlacing pattern can also be obtained in rhombus tilings,
see \cite{AJvM2}. In work in progress, \cite{AJvM3}, this leads to a generalization of the discrete Tacnode kernel obtained in \cite{ACJvM} to an arbitrary relation between the
size of the overlap and the number of particles on each line in the overlap region so this gives a more general Discrete Tacnode kernel.

\begin{figure}
\begin{center}
\includegraphics[height=2in]{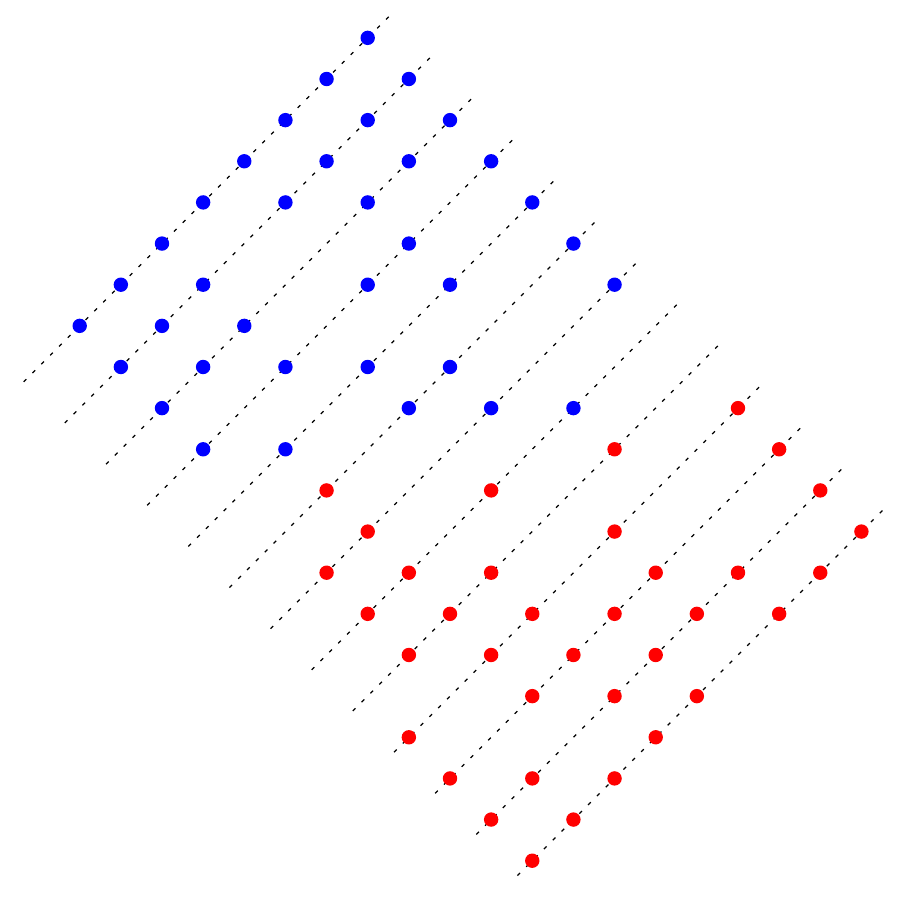}
\caption{The two-sided interlacing pattern.}
\label{fig:twosidedinterlacing}
\end{center}
\end{figure}

\subsection{Hierarchy of kernels} \label{sec:7.3}
The kernels discussed in sections \ref{sec:7.1} and \ref{sec:7.2} are not unrelated. Some of then can be obtained by taking appropriate scaling limits of the other kernels.
From the models it is reasonable to expect that the Pearcey kernel is a limit of the Tacnode kernel and that the Airy kernel can be obtained as a limit of the Pearcey kernel,
and hence also of the tacnode kernel. Also, for example, we should be able to get the Airy kernel as a limit of the GUE-minor kernel. In fact, we expect a hierarchy
of kernels as shown in figure \ref{fig:Hierarchy}.

\begin{figure}
\begin{center}
\includegraphics[height=2in]{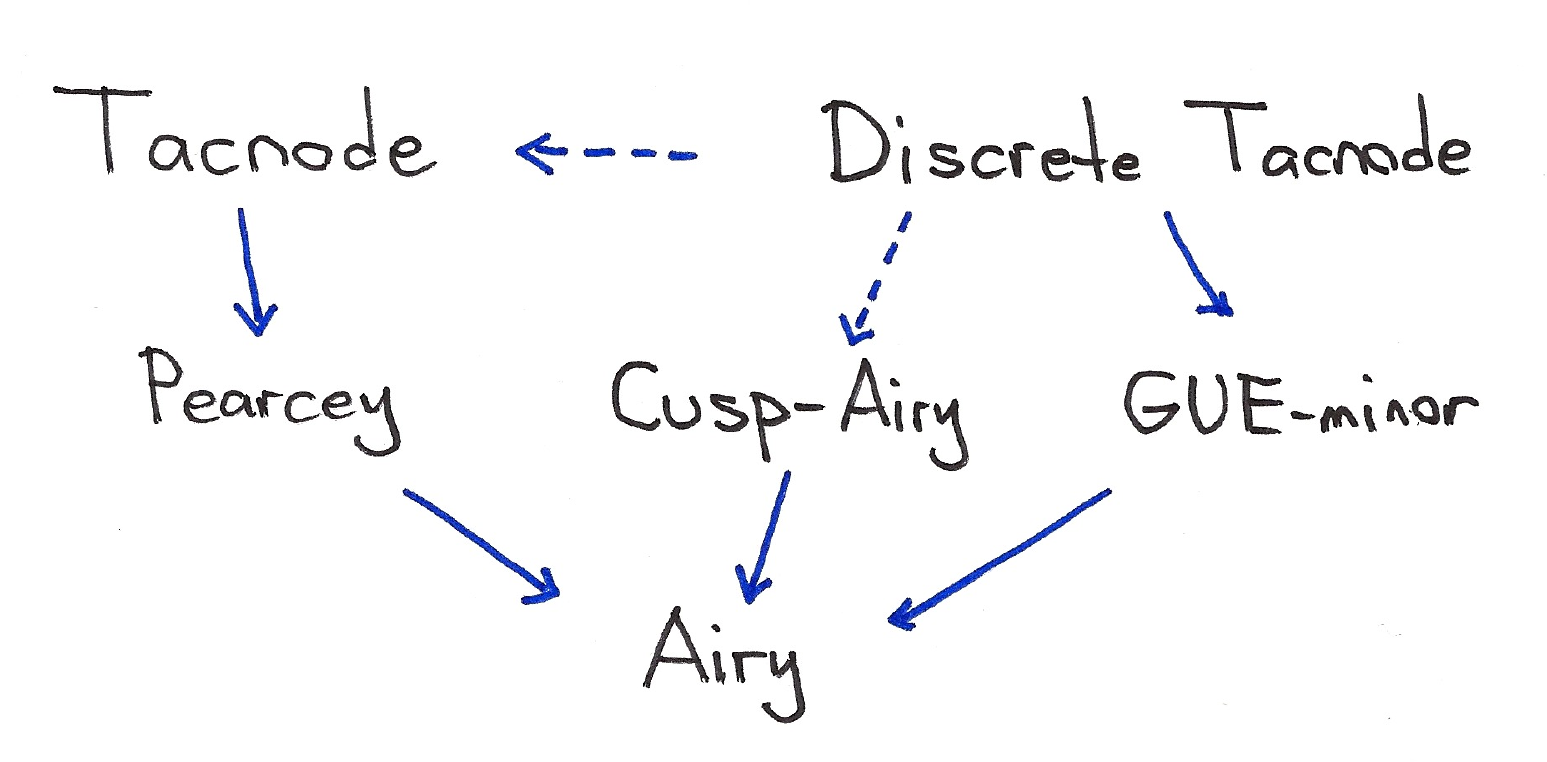}
\caption{Hierarchy of limiting correlation kernels.}
\label{fig:Hierarchy}
\end{center}
\end{figure}

The full arrows represent cases that have been proved, \cite{ACvM}, \cite{BC} or that would be relatively straightforward to prove, whereas the dotted arrows are not known. At the top of the hierarchy one might expect a sufficiently general form of the Discrete Tacnode kernel.
 It might be in some sense the most general kernel that fits into this hierarchy and occurs naturally in appropriate classes of random tiling models. There can also, depending on which class of tilings one considers be other types of limits also, e.g. the one studied in \cite{FeVe2}.
 
 This relates to a universality problem, to classify all possible scaling limits in suitable classes of models, e.g. rhombus tilings of an appropriate family of simply connected polygonal regions depending on the geometric
situation. There are certainly also cases where we expect non-universal limits like for the singular boundary cases discussed in \cite{DuMeII}. Also, outlier particles can lead to modifications of the kernels, see \cite{AFvMWan}. It is not completely clear which kernels should be considered as the basic universal scaling limits in some natural sense.

\section{The two-periodic Aztec diamond} \label{sec:8}

\subsection{Definition of the model and phases} \label{sec:8.1}

Consider the Aztec diamond shape, or more specifically the Aztec diamond graph, see figure \ref{fig:TPweights} . We introduce dimer weights in this graph in the following way. For a two-colouring of
the faces, the edge weights around the faces with a particular colour alternate between $a$ and $1$, where $0<a<1$ is a parameter. We call
this a \emph{two-periodic weighting}, and the corresponding dimer or random tiling model is referred to as the \emph{two-periodic Aztec diamond}. A simulation can be seen in figure
\ref{fig:twoperiodicsim}. For the two-periodic Aztec diamond we also have a height representation as defined in section \ref{sec:3.3}. Simulations of the random height function can be seen in figures \ref{fig:TPAztecheight} and \ref{fig:TPAztecheight2}.

\begin{figure}
\begin{center}
\includegraphics[height=2in]{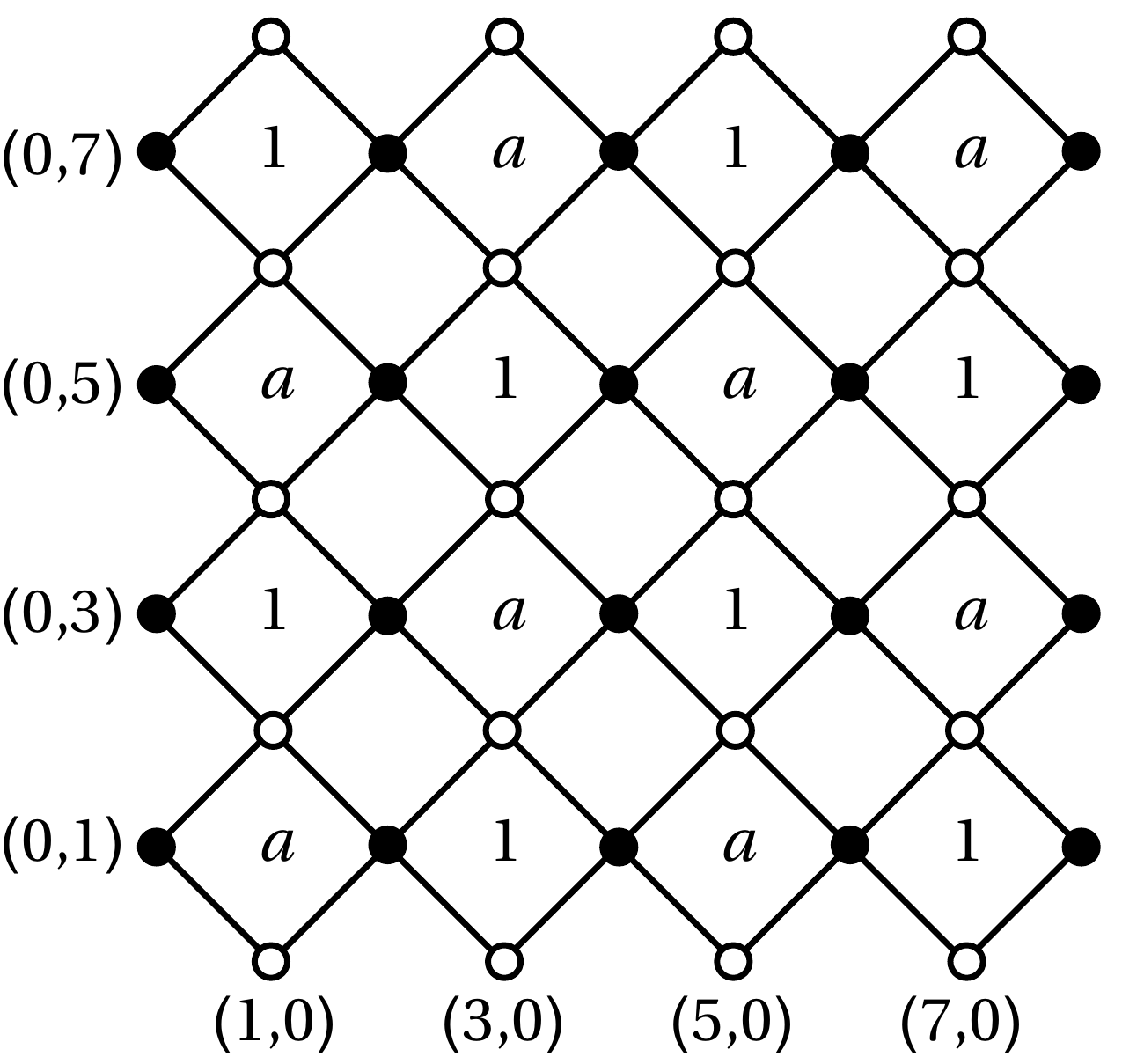}
\caption{The graph, weights and coordinate system in the two-periodic Aztec diamond. All edges surrounding a face has the weight indicated in the face.}
\label{fig:TPweights}
\end{center}
\end{figure}

\begin{figure}
\begin{center}
\includegraphics[height=2.5in]{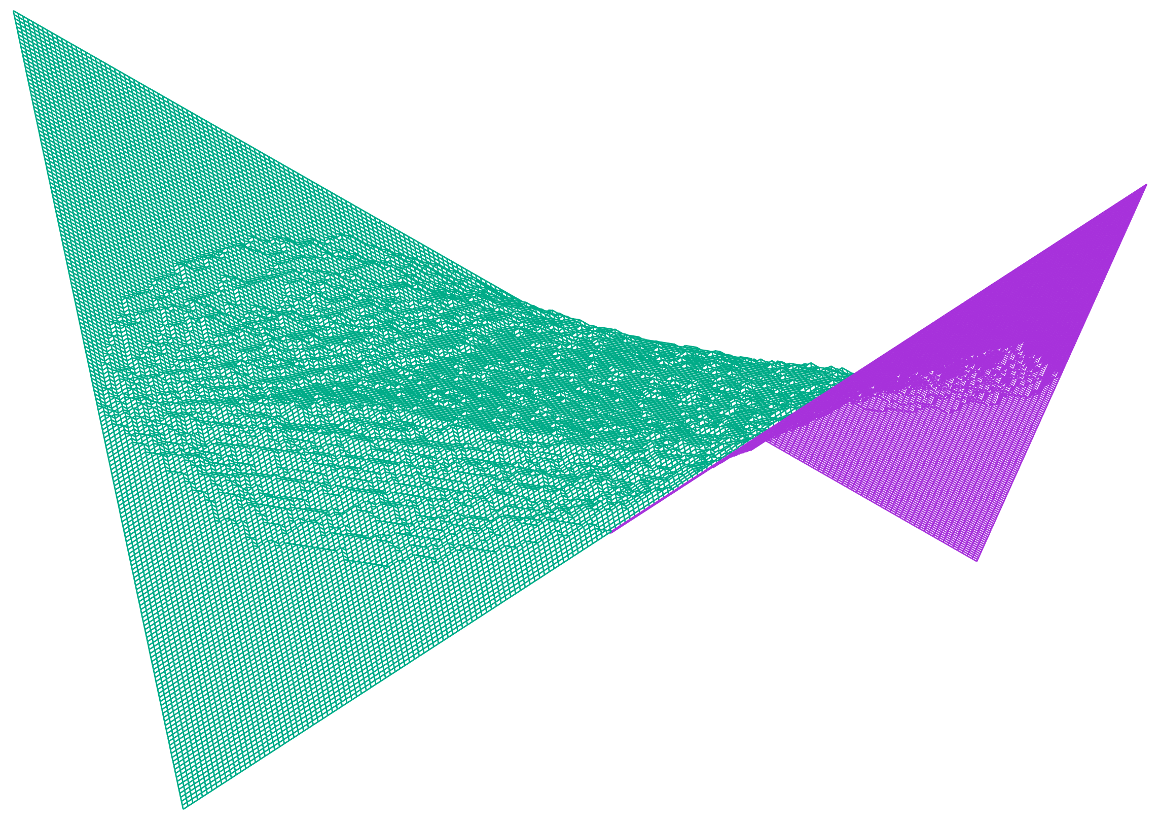}
\caption{A simulation of the random height function in the two-periodic Aztec diamond. (Picture by V. Beffara)}
\label{fig:TPAztecheight}
\end{center}
\end{figure}

\begin{figure}
\begin{center}
\includegraphics[height=2in]{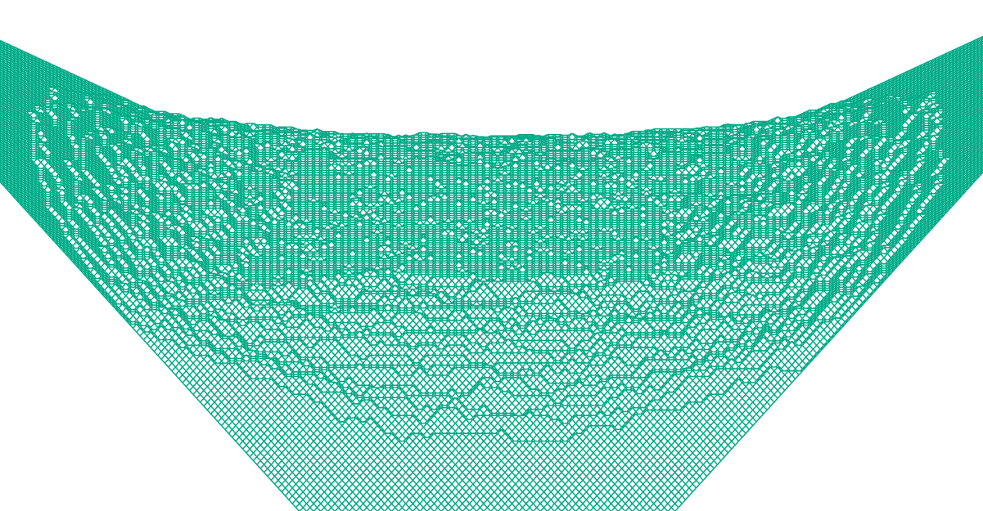}
\caption{A simulation of the random height function in the two-periodic Aztec diamond. (Picture by V. Beffara)}
\label{fig:TPAztecheight2}
\end{center}
\end{figure}

To be more precise, we choose the Kasteleyn matrix as follows. The black and white vertices in the Aztec diamond graph, figure \ref{fig:TPweights}, are
\begin{align*}
\mathbf{W}&=\{(x_1,x_2)\in(2\mathbb{Z}+1)\times (2\mathbb{Z})\,;\,1\le x_1\le 2n-1, 0\le x_2\le 2n\},\\
\mathbf{B}&=\{(x_1,x_2)\in(2\mathbb{Z})\times (2\mathbb{Z}+1)\,;\,0\le x_1\le 2n, 1\le x_2\le 2n-1\}.
\end{align*}
The black and white vertices can be divided into two types. For $i\in\{0,1\}$, we write
\begin{align*}
\mathbf{B}_i&=\{(x_1,x_2)\in\mathbf{B}\,;\,x_1+x_2\,\text{mod} 4=2i+1\},\\
\mathbf{W}_i&=\{(x_1,x_2)\in\mathbf{W}\,;\,x_1+x_2\,\text{mod} 4=2i+1\}.
\end{align*}
Let
\begin{equation}\label{e1e2}
e_1=(1,1),\quad e_2=(-1,1).
\end{equation}
The Kasteleyn matrix for the two-periodic Aztec diamond of size $n=4m$ is defined by
\begin{equation} \label{TPKasteleyn}
      \mathbb{K}^{\text{TP}}_{a,1}(x,y)=\left\{ \begin{array}{ll}
                     a (1-j) +  j  & \mbox{if } y=x+e_1, x \in \mathtt{B}_j \\
                     (a j + (1-j) ) \mathrm{i} & \mbox{if } y=x+e_2, x \in \mathtt{B}_j\\
                     a j + (1-j)  & \mbox{if } y=x-e_1, x \in \mathtt{B}_j \\
                     (a (1-j) + j ) \mathrm{i} & \mbox{if } y=x-e_2, x \in \mathtt{B}_j\\
			0 & \mbox{if $(x,y)$ is not an edge}
                     \end{array} \right.,
\end{equation}
where $\mathrm{i}^2=-1$.

If we look at figure \ref{fig:twoperiodicsim} we see that there seems to be three types of local structures in the random tiling. Also, if we look at the random height function,
see figures \ref{fig:TPAztecheight} and \ref{fig:TPAztecheight2}, we see three types of surface structure locally. In fact, there is a classification of all the possible phases, translation invariant Gibbs states, on the corresponding infinite graphs in this type of dimer models. We will not go into the details, see \cite{KOS}. We will call the three types of phases \emph{liquid}, \emph{solid} and \emph{gas}. They can be characterized by the decay of correlations between dimers. Consider two dimers at distance $r$. In a liquid phase the correlation between the dimers decay like a power law, typically like $c/r^2$. In a gas phase the correlation decays exponentially in $r$. Finally, in a solid phase the dimers are perfectly correlated. The Gibbs states can be specified
by an infinite inverse Kasteleyn matrix, $\mathbb{K}^{-1}_{\text{gas}}$, $\mathbb{K}^{-1}_{\text{liquid}}$ and $\mathbb{K}^{-1}_{\text{solid}}$, which are given by certain
double contour integrals (or as Fourier coefficients), see (\ref{Kinversegas}). The infinite plane dimer models have a height function with a certain average slope which
characterizes it, \cite{KOS}. We expect the finite inverse Kasteleyn matrix $\mathbb{K}_n^{-1}$, when rescaled around a particular point, to converge to one of these three types
of infinite inverse Kasteleyn matrices. Which limit we get should, conjecturally, be given by the slope of the asymptotic height function at that point. This can be verified
in special cases, e.g. \cite{Gor}, but has not been proved more generally.

If we think of the height function, see figures \ref{fig:TPAztecheight} and \ref{fig:TPAztecheight2}, as some kind of crystal, we see that we can think of the solid parts as representing perfect crystal facets, the liquid parts as a
rough crystal surface and the gas phase as representing a crystal facet with local random dislocations.

We see that in the two-periodic Aztec diamond we have two types of boundaries between regions with different types of phases. There is a liquid-solid boundary and a
liquid-gas boundary. Looking at the pictures, we see that, just as for the standard Aztec diamond, the liquid-solid boundary is well-defined microscopically. It is located 
where we see the first dimer/domino breaking the regular pattern of the solid region. We can expect the fluctuations of this boundary to be described, in an appropriate
scaling limit, by the Airy process. If we look at the liquid-gas boundary however, we see that it is not  clear how to define the boundary microscopically, at least if we look
only locally. Hence, the question whether this boundary is also in some sense described by the Airy process is much less clear. The two-periodic Aztec diamond
is a model where we can try to investigate this problem.

\subsection{The inverse Kasteleyn matrix}\label{sec:8.2}

Somewhat remarkably, it is possible to find a useful formula for the inverse of the Kasteleyn matrix (\ref{TPKasteleyn}). Before we can present this formula we need some notation.
The size of the Aztec diamond is always divisible by 4,$n=4m$, below. Let
\begin{equation*}
c=\frac{a}{1+a^2}
\end{equation*}
and
\begin{equation*}
\tilde{c}(u,v)=2(1+a^2)+a(u+u^{-1})(v+v^{-1}).
\end{equation*}
The last expression is related to the so called characteristic or spectral polynomial $P(z,w)$ of the model, \cite{KOS}, via
\begin{equation*}
P(z,w)=-\tilde{c}(\sqrt{\frac zw}, \sqrt{zw}).
\end{equation*}
For $\eps_1,\eps_2\in\{0,1\}$, we write
\begin{equation*}
h(\eps_1,\eps_2)=\eps_1(1-\eps_2)+\eps_2(1-\eps_1).
\end{equation*}
For $x=(x_1,x_2)\in\mathbf{W}_{\eps_1}$ and $y=(y_1,y_2)\in\mathbf{B}_{\eps_2}$, we set
\begin{equation}\label{Kinversegas}
\mathbb{K}_{r_1,r_2}(x,y)=-\frac{i^{1+h(\eps_1,\eps_2)}}{(2\pi \mathrm{i})^2}\int_{\gamma_{R_1}}\frac {du_1}{u_1}\int_{\gamma_{R_2}}\frac {du_2}{u_2}
\frac{a^{\eps_2}u_2^{1-h(\eps_1,\eps_2)}+a^{1-\eps_1}u_1u_2^{h(\eps_1,\eps_2)}}
{\tilde{c}(u_1,u_2)u_1^{(x_1-y_1+1)/2}u_2^{(x_2-y_2+1)/2}},
\end{equation}
where $R_1=\sqrt{r_1/r_2}$, $R_2=\sqrt{r_1r_2}$. If $(r_1,r_2)=(1,1)$, the formula in (\ref{Kinversegas}) gives the full plane inverse Kasteleyn matrix for a gas phase. If $P(z,w)=0$
for some $(z,w)\in\gamma_{r_1}\times\gamma_{r_2}$, then we have a liquid phase. For the relation between this formula and the formula in \cite{KOS} see \cite{CJ}.

Write
\begin{equation*}
G(w)=\frac 1{\sqrt{2c}}(\omega-\sqrt{\omega^2+2c}),
\end{equation*}
where the square root is defined by $\sqrt{\omega^2+2c}=e^{\frac 12\log(\omega+\mathrm{i}\sqrt{2c})+\frac 12\log(\omega-\mathrm{i}\sqrt{2c})}$, $\omega\in\mathbb{C}\setminus
\mathrm{i}[-\sqrt{2c},\sqrt{2c}]$, with the argument in the logarithms taken in $(-\pi/2,3\pi/2)$. Set
\begin{equation*}
{ H_{x_1,x_2}(\omega)}=\frac{\omega^{n/2}  G \left(\omega \right)^{\frac{n-x_1}{2}}}{G \left(\omega^{-1} \right)^{\frac{n-x_2}{2}}},
\end{equation*}
and define, for $\eps_1,\eps_2\in\{0,1\}$, $x=(x_1,x_2)\in\mathbf{W}_{\eps_1}$, $y=(y_1,y_2)\in\mathbf{B}_{\eps_2}$, and $\sqrt{2c}<p<1$,
\begin{equation}\label{Bepsilon}
\mathcal{B}_{\eps_1,\eps_2}(a;x,y)=\frac{i^{\frac{x_2-x_1+y_1-y_2}{2}}}{(2\pi \mathrm{i})^2}\int_{\gamma_{p}}\frac{d\omega_1}{\omega_1}
\int_{\gamma_{1/p}}d\omega_2\frac{\omega_2}{\omega_2^2-\omega_1^2}\frac{H_{x_1+1,x_2}(\omega_1)}{H_{y_1,y_2+1}(\omega_2)}
Z_{\eps_1,\eps_2}(a;\omega_1,\omega_2).
\end{equation}
Here, $Z_{\eps_1,\eps_2}$ is a complicated function that depends on $x,y$ only through $\eps_1,\eps_2$ and is independent of $n$. See \cite{CJ}
for its precise definition. We can now give the formula for the inverse Kasteleyn matrix for the two-periodic Aztec diamond.

\begin{thma}\label{thm:KinverseTPAz}
Consider the two-periodic Aztec diamond with Kasteleyn matrix (\ref{TPKasteleyn}) and $n=4m$, $m\ge 1$. Let $x=(x_1,x_2)\in\mathbf{W}_{\eps_1}$ and $y=(y_1,y_2)\in\mathbf{B}_{\eps_2}$, $\eps_1,\eps_2\in\{0,1\}$. Then,
\begin{align}\label{KinverseTPAz}
&\big(\mathbb{K}^{\text{TP}}_{a,1}\big)^{-1}(x,y)=\mathbb{K}_{1,1}^{-1}(x,y)-\bigg[\mathcal{B}_{\eps_1,\eps_2}(a;x_1,x_2,y_1,y_2)-\frac {\mathrm{i}}a(-1)^{\eps_1+\eps_2}
\mathcal{B}_{1-\eps_1,\eps_2}(1/a;2n-x_1,x_2,2n-y_1,y_2)\notag\\
&-\frac {\mathrm{i}}a(-1)^{\eps_1+\epsilon_2}\mathcal{B}_{\eps_1,1-\eps_2}(1/a;x_1,2n-x_2,y_1,2n-y_2)
+\mathcal{B}_{1-\eps_1,1-\eps_2}(a;2n-x_1,2n-x_2,2n-y_1,2n-y_2)\bigg].
\end{align}
\end{thma}
We will not discuss the proof of the theorem which is complicated and not very natural. The theorem is proved in \cite{CJ} based on a formula for the generating
function for the inverse Kasteleyn matrix derived in \cite{CY}. 

Using saddle point arguments, it is possible to derive an asymptotic formula for the inverse Kasteleyn
matrix in the vicinity of the part of the lower left to upper right diagonal that lies in the third quadrant in the Aztec diamond graph. The analysis could be done at other
points also but would be more complicated and was not carried out in \cite{CJ}. Here, we only consider points close to the liquid-gas boundary, see figure \ref{fig:zoompoint}.
Set
\begin{equation*}
c_0=\frac{(1-2c)^{2/3}}{(c(1+2c))^{1/3}},\quad \lambda_1=\frac{\sqrt{1-2c}}{c_0},\quad \lambda_2=\frac{(1-2c)^{3/2}}{cc_0^2}
\end{equation*}
(slight change in notation compared to \cite{CJ} and \cite{BCJ}), and
\begin{equation}
\mathtt{g}_{\eps_1,\eps_2} =
\left\{
\begin{array}{ll}
\frac{\mathrm{i} \left(\sqrt{a^2+1}+a\right)}{1-a}
 &  \mbox{if } (\eps_1,\eps_2)=(0,0) \\
\frac{\sqrt{a^2+1}+a-1}{\sqrt{2a} (1-a) }
 & \mbox{if } (\eps_1,\eps_2)=(0,1) \\
-\frac{\sqrt{a^2+1}+a-1}{\sqrt{2a} (1-a) }
&\mbox{if } (\eps_1,\eps_2)=(1,0)\\
\frac{\mathrm{i}\left(\sqrt{a^2+1}-1\right)}{(1-a) a}
&\mbox{if } (\eps_1,\eps_2)=(1,1).
\end{array}
\right.
\end{equation}

\begin{figure}
\begin{center}
\includegraphics[height=2in]{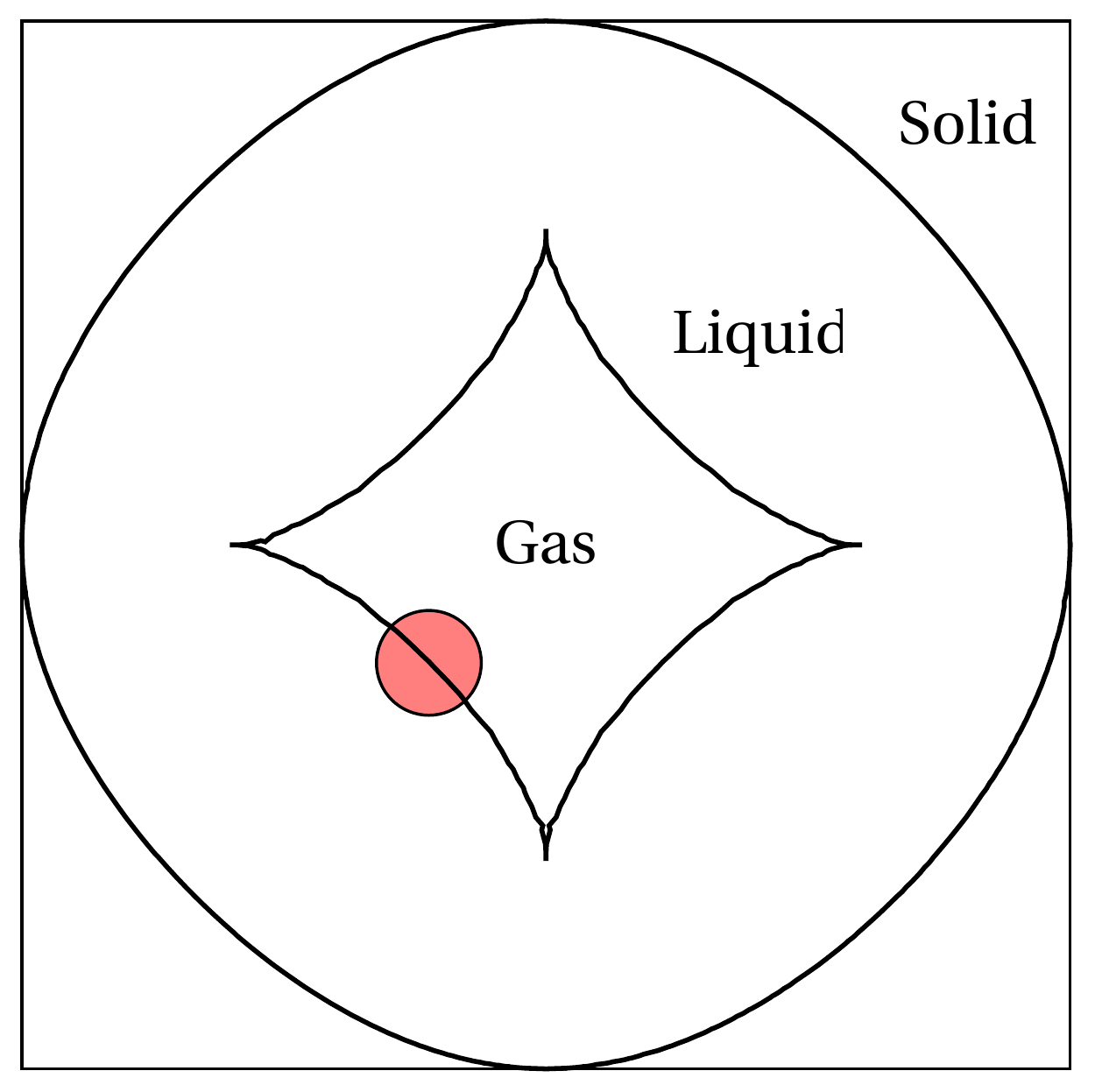}
\caption{The point at the liquid-gas boundary where we study the asymptotics.}
\label{fig:zoompoint}
\end{center}
\end{figure}

Zooming in around the red point in figure \ref{fig:zoompoint} corresponds the right scaling is
\begin{align}\label{TPscaling}
x&=\big[n(1-\frac 12\sqrt{1-2c})+\alpha_x\lambda_1n^{1/3}\big]e_1-\big[\beta_x\lambda_2n^{2/3}\big]e_2+u_1,\notag\\
y&=\big[n(1-\frac 12\sqrt{1-2c})+\alpha_y\lambda_1n^{1/3}\big]e_1-\big[\beta_y\lambda_2n^{2/3}\big]e_2+u_2,
\end{align}
where $\alpha_x,\alpha_y,\beta_x,\beta_y\in\mathbb{R}$ are the new rescaled variables (we have $\beta_x\to-\beta_x$ compared to \cite{CJ})
and $u_1,u_2\in\mathbb{Z}$.

\begin{thma}\label{thm:TPasymptotics}
Let $x=(x_1,x_2)\in\mathbf{W}_{\eps_1}$ and $y=(y_1,y_2)\in\mathbf{B}_{\eps_2}$, $\eps_1,\eps_2\in\{0,1\}$, be given by (\ref{TPscaling}).
Then,
\begin{align}\label{TPasymptotics}
\big(\mathbb{K}^{\text{TP}}_{a,1}\big)^{-1}(x,y)&=\mathbb{K}_{1,1}^{-1}(x,y)-\mathrm{i}^{y_1-x_1+1}|G(\mathrm{i})|^{\frac 12(-2-x_1+x_2+y_1-y_2)}
e^{\beta_x\alpha_x-\beta_y\alpha_y+\frac 23(\beta_x^3-\beta_y^3)}
\notag\\
&\times c_0\mathtt{g}_{\eps_1,\eps_2}\tilde{K}_{\text{extAi}}(\beta_x,\alpha_x+\beta_x^2;\beta_y,\alpha_y+\beta_y^2)n^{-1/3}(1+O(n^{-1/3}))
\end{align}
as $n\to\infty$ uniformly for $\alpha_x,\alpha_y,\beta_x,\beta_y, u_1,u_2$ in a compact set, where $\tilde{K}_{\text{extAi}}$ is given by (\ref{KtildeextAiry}).
\end{thma}

For the proof see \cite{CJ}.

We see from (\ref{TPasymptotics}) that the main term in the asymptotics is just the infinite full plane inverse Kasteleyn matrix for a gas phase. The term involving
the extended Airy kernel appears only at a lower order. This can be vaguely interpreted as follows. What we see at the liquid-gas boundary is essentially a
gas phase, but there are also some long range connected structures, the connected curves in figure \ref{fig:twoperiodicsim}, that have a distance of order $n^{1/3}$
from each other. Thus, we cannot directly see an Airy process at the liquid-gas boundary. This is related to the difficulty described above of defining a precise microscopic
liquid-gas boundary and separating its statistical behaviour from the gas background.

\begin{remark} {\rm There is an analogue of theorem \ref{thm:TPasymptotics} at the liquid-solid boundary. This theorem could be used to prove a theorem similar to theorem
\ref{thm:Azasymptotics}.}
\end{remark}

Looking at figures \ref{fig:TPAztecheight} and \ref{fig:TPAztecheight2}. we can see that the long-range structures at the liquid-gas boundary that should give the Airy process, or more generally the Airy line ensemble, are
related to long range correlated changes in the height function. On top of these height changes we have short-range fluctuations in the height function which can be
thought of as coming from the "gas phase". We can attempt to catch the long-range effects by somehow averaging out the short-range fluctuations. This is the idea behind
the next theorem, which shows that in some sense we do have convergence to the Airy process at the liquid-gas boundary also. We will not be precise with all definitions. For a complete treatment see \cite{BCJ}.

\subsection{Airy point process limit theorem}\label{sec:8.3}

Let $A_1,\dots, A_L$ be disjoint closed and bounded intervals in $\mathbb{R}$. We can define discrete analogues $I_1,\dots, I_L$ of $A_1,\dots, A_L$ in the Aztec diamond centered
at the liquid-gas boundary. Essentially, $I_j$ consists of points centered on edges in the Aztec diamond graph,
\begin{equation*}
\big[n(1-\frac 12\sqrt{1-2c})\big]e_1+ke_1+(\frac 12,\frac 12)
\end{equation*}
such that $k/dn^{1/3}\in A_j$, $1\le j\le L$, and $d$ is an appropriate constant. Let $M=\big[(\log n)^4\big]$ and take $M$ copies $I_{j,k}$, $1\le k\le M$, of $I_j$ in the direction 
$e_2$ separated by a distance of order $\sqrt{M}$, see figure \ref{fig:linesmeasure}. Let
\begin{equation*}
\Delta h(I_{j,k})=\text{the height change along}\,\, I_{j,k}.
\end{equation*}
This can be expressed as a signed sum over $a$-edges, i.e. dimers covering edges in $I_{j,k}$ with weight $a$. The height change captures the height change due to long-range
structures cutting through the interval and short range height changes close to the endpoints of the interval. Since correlations in the gas phase decay exponentially
with distance, we can hope that by averaging $\Delta h(I_{j,k})$ over $k$ only the height change due to the long-range structures visible in the figures will survive. Motivated
by this we define a random measure by

\begin{figure}
\begin{center}
\includegraphics[height=2in]{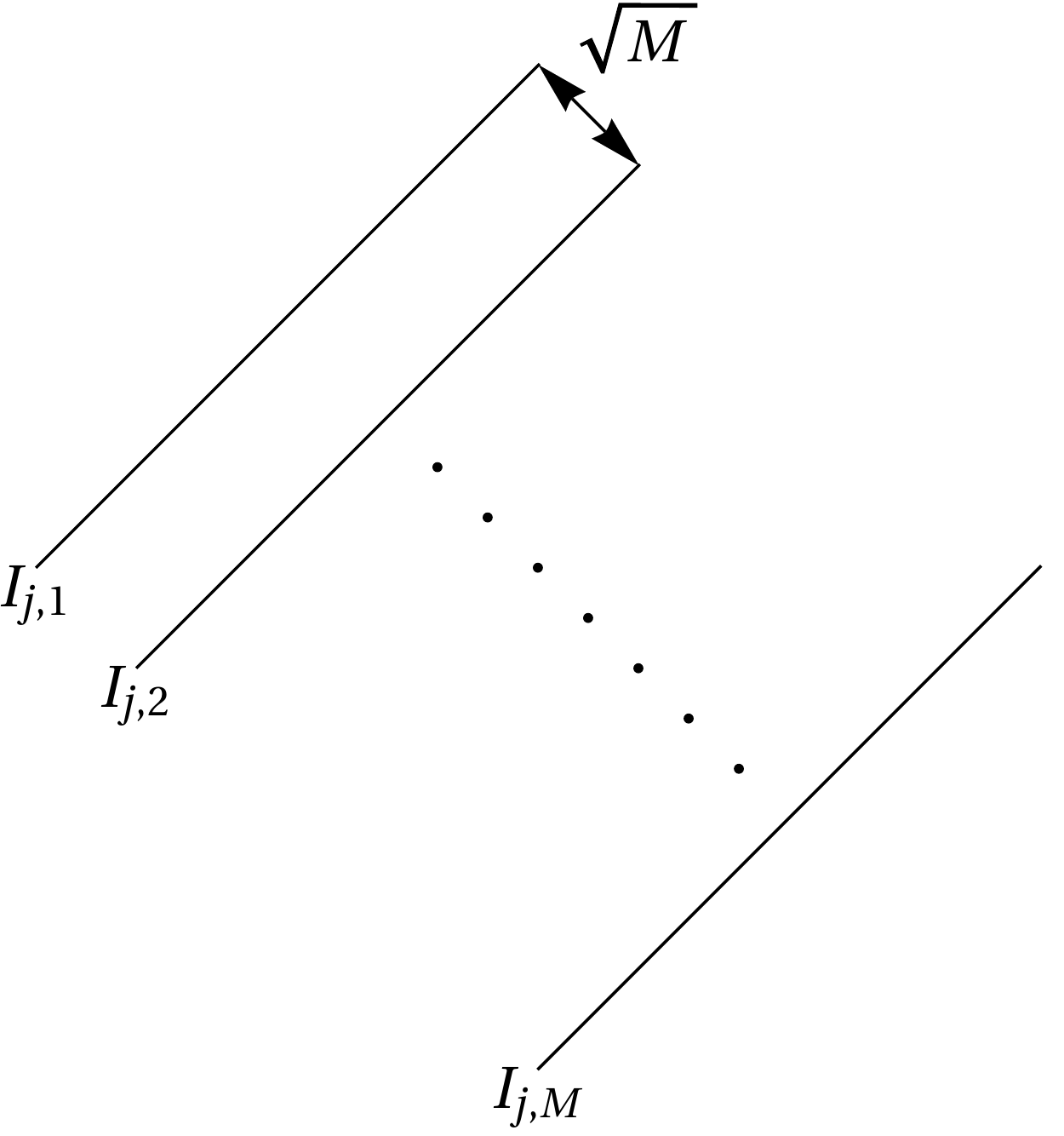}
\caption{Copies of the discrete intervals.}
\label{fig:linesmeasure}
\end{center}
\end{figure}

\begin{equation}\label{randommeasure}
\mu_n(A_j)=\frac 1M\sum_{k=1}^M\Delta h(I_{j,k}),
\end{equation}
$1\le j\le L$. The Airy kernel defines a determinantal point process on $\mathbb{R}$ which we can think of as a random measure denoted by $\mu_{\Ai}$. We then have the
following theorem.

\begin{thma}\label{thm:TPlimit}
We have that
\begin{equation*}
(\mu_n(A_1),\dots,\mu_n(A_L))\to (\mu_{\Ai}(A_1),\dots,\mu_{\Ai}(A_L))
\end{equation*}
in distribution as $n\to\infty$. Thus $\mu_n$ converges weakly to $\mu_{\Ai}$.
\end{thma}
 
The theorem can be generalized to the extended Airy kernel point process, see \cite{BCJ}. The proof of the theorem \ref{thm:TPlimit} uses Laplace transforms and the cumulant method for determinantal point processes. The proof is rather subtle since we have two length scales, the local length scale of $\mathbb{K}_{1,1}^{-1}$ which is of order 1 and the length scale of
order $n^{1/3}$ of the Airy kernel, and the result depends on subtle cancellations.


\section{Appendix} \label{sec:appendix}
We briefly summarize some facts about partitions and Schur polynomials that we use in the paper and refer to \cite{Mac} and \cite{Sta} for more details. 

A \emph{partition} $\lambda=(\lambda_1,\lambda_2,\lambda_3,\dots)$ of $n$ is a sequence of non-negative integers such that $|\lambda|:=\lambda_1+\lambda_2+\dots=n$, and $\lambda_1\ge\lambda_2\ge\dots$.
The largest $\ell$ such that $\lambda_\ell>0$ is called the \emph{length} of the partition and is denoted by $\ell(\lambda)$. The partition of zero is the empty partition
and is denoted by $\emptyset$. We can think of a partition as a \emph{Young diagram}, where we have $\lambda_i$ left-adjusted boxes in row $i$, $i\ge 1$. If we consider
the length of the columns of the Young diagram from left to right we get a new partition $\lambda'$, the \emph{conjugate partition}. A partition can also be written as 
$\lambda=\langle 1^{m_1}2^{m_2}3^{m_3}\dots\rangle$, where $m_i$ is equal to the number of $i$'s in the partition. For two partitions $\lambda,\mu$, we write 
$\mu\subseteq\lambda$ if $\mu_i\le\lambda_i$ for all $i\ge 1$. We call $\lambda/\mu$ a \emph{skew partition} if $\mu\subseteq\lambda$, and we can think of it as those
boxes in the Young diagram of $\lambda$ that are not in the Young diagram for $\mu$, we get a \emph{skew Young diagram}. The skew partition $\lambda/\mu$ is a \emph{horizontal
strip} if $\lambda/\mu$ has at most one square in each column, i.e. $0\le\lambda_i'-\mu_i'\le 1$, $i\ge 1$.

Let $x_1,\dots, x_n$ be variables. The $k$:th \emph{elementary symmetric polynomial} $e_k(x_1,\dots,x_n)$ can be defined via its generating function,
\begin{equation}\label{elsymmpol}
\sum_{k\in\mathbb{Z}}e_k(x_1,\dots,x_n)z^k=\prod_{i=1}^n(1+x_iz).
\end{equation}
Note that $e_k(x_1,\dots,x_n)\equiv 0$ if $k<0$ or $k>n$. The 
$k$:th \emph{complete symmetric polynomial} $h_k(x_1,\dots,x_n)$ is defined by
\begin{equation}\label{comsymmpol}
\sum_{k\in\mathbb{Z}}h_k(x_1,\dots,x_n)z^k=\prod_{i=1}^n\frac 1{1-x_iz},
\end{equation}
and we see that $h_k=0$ if $k<0$.
For a skew partition $\lambda/\mu$ the \emph{skew Schur polynomial} $s_{\lambda/\mu}(x_1,\dots,x_n)$  is given by the Jacobi-Trudi identity,
\begin{equation}\label{JacobiTrudi1}
s_{\lambda/\mu}(x_1,\dots,x_n)=\det\big(h_{\lambda_i-\mu_j-i+j}(x_1,\dots,x_n)\big)_{1\le i,j\le M},
\end{equation}
for $M\ge\ell(\lambda)$.
We also have the Jacobi-Trudi identity
\begin{equation}\label{JacobiTrudi2}
s_{\lambda'/\mu'}(x_1,\dots,x_n)=\det\big(e_{\lambda_i-\mu_j-i+j}(x_1,\dots,x_n)\big)_{1\le i,j\le M},
\end{equation}
for $M\ge \lambda_1$.
If $\mu=\emptyset$, then $s_\lambda(x_1,\dots, x_n)=s_{\lambda/\emptyset}(x_1,\dots, x_n)$ is the \emph{Schur polynomial} and it is also given by a quotient of
determinants
\begin{equation}\label{ClassicalSchur}
s_\lambda(x_1,\dots, x_n)=\frac{\det\big(x_i^{\lambda_j+n-j}\big)_{1\le i,j\le n}}{\det\big(x_i^{n-j}\big)_{1\le i,j\le n}},
\end{equation}
which is called the classical definition of the Schur polynomial.

We will write $\lambda\succ\mu$ if the two partitions $\lambda$ and $\mu$ are \emph{interlacing}, i.e.
\begin{equation}\label{Partitioninterlacing}
\lambda_1\ge\mu_1\ge\lambda_2\ge\mu_2\ge\lambda_3\ge\dots.
\end{equation}
This means that $\lambda/\mu$ is a horizontal strip. It follows from (\ref{JacobiTrudi2}) that for one variable $a$, we have the identity
\begin{equation}\label{Schurinterlacing}
s_{\lambda/\mu}(a)=a^{|\lambda|-|\mu|}\mathbb{I}_{\lambda\succ\mu}.
\end{equation}

A very basic and useful identity for determinantal point processes is the following formula, the Cauchy-Binet or Andrieff identity. Let $(X,\mu)$ be some measurable space and
$f_i:X\mapsto\mathbb{C}$, $g_i:X\mapsto\mathbb{C}$, $1\le i\le n$, some measurable functions. Then,
\begin{equation}\label{CBAnd}
\frac 1{n!}\int_{X^n}\det\big(f_i(x_j)\big)_{1\le i,j\le n}\det\big(g_i(x_j)\big)_{1\le i,j\le n}d^n\mu(x)=
\det\bigg(\int_X f_i(x)g_j(x)\,d\mu(x)\bigg)_{1\le i,j\le n}
\end{equation}
provided all expressions are well-defined. This formula can also be used to prove many properties of Schur polynomials.



\end{document}